\documentclass[final,leqno]{siamltex}

\usepackage[pdftex]{graphics}
\usepackage{mathrsfs,bbm,stmaryrd}
\usepackage{amsmath,amsfonts,amssymb}
\usepackage{fancyheadings,subfigure}
\usepackage[small,normal,bf,up]{caption}
\usepackage{epsfig,graphicx,wrapfig,fancyvrb,listings,color,textcomp,float}
\usepackage{algorithmic,url,hyperref}
\usepackage[section]{algorithm}
\usepackage{marvosym}
\newcommand{\scalar}[1]{\left\langle\, #1 \,\right\rangle} 
\newcommand{\norm}[1]{\left\Vert #1 \right\Vert} 

\renewcommand{\Re}{{\mathbbm R}}
\newcommand{\EE}{ {\{ \textsc{e}, \textsc{e}\} } } 
\newcommand{\EI}{ {\{ \textsc{e}, \textsc{i}\} } } 
\newcommand{\IE}{ {\{ \textsc{i}, \textsc{e}\} } } 
\newcommand{\II}{ {\{ \textsc{i}, \textsc{i}\} } } 
\newcommand{\E}{ {\{ \textsc{e}\} } } 
\newcommand{\I}{ {\{ \textsc{i}\} } } 

\def\tableofcontents{\section*{Contents.\@mkboth{CONTENTS}{CONTENTS}\hskip 1em}  
 \@starttoc{toc}} 

\newtheorem{remark}{Remark}
\newtheorem{example}{Example}

\newcommand{\one}{1\hspace{-0,9ex}1} 
\newcommand{\e}{e\hspace{-0,9ex}e} 
\begin{document}

\title{A class of generalized \\ additive Runge-Kutta methods\thanks{
The work of A. Sandu has been supported in part by NSF through awards NSF
OCI--8670904397, NSF CCF--0916493, NSF DMS--0915047, NSF CMMI--1130667, 
NSF CCF--1218454, AFOSR FA9550--12--1--0293--DEF, AFOSR 12-2640-06,
and by the Computational Science Laboratory at Virginia Tech. The work of M. G\"unther has been supported in part by BMBF through grant 
03MS648E.}
}

\author{Adrian Sandu\thanks{Virginia Polytechnic Institute and State University, Computational Science Laboratory, Department of 
Computer Science, 2202 Kraft Drive, Blacksburg, VA 24060, USA ({\tt sandu@cs.vt.edu})}
\and
Michael G\"unther\thanks{Bergische Universit\"at Wuppertal,
        Institute of Mathematical Modelling, Analysis and Compuational
        Mathematics (IMACM), Gau\ss strasse 20, D-42119 Wuppertal, Germany 
        ({\tt guenther@uni-wuppertal.de}).}}

% Numerische Mathematik header
%\title{A class of generalized additive Runge-Kutta methods}
%\author{Adrian Sandu$^\dagger$ \and Michael G\"unther$^\ddagger$}
%
%\institute{$^\dagger$Computational Science Laboratory \at
%              Department of Computer Science \\
%              Virginia Tech \\
%              2202 Kraft Drive, Blacksburg, VA 24060, USA \\
%              Tel.: +1-540-231-2193 \\
%              Fax: +1-540-231-9218 \\
%              \email{sandu@cs.vt.edu}           %  \\
%%             \emph{Present address:} of F. Author  %  if needed
%           $^\ddagger$Bergische Universit\"at Wuppertal \at
%           Wick\"uler Park, Bendahler Strasse 29 \\
%           D-42285 Wuppertal, Germany \\
%              Tel.: +49 (0) 202 439 4768 \\
%              %Fax: \\
%              \email{guenther@math.uni-wuppertal.de}             
%}
%\date{\today}

\thispagestyle{empty}
\setcounter{page}{0}

\begin{Huge}
\begin{center}
% Computer Science Technical Report CSTR-{\tt insert number here} \\
Computational Science Laboratory Technical Report CSL-TR-{\tt 5/2013} \\
\today
\end{center}
\end{Huge}
\vfil
\begin{huge}
\begin{center}
Adrian Sandu and Michael G\"{u}nther
\end{center}
\end{huge}

\vfil
\begin{huge}
\begin{it}
\begin{center}
``{\tt A class of generalized \\ additive Runge-Kutta methods}''
\end{center}
\end{it}
\end{huge}
\vfil

\begin{large}
\begin{center}
Computational Science Laboratory \\
Computer Science Department \\
Virginia Polytechnic Institute and State University \\
Blacksburg, VA 24060 \\
Phone: (540)-231-2193 \\
Fax: (540)-231-6075 \\ 
Email: \url{sandu@cs.vt.edu} \\
Web: \url{http://csl.cs.vt.edu}
\end{center}
\end{large}

\vspace*{1cm}

\begin{tabular}{ccc}
\includegraphics[width=2.5in]{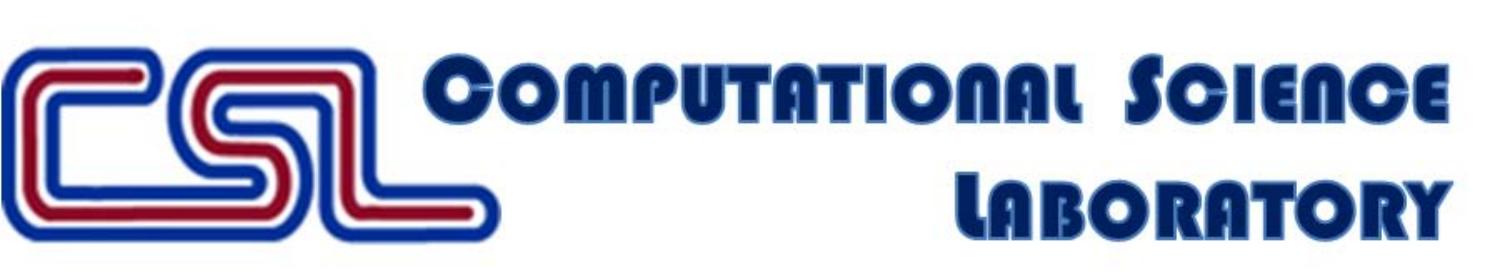}
&\hspace{2.5in}&
\includegraphics[width=2.5in]{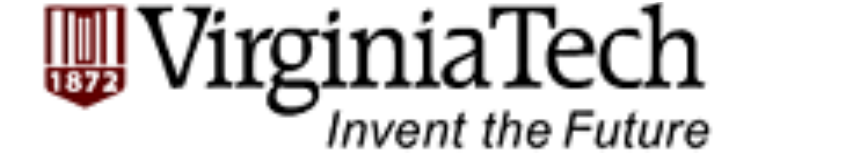} \\
{\bf\em Innovative Computational Solutions} &&\\
\end{tabular}

\newpage

\maketitle

%\tableofcontents

\begin{abstract}
This work generalizes the additively partitioned Runge-Kutta methods by allowing 
for different stage values as arguments of different components of the right hand side.
An order conditions theory is developed for the new family of generalized additive methods,
and stability and monotonicity investigations are carried out. The paper discusses the construction and properties of implicit-explicit 
and implicit-implicit,methods in the new framework. The new family, named GARK, introduces additional flexibility when compared 
to traditional partitioned Runge-Kutta methods, and therefore offers additional opportunities for the development of flexible solvers 
for systems with multiple scales, or driven by multiple physical processes.
\end{abstract}



\pagestyle{myheadings}
\thispagestyle{plain}
\markboth{ADRIAN SANDU AND MICHAEL G\"UNTHER}{GARK METHODS}

%%%%%%%%%%%%%%%%%%%%%%%%%%
\section{Introduction}
%%%%%%%%%%%%%%%%%%%%%%%%%%

In many applications, initial value problems of ordinary differential equations are given as {\it additively} partitioned systems
\begin{equation} 
\label{eqn:additive-ode}
 y'= f(y) = \sum_{m=1}^N f^{\{m\}} (y)\,, \quad  y(t_0)=y_0\,,
\end{equation}
where the right-hand side $f: \Re^d \rightarrow \Re^d$ is split into $N$ different parts 
with respect to, for example, stiffness, nonlinearity, dynamical behavior, and evaluation cost. Additive partitioning also includes the special case of {\em component} partitioning where
the solution vector is split into $N$ disjoint sets, $\{1,2,\ldots,d\} = \cup_{m=1}^N\, {\cal I}^{\{m\}}$, with the $m$-th set containing the components $y^i$ with indices $i \in {\cal I}^{\{m\}}$.
One defines a corresponding partitioning of the right hand side 
\begin{equation}
\label{eqn:component-partitioning}
f^{\{m\}}(y) := \sum_{i \in {\cal I}^{\{m\}}} \e_i\, \e_i^T\, f(y), \quad
f^{\{m\}i} (y) = \left\{
\begin{array}{ll}
f^i(y), ~ & i \in {\cal I}^{\{m\}} \\
0, ~ & i \notin {\cal I}^{\{m\}}
\end{array}
\right.,
\end{equation}
where $\e_i$ is the $i$-th column of the identity matrix, and superscripts (without parentheses) represent
vector components.
A particular case is {\em coordinate} partitioning where $d=N$ and ${\cal I}^{\{m\}}=\{m\}$, $m=1,\ldots,d$:
\begin{equation} 
\label{eqn:component-ode}
y' = \begin{bmatrix} y^1  \\ \vdots \\ y^N  \end{bmatrix}'   
= \sum_{m=1}^d  \begin{bmatrix} 0  \\ y^m      \\ 0 \end{bmatrix}'   
= \sum_{m=1}^d  \begin{bmatrix} 0  \\ f^m(y)  \\ 0 \end{bmatrix}\,.
\end{equation}

The development of Runge-Kutta (RK) methods that are tailored to the partitioned system~\eqref{eqn:additive-ode} 
started with the early work of Rice \cite{Rice_1960}. Hairer \cite{Hairer_1981_Pseries} developed the concept of P-trees and laid the foundation for
the modern order conditions theory for partitioned Runge-Kutta methods. The investigation of practical
partitioned Runge-Kutta methods with implicit and explicit components  \cite{Rentrop_1985} was revitalized by the work of Ascher, Ruuth, and Spiteri \cite{Ascher_1997_IMEX_RK}.
Additive Runge-Kutta methods have been investigated by Cooper and Sayfy \cite{Cooper_1983_ARK}
and Kennedy and Carpenter \cite{Kennedy_2003}, and
partitioning strategies have been discussed by Weiner \cite{Weiner_1993_partitioning}. 

This work focuses on specialized schemes which exploit the different dynamics in the right hand sides $f^{\{m\}}$
(e.g., stiff and non-stiff),
allow for arbitrarily high orders of accuracy, and posses good stability properties with respect to dissipative systems. 
The approach taken herein 
generalizes the additively partitioned Runge-Kutta family of methods \cite{Kennedy_2003} by allowing for different stage values as arguments of different
components of the right hand side. 

The paper is organized as follows.
Section \ref{sec:GARK} introduces the new family of generalized additively partitioned Runge-Kutta schemes.
Generalized implicit-explicit Runge-Kutta schemes are discussed in Section \ref{sec:imex}.
A stability and monotonicity analysis is performed in Section \ref{sec:stability}.
Section \ref{sec:imim} builds implicit-implicit 
generalized additively partitioned Runge-Kutta methods.
Conclusions are drawn in Section \ref{sec:conclusions}.

%%%%%%%%%%%%%%%%%%%%%%%%%%%%%%%%
\section{Generalized additively partitioned Runge-Kutta schemes}\label{sec:GARK}
%%%%%%%%%%%%%%%%%%%%%%%%%%%%%%%%
In this section we extend additive Runge-Kutta to generalized additively partitioned Runge-Kutta schemes. The order conditions for this new class are derived using the N-tree theory of Sanz-Serna \cite{SaSe1997}.

%%%%%%%%%%%%%%%%%%%%%%%%%%%%%%%%
\subsection{Traditional additive Runge-Kutta methods}
%%%%%%%%%%%%%%%%%%%%%%%%%%%%%%%%

Kennedy and Carpenter \cite{Kennedy_2003} developed additive Runge-Kutta (ARK) methods for systems with partitioned right hand sides~\eqref{eqn:additive-ode}.
One step of an ARK scheme reads
\begin{subequations}
\label{eqn:ARK}
\begin{eqnarray}
\label{eqn:ARK-stage}
Y_i &=& y_{n} + h \sum_{m=1}^N \sum_{j=1}^s a_{i,j}^{\{m\}} \, f^{\{m\}}(Y_j), \quad i=1,\ldots,s\,,  \\
\label{eqn:ARK-solution}
y_{n+1} &=& y_n + h \sum_{q=1}^N\, \sum_{i=1}^s b_{i}^{\{q\}} \, f^{\{q\}}(Y_i) \,.
\end{eqnarray}
\end{subequations}
The corresponding Butcher tableau is
\begin{equation}
\begin{array}{cccc}
\mathbf{A}^{\{1\}} & \mathbf{A}^{\{2\}} & \ldots & \mathbf{A}^{\{N\}} \\ \hline
\mathbf{b}^{\{1\}} & \mathbf{b}^{\{2\}} & \ldots &\mathbf{b}^{\{N\}}
\end{array}\,.
\end{equation}
%

%%%%%%%%%%%%%%%%%%%%%%%%%%%%%%%%
\subsection{Generalized additively partitioned Runge-Kutta schemes}
%%%%%%%%%%%%%%%%%%%%%%%%%%%%%%%%

The {\em generalized additively partitioned Runge-Kutta} (GARK) family of methods
extends the traditional approach \eqref{eqn:ARK} by allowing for different stage values with different components of the right hand side. 

\begin{definition}[GARK methods]
One step of a GARK scheme applied to solve \eqref{eqn:additive-ode} reads: 
\begin{subequations}
\label{eqn:GARK}
\begin{eqnarray}
\label{eqn:GARK-stage}
Y_i^{\{q\}} &=& y_{n} + h \sum_{m=1}^N \sum_{j=1}^{s^{\{m\}}} a_{i,j}^{\{q,m\}} \, f^{\{m\}}\left(Y_j^{\{m\}}\right), \\
\nonumber
&& \qquad i=1,\dots,s^{\{q\}}, \quad q=1,\ldots,N\,,  \\
\label{eqn:GARK-solution}
y_{n+1} &=& y_n + h \sum_{q=1}^N\, \sum_{i=1}^{s^{\{q\}}} b_{i}^{\{q\}} \, f^{\{q\}}\left(Y_i^{\{q\}}\right) \,.
\end{eqnarray}
\end{subequations}
\end{definition}

The corresponding generalized Butcher tableau is
\begin{equation}
\label{eqn:general-Butcher-tableau}
\begin{array}{cccc}
\mathbf{A}^{\{1,1\}} & \mathbf{A}^{\{1,2\}} & \ldots & \mathbf{A}^{\{1,N\}} \\
\mathbf{A}^{\{2,1\}} &\mathbf{A}^{\{2,2\}} & \ldots & \mathbf{A}^{\{2,N\}} \\
\vdots & \vdots & & \vdots \\
\mathbf{A}^{\{N,1\}} & \mathbf{A}^{\{N,2\}} & \ldots & \mathbf{A}^{\{N,N\}} \\ \hline
\mathbf{b}^{\{1\}} & \mathbf{b}^{\{2\}} & \ldots &\mathbf{b}^{\{N\}}
\end{array}
\end{equation}
In contrast to traditional additive methods~\cite{Kennedy_2003} different stage values are used with different components of the right hand side. 
The methods $(\mathbf{A}^{\{q,q\}},\mathbf{b}^{\{q\}})$ can be regarded as stand-alone integration schemes
applied to each individual component $q$. The off-diagonal matrices $\mathbf{A}^{\{q,m\}}$, $m \ne q$, can be 
viewed as a coupling mechanism among components.

For the special case of component partitioning \eqref{eqn:component-partitioning} a GARK step \eqref{eqn:GARK} reads
\begin{subequations}
\label{eqn:GARK-componentwise}
\begin{eqnarray}
Y_i^{\{q\}} & = & 
\begin{bmatrix}
y_{n}^{\{1\}} \\ %y_{n}^{\{2\}} \\ 
\vdots \\ y_{n}^{\{N\}}
\end{bmatrix}
+
h\,
\begin{bmatrix}
 \sum_{j=1}^{s^{\{1\}}} a_{i,j}^{\{q,1\}} f^{\{1\}}\left(Y_j^{\{1\}}\right) \\
%\sum_{j=1}^s a_{i,j}^{\{q,2\}} f^{\{2\}}\left(Y_j^{\{2\}}\right) \\
\vdots \\
\sum_{j=1}^{s^{\{N\}}} a_{i,j}^{\{q,N\}} f^{\{N\}}\left(Y_j^{\{N\}}\right)
\end{bmatrix}, \\[3pt]
\begin{bmatrix}
y_{n+1}^{\{1\}} \\ %y_{n+1}^{\{2\}} \\ 
\vdots \\ y_{n+1}^{\{N\}}
\end{bmatrix}
& = & 
\begin{bmatrix}
y_{n}^{\{1\}} \\ %y_{n}^{\{2\}} \\ 
\vdots \\ y_{n}^{\{N\}}
\end{bmatrix}
+ h \begin{bmatrix}
\sum_{i=1}^{s^{\{1\}}} b_{i}^{\{1\}} f^{\{1\}}\left(Y_i^{\{1\}}\right) \\
% \sum_{i=1}^s b_{i}^{\{2\}} f^{\{2\}}\left(Y_i^{\{2\}}\right) \\
\vdots \\
\sum_{i=1}^{s^{\{N\}}} b_{i}^{\{N\}} f^{\{N\}}\left(Y_i^{\{N\}}\right) 
\end{bmatrix}\,.
\end{eqnarray}
\end{subequations}
%One notes that this scheme defines a partitioned Runge-Kutta scheme in the case $q=1$, and generalizes partitioned Runge-Kutta schemes for $q>1$.

\begin{definition}[Internally consistent GARK methods]
A GARK scheme \eqref{eqn:GARK} is called internally consistent if
\begin{equation}
\label{eqn:GARK-internal-consistency}
\sum_{j=1}^{s^{\{1\}}} a_{i,j}^{\{q,1\}} %= \sum_{j=1}^{s^{\{2\}}} a_{i,j}^{\{q,2\}} 
= \dots = \sum_{j=1}^{s^{\{N\}}} a_{i,j}^{\{q,N\}} = c_i^{\{q\}}\,,
\quad i=1,\dots, s^{\{q\}}\,, \quad q=1,\dots,N.
\end{equation}
\end{definition}
In order to understand internal consistency
consider the system 
\[
 y'= f(y) = \sum_{m=1}^N \begin{bmatrix} \alpha^{\{m\}} \\ \beta^{\{m\}} \end{bmatrix} \,, \quad  y(t_0)=\begin{bmatrix} 0 \\ 0 \end{bmatrix}\,,
 \quad \sum_{m=1}^N \alpha^{\{m\}}= \sum_{m=1}^N \beta^{\{m\}}=1\,,
\]
with the exact solution $y(t) =[ t,t]^T$.
The GARK stages \eqref{eqn:GARK-stage} are
\begin{eqnarray*}
Y_i^{\{q\}} &=& \begin{bmatrix} t_{n} + h\, \sum_{m=1}^N \alpha^{\{m\}}\, \sum_{j=1}^{s^{\{m\}}} a_{i,j}^{\{q,m\}} 
\\ t_{n} + h\, \sum_{m=1}^N \beta^{\{m\}}\, \sum_{j=1}^{s^{\{m\}}} a_{i,j}^{\{q,m\}} 
\end{bmatrix} = 
\begin{bmatrix} t_{n} + h \, c_i^{\{q\}} \\  t_{n} + h \, c_i^{\{q\}} \end{bmatrix} \,.
\end{eqnarray*}
The internal consistency condition \eqref{eqn:GARK-internal-consistency} ensures that all components of the 
stage vectors are calculated at the same internal approximation times.

%%%%%%%%%%%%%%%%%%%%%%%%%%%%%%
\subsection{Order conditions}
%%%%%%%%%%%%%%%%%%%%%%%%%%%%%%
%
We derive the GARK order conditions by applying Araujo, Murua, and Sanz-Serna's N-tree theory \cite{SaSe1997} to~\eqref{eqn:GARK},  while taking into account the fact that the internal stages $Y_i^{\{m\}}$ and the stage numbers $s^{\{m\}}$ depend on the partition $f^{\{m\}}$.

N-trees \cite{SaSe1997} are a generalization of P-trees from the case of component partitioning~\eqref{eqn:component-ode} to the general case of right-hand side partitioning~\eqref{eqn:component-partitioning}. The set NT of N-trees consists of all Butcher trees with colored vertices; each vertex is assigned one of $N$ different colors corresponding to the $N$ components of the partition. Similar to regular Butcher trees each vertex is also assigned a label.  The order $\rho(u)$ is the number of nodes of $u \in \textnormal{NT}$.

The empty N-tree is denoted by $\emptyset$. The N-tree with a single vertex of color $\nu$ is denoted by $\tau^{[\nu]}$. The N-tree $u \in \textnormal{NT}$ with $\rho(u)>1$ and a root of color $\nu$ can be represented as $u=[u_1,\ldots,u_m]^{[\nu]}$, where $\{u_1,\ldots,u_m\}$ are the  non-empty subtrees (N-trees) arising from removing the root of $u$. 

 Similar to regular Butcher trees one denotes by $\sigma(u)$ the number of symmetries of $u \in \textnormal{NT}$, and by $\gamma(u)$ 
 the density defined recursively by
\begin{eqnarray*}
\gamma(\emptyset) & = & 1, \\
\gamma\left(\tau^{[\nu]}\right) & = & 1, \quad \nu=1,\ldots,N, \\
\gamma(u)  & = & \rho(u) \gamma(u_1) \cdots \gamma(u_m) \text{ for } u=[u_1,\ldots,u_m]^{[\nu]}.
\end{eqnarray*}

An elementary differential $F(u)(\cdot) : \Re^d \to \Re^d$ is associated to each N-tree $u \in \textnormal{NT}$. 
The elementary differentials are defined recursively for each component 
$i=1,\ldots,d$ as follows:
\begin{eqnarray*}
F^i(\emptyset)(y) & = & y^i, \\
F^i\left(\tau^{[\nu]}\right)(y) & = & f^{[\nu]i}(y), \quad \nu=1,\dots,N\,, \\
F^i(u)(y) & = &\sum_{i_1,\ldots,i_m=1}^{d}  \frac{\partial^m\, f^{[\nu]i}}{\partial y^{i_1} \cdots \partial y^{i_m}}(y) F^{i_1}(u_1)(y) \cdots F^{i_m}(u_m)(y) \\
&& \quad \textnormal{for} ~~u=[u_1,\ldots,u_m]^{[\nu]}\,.
\end{eqnarray*}
An NB-series is a formal power expansion
\[
\textnormal{NB}(\mathbf{c},y(t)):= \sum_{u \in \textnormal{NT}} \frac{h^{\rho(u)}}{\sigma(u)}\, \mathbf{c}(u)\, F(u)(y(t))\,,
\]
where $\mathbf{c} : \textnormal{NT} \to \Re$ is a mapping that assigns a real number  to each N-tree.
For example, the exact solution of \eqref{eqn:additive-ode} can be written as the following NB-series \cite{SaSe1997}
\begin{equation}
\label{nbseries.exactsol}
y(t+h)=\textnormal{NB}(\mathbf{c},y(t)) \quad \textnormal{with} \quad \mathbf{c}(u)=\frac{1}{\gamma(u)}.
\end{equation}

\begin{theorem}[GARK order conditions]
The order conditions for a GARK method \eqref{eqn:GARK} are obtained from the order conditions of ordinary Runge-Kutta methods.  
The usual labeling of the Runge-Kutta coefficients (subscripts $i,j,k,\ldots$) is accompanied by a corresponding labeling of the different partitions for the N-tree (superscripts $\sigma,\nu,\mu,\ldots$). 
\end{theorem}

Let $\one^{{\{\nu\}}}$ be a vector of ones of dimension $s^{{\{\nu\}}}$, and  $\mathbf{c}^{\{\sigma,\nu\}} = \mathbf{A}^{\{\sigma,\nu\}} \cdot \one^{{\{\nu\}}}$.
The specific conditions for orders one to four are as follows.
\begin{subequations} 
\label{eqn:GARK-order-conditions-1-to-4}
\begin{align}
\label{eqn:GARK-order-conditions-1}
\mathbf{b}^{\{\sigma\}}\,^T \cdot \one^{{\{\sigma\}}} = 1,  & \quad \forall\; \sigma. &  (order~ 1) \\
%\end{eqnarray}
%%
%\begin{eqnarray}
\label{eqn:GARK-order-conditions-2}
 \mathbf{b}^{\{\sigma\}}\,^T \cdot \mathbf{c}^{\{\sigma,\nu\}}  =  \frac{1}{2}, &\quad  \forall\; \sigma, \nu,  &  (order~ 2) \\
%\end{eqnarray}
%%
%\begin{eqnarray}
\label{eqn:GARK-order-conditions-3a}
\mathbf{b}^{\{\sigma\}}\,^T \cdot \left(  \mathbf{c}^{\{\sigma,\nu\}}  \mathbf{c}^{\{\sigma,\mu\}} \right)
 =  \frac{1}{3}, & \quad \forall\; \sigma,\nu, \mu, & (order~ 3) \\
\label{eqn:GARK-order-conditions-3b}
\mathbf{b}^{\{\sigma\}}\,^T \cdot \mathbf{A}^{\{\sigma,\nu\}}  \cdot \mathbf{c}^{\{\nu, \mu\}} 
 =  \frac{1}{6}, & \quad \forall\; \sigma,\nu, \mu,  & (order~ 3) \\
%\end{eqnarray}
%%
%\begin{eqnarray}
\label{eqn:GARK-order-conditions-4a}
\mathbf{b}^{\{\sigma\}}\,^T \cdot \left( \mathbf{c}^{\{\sigma,\nu\}}  \mathbf{c}^{\{\sigma,\lambda\}}  \mathbf{c}^{\{\sigma,\mu\}} \right)
 =  \frac{1}{4}, & \quad \forall\; \sigma,\nu,\lambda,\mu, & (order~ 4) \\
\label{eqn:GARK-order-conditions-4b}
\left(\mathbf{b}^{\{\sigma\}} \mathbf{c}^{\{\sigma,\mu\}} \right)^T \cdot \mathbf{A}^{\{\sigma,\nu\}}  \cdot \mathbf{c}^{\{\nu,\lambda\}}  
 =  \frac{1}{8}, & \quad \forall\; \sigma,\nu,\lambda,\mu, & (order~ 4) \\
\label{eqn:GARK-order-conditions-4c}
\mathbf{b}^{\{\sigma\}}\,^T \cdot \mathbf{A}^{\{\sigma,\nu\}}  \cdot
\left( \mathbf{c}^{\{\nu,\lambda\}}  \mathbf{c}^{\{\nu,\mu\}} \right)
 =  \frac{1}{12},  &\quad \forall\; \sigma,\nu,\lambda,\mu, &  (order~ 4) \\
\label{eqn:GARK-order-conditions-4d}
\mathbf{b}^{\{\sigma\}}\,^T \cdot  \mathbf{A}^{\{\sigma,\nu\}}  \cdot \mathbf{A}^{\{\lambda,\mu\}}  \cdot \mathbf{c}^{\{\lambda,\mu\}}
 =  \frac{1}{24}, & \quad \forall\; \sigma,\nu,\lambda,\mu. & (order~ 4)
\end{align}
\end{subequations}
Here, and throughout this paper, the matrix and vector multiplication is denoted by dot (e.g., $\mathbf{b}^T \cdot \mathbf{c}$ is a dot product), 
while two adjacent vectors denote  component-wise multiplication  (e.g., $\mathbf{b}\, \mathbf{c}$ is a vector of element-wise products).
These order conditions are given in element-wise form in Appendix \ref{sec:oder-conditions-detailed}.

\begin{proof}
We derive the NB-series of the solution of the GARK scheme \eqref{eqn:GARK} following the approach in~\cite{SaSe1997} for the NB-series of ARK schemes. 
The GARK solution, the stage vectors, and the stage function values are expanded in NB-series
\begin{eqnarray*}
y_{n+1} &=& \textnormal{NB}(\mathbf{d},y_n) \,, \\
Y_i^{\{q\}} &=& \textnormal{NB}(\mathbf{d}^{\{q\}}_i,y_n)\,, \\
h f^{\{m\}}\left(Y_i^{\{m\}}\right)  &=& \textnormal{NB}(\mathbf{g}^{\{m\}}_i,y_n)\,.
\end{eqnarray*}
From \eqref{nbseries.exactsol} we have that the method GARK is of order $p$ iff 
\begin{equation}
\label{eqn:GARK_nbseries-order}
\mathbf{d}(u) = \frac{1}{\gamma(u)}, \quad \forall\, u \in \textnormal{NT} \quad \textnormal{with}\quad  1 \leq \rho(u) \leq p\,.
\end{equation}

The coefficients $\mathbf{d}$,
$\mathbf{d}^{\{q\}}_i$ and $\mathbf{g}^{\{m\}}_i$ 
are related through the numerical equations \eqref{eqn:GARK}.
For $u \in \textnormal{NT} \backslash \{ \emptyset \}$ relation \eqref{eqn:GARK-solution} yields
\begin{subequations}
\label{eqn:nbtrees}
\begin{eqnarray}
\label{eqn:nbtrees-solution}
\mathbf{d}(u) & = & \sum_{q=1}^N \sum_{i=1}^{s^{\{q\}}} b_i^{\{q\}} \mathbf{g}_i^{\{q\}}(u), 
\end{eqnarray}
and relation \eqref{eqn:GARK-stage} gives
\begin{eqnarray}
\label{eqn:nbtrees-stage}
\mathbf{d}_i^{\{q\}}(u) & = & \sum_{m=1}^N \sum_{j=1}^{s^{\{m\}}} a_{i,j}^{\{q,m\}} \mathbf{g}_j^{\{m\}}(u), \\
\nonumber
&& \quad i=1,\ldots,s^{\{q\}}, \quad q=1,\ldots, N.
\end{eqnarray}
From the properties of a derivative of a B-series on has
\begin{eqnarray}
\label{eqn:nbtrees-stage-fun}
\mathbf{g}_i^{\{m\}}(u) & = & 
\left\{
\begin{array}{ll}
0 & \textnormal{for}~~ u = \emptyset\,, \\
\delta_{q,m} & \textnormal{for}~~ u = \tau^{\{q\}}\,, \\
 \delta_{q,m} \prod_{k=1}^\ell \mathbf{d}_i^{\{m\}}(u_k) \quad & \textnormal{for}~~ u = [u_1,\ldots,u_\ell]^{\{q\}}\,,
 \end{array}
\right. \\
\nonumber
&& i=1,\ldots,s^{\{m\}}, \quad m,q=1,\ldots, N\,.
\end{eqnarray}
\end{subequations}
Using equations \eqref{eqn:nbtrees-stage}--\eqref{eqn:nbtrees-stage-fun} recursively, we have that for $u=[u_1,\ldots,u_\ell]^{\{q\}}$ 
\[
\mathbf{g}_i^{\{m\}}(u) = \delta_{q,m} \sum_{n_1,\ldots,n_\ell} \sum_{j_1,\ldots,j_\ell} a_{i,j_1}^{\{m,n_1\}} \cdots 
a_{i,j_\ell}^{\{m,n_\ell\}} \mathbf{g}_{j_1}^{\{n_1\}}(u_1) \cdots \mathbf{g}_{j_\ell}^{\{n_\ell\}}(u_\ell)\,.
\]
Equation \eqref{eqn:nbtrees-stage-fun} becomes
\begin{eqnarray}
\label{eqn:nbtrees-stage-fun-3}
\mathbf{g}_i^{\{m\}}(u) & = & 
\left\{
\begin{array}{l}
0 \quad \textnormal{for}~~ u = \emptyset\,, \\
0 \quad \textnormal{for}~~ u = [u_1,\ldots,u_\ell]^{\{q\}}~~ \textnormal{and}~~q \ne m\,, \\
1 \quad \textnormal{for}~~ u = \tau^{\{m\}}\,, \\
 \sum_{j_1,\ldots,j_\ell=1}^{s^{\{m\}}} a_{i,j_1}^{\{m,\nu_1\}} \cdots 
a_{i,j_\ell}^{\{m,\nu_\ell\}} \mathbf{g}_{j_1}^{\{\nu_1\}}(u_1) \cdots \mathbf{g}_{j_\ell}^{\{\nu_\ell\}}(u_\ell)  \\
\qquad  \textnormal{for}~~ u = [u_1,\ldots,u_\ell]^{\{m\}}~~ \textnormal{and}~~\textnormal{color(root}(u_j)) = \nu_j\,,
 \end{array}
\right. \\
\nonumber
&& i=1,\ldots,s^{\{m\}}, \quad m=1,\ldots, N\,.
\end{eqnarray}
From \eqref{eqn:nbtrees-solution} we have
\begin{eqnarray}
\label{eqn:nbtrees-solution-expanded}
\mathbf{d}(u) & = & \sum_{i=1}^{s^{\{m\}}} b_i^{\{m\}} \mathbf{g}_i^{\{m\}}(u)\quad \textnormal{for}~~ u = [u_1,\ldots,u_\ell]^{\{m\}}\,.
%\nonumber
%&=&  \sum_{i=1}^{s^{\{q\}}} b_i^{\{q\}} 
%\sum_{n_1,\ldots,n_\ell} \sum_{j_1,\ldots,j_\ell} a_{i,j_1}^{\{q,n_1\}} \cdots 
%a_{i,j_\ell}^{\{q,n_\ell\}} \mathbf{g}_{j_1}^{\{n_1\}}(u_1) \cdots \mathbf{g}_{j_\ell}^{\{n_\ell\}}(u_\ell) \\
%\nonumber
%&=&  \sum_{i=1}^{s^{\{q\}}} b_i^{\{q\}} 
%\sum_{j_1,\ldots,j_\ell} a_{i,j_1}^{\{q,\nu_1\}} \cdots 
%a_{i,j_\ell}^{\{q,\nu_\ell\}} \mathbf{g}_{j_1}^{\{\nu_1\}}(u_1) \cdots \mathbf{g}_{j_\ell}^{\{\nu_\ell\}}(u_\ell)\,.
\end{eqnarray}

For a regular Runge-Kutta method we ignore the coloring of the tree nodes and consider $u \in T$, where $T$ are the regular Butcher trees. The NB-series expansions are regular
B-series expansions
\begin{eqnarray*}
y_{n+1} &=& \textnormal{B}(\mathbf{c},y_n) \,, \quad
Y_i = \textnormal{B}(\mathbf{c}_i,y_n)\,, \quad
h f\left(Y_i\right)  = \textnormal{B}(\mathbf{f}_i,y_n)\
\end{eqnarray*}
whose coefficients are given by
\begin{eqnarray}
%
%\mathbf{c}_i(u) & = &\sum_{j=1}^{s} a_{i,j} \mathbf{f}_j(u),  \quad i=1,\ldots,s.\\
%
\label{eqn:btrees-stage-fun}
\mathbf{f}_i(u) & = & 
\left\{
\begin{array}{l}
0 \quad \textnormal{for}~~ u = \emptyset\,, \\
1 \quad \textnormal{for}~~ u = \tau\,, \\
 \sum_{j_1,\dots,j_\ell}^{s} a_{i,j_1} \dots a_{i,j_\ell}\; \mathbf{f}_{j_1}(u) \dots  \mathbf{f}_{j_\ell}(u) \\
 \qquad  \textnormal{for}~~ u = [u_1,\ldots,u_\ell]^{\{\bullet\}}\,,
 \end{array}
\right.  \\
\nonumber
&& \quad  i=1,\ldots,s\,, \\
\label{eqn:btrees-solution-expanded}
\mathbf{c}(u) & = & \sum_{i=1}^{s} b_i \mathbf{f}_i(u).
\end{eqnarray}
The RK method has order $p$ iff
\begin{equation}
\label{eqn:RK_bseries-order}
\mathbf{c}(u) = \frac{1}{\gamma(u)}, \quad \forall\, u \in \textnormal{T} \quad \textnormal{with}\quad  1 \leq \rho(u) \leq p\,.
\end{equation}

For any N-tree $u \in \textnormal{NT}$ the recurrences \eqref{eqn:nbtrees-stage-fun-3} and 
\eqref{eqn:nbtrees-solution-expanded} mimic the B-series relations \eqref{eqn:btrees-stage-fun}
and \eqref{eqn:btrees-solution-expanded}, respectively. Each coefficient subscript in 
\eqref{eqn:btrees-stage-fun}, \eqref{eqn:btrees-solution-expanded} is paired with a unique color superscript
in \eqref{eqn:nbtrees-solution-expanded} and \eqref{eqn:btrees-stage-fun}, for example $i \leftrightarrow \{m\}$,
$j_1\leftrightarrow \{\nu_1\}$, and $j_\ell \leftrightarrow \{\nu_\ell\}$. The NB-series coefficient $\mathbf{d}(u)$
written in terms of the GARK scheme coefficients $b_i^{\{m\}}$,   $a_{i,j}^{\{m,\nu\}}$
has precisely the same form as the B-series coefficient $\mathbf{c}(u)$
written in terms of the RK scheme coefficients $b_i$,   $a_{i,j}$, with color superscripts
added to match the corresponding subscripts. 
(The calculation of $\mathbf{c}(u)$ ignores the vertex coloring of $u$.)
Since \eqref{eqn:GARK_nbseries-order}
has to hold for any coloring of the vertices of $u \in \textnormal{NT}$, the corresponding RK condition \eqref{eqn:RK_bseries-order}
has to hold for arbitrary sets of superscripts, provided that they are paired correctly with the coefficient subscripts.
%\qed
\end{proof}

\begin{remark}
For N-trees $u \in \textnormal{NT}$ with all vertices of the same color $q$ one obtains the traditional RK order conditions 
for the individual method  $\left( \mathbf{A}^{\{q,q\}}, \,\mathbf{b}^{\{q\}} \right)$. As an example 
consider the order conditions \eqref{eqn:GARK-order-conditions-1-to-4} for equal superscripts $\sigma=\nu=\lambda=\mu=q$.
Therefore a necessary condition for a GARK method to have order $p$ is that each individual component method has order $p$.
In addition, the GARK method needs to satisfy the coupling conditions resulting from trees with vertices of different colors.
\end{remark}

\begin{remark}
For internally consistent GARK methods \eqref{eqn:GARK-internal-consistency} the order conditions
simplify considerably. 
Conditions \eqref{eqn:GARK-order-conditions-1}, \eqref{eqn:GARK-order-conditions-2}, and \eqref{eqn:GARK-order-conditions-3a}
become
\[
\mathbf{b}^{\{\sigma\}}\,^T \cdot \one^{{\{\sigma\}}} = 1,   \qquad 
 \mathbf{b}^{\{\sigma\}}\,^T \cdot \mathbf{c}^{\{\sigma\}}  =  \frac{1}{2}, \qquad 
\mathbf{b}^{\{\sigma\}}\,^T \cdot \left(  \mathbf{c}^{\{\sigma\}}  \mathbf{c}^{\{\sigma\}} \right)
 =  \frac{1}{3},  \qquad \forall\; \sigma,
\]
and correspond to order two conditions, and to the first order three condition of each individual component method  ($\mathbf{A}^{\{\sigma,\sigma\}}$, $\mathbf{b}^{\{\sigma\}}$),
$ \sigma=1,\ldots, N$.

Condition \eqref{eqn:GARK-order-conditions-3b} becomes
\[
\mathbf{b}^{\{\sigma\}}\,^T \cdot \mathbf{A}^{\{\sigma,\nu\}}  \cdot \mathbf{c}^{\{\nu\}} 
 =  \frac{1}{6},  \quad \forall\; \sigma,\nu\,,
\]
and gives raise to $N$ order three conditions for each component method, and $N^2-N$ coupling conditions.
Compare this with the $N^3-N$ coupling conditions in the absence of internal consistency.

The $4N^4$ order four conditions \eqref{eqn:GARK-order-conditions-4a}--\eqref{eqn:GARK-order-conditions-4d}
reduce to the following $N+2N^2+N^3$ equations
\begin{eqnarray*}
\mathbf{b}^{\{\sigma\}}\,^T \cdot \left( \mathbf{c}^{\{\sigma\}}  \mathbf{c}^{\{\sigma\}}  \mathbf{c}^{\{\sigma\}} \right)
 &=&  \frac{1}{4},  \quad \forall\; \sigma,\nu,\lambda,  \\
\left(\mathbf{b}^{\{\sigma\}} \mathbf{c}^{\{\sigma\}} \right)^T \cdot \mathbf{A}^{\{\sigma,\nu\}}  \cdot \mathbf{c}^{\{\nu\}}  
 &=&  \frac{1}{8},  \quad \forall\; \sigma,\nu,\lambda, \\
\mathbf{b}^{\{\sigma\}}\,^T \cdot \mathbf{A}^{\{\sigma,\nu\}}  \cdot
\left( \mathbf{c}^{\{\nu\}}  \mathbf{c}^{\{\nu\}} \right)
 &=&  \frac{1}{12},  \quad \forall\; \sigma,\nu,\lambda,  \\
\mathbf{b}^{\{\sigma\}}\,^T \cdot  \mathbf{A}^{\{\sigma,\nu\}}  \cdot \mathbf{A}^{\{\lambda,\mu\}}  \cdot \mathbf{c}^{\{\lambda\}}
& = &  \frac{1}{24},  \quad \forall\; \sigma,\nu,\lambda\,.
\end{eqnarray*}
\end{remark}

%%%%%%%%%%%%%%%%%%%%%%%%%%%%%%
\section{Implicit-explicit GARK schemes}\label{sec:imex}
%%%%%%%%%%%%%%%%%%%%%%%%%%%%%%

We now focus on systems \eqref{eqn:additive-ode} with a two-way partitioned right hand side
\begin{equation} 
\label{eqn:imex-ode}
 y'= f(y) + g(y)\,, \quad  y(t_0)=y_0\,,
\end{equation}
where $f$ is non-stiff and $g$ is stiff.

%%%%%%%%%%%%%%%%%%%%%%%%%%%%%%
\subsection{Formulation of IMEX-GARK schemes}
%%%%%%%%%%%%%%%%%%%%%%%%%%%%%%

An implicit-explicit (IMEX) GARK scheme is a two way partitioned method \eqref{eqn:GARK}
where one component method $(\mathbf{A}^\EE,\mathbf{b}^\E)$ is explicit, and the other one $(\mathbf{A}^\II,\mathbf{b}^\I)$ implicit:
\begin{subequations}
\label{eqn:imexRK}
\begin{eqnarray}
\label{eqn:imex-explicit-stage}
Y_i &=& y_{n} + h  \sum_{j=1}^{s^\E}  a_{i,j}^\EE \, f(Y_j)   + h  \sum_{j=1}^{s^\I} a_{i,j}^\EI \, g(Z_j)\,, ~~ i = 1,\dots, s^\E, ~~\\
\label{eqn:imex-implicit-stage}
Z_i  &=& y_{n} + h  \sum_{j=1}^{s^\E} a_{i,j}^\IE \, f(Y_j)   + h  \sum_{j=1}^{s^\I}  a_{i,j}^\II \, g(Z_j)\,, ~~ i = 1,\dots, s^\I,\\
\label{eqn:imex-sol}
y_{n+1} &=& y_n + h \sum_{i=1}^{s^\E}  b_{i}^\E \, f(Y_i)  + h \sum_{i=1}^{s^\I} b_{i}^\I \, g(Z_i)\,.
\end{eqnarray}
\end{subequations}
The corresponding generalized Butcher tableau is
\begin{equation}
\label{eqn:generalized-tableau}
\begin{array}{c|c|c|c}
\mathbf{c}^\EE & \mathbf{A}^\EE &\mathbf{A}^\EI & \mathbf{c}^\EI  \\ 
\mathbf{c}^\IE & \mathbf{A}^\IE & \mathbf{A}^\II  & \mathbf{c}^\II \\
\hline
& \mathbf{b}^\E & \mathbf{b}^\I 
\end{array}\,,
\end{equation}
with
\[
\mathbf{A}^{\{\sigma,\nu\}} \in \Re^{s^{\{\sigma\}} \times s^{\{\nu\}}}\,, \quad
\mathbf{b}^{\{\sigma\}} \in \Re^{s^{\{\sigma\}}}\,, \quad
\mathbf{c}^{\{\sigma,\nu\}} = \mathbf{A}^{\{\sigma,\nu\}} \cdot \one^{\{\nu\}} \in \Re^{s^{\{\sigma\}} }\,,
% \quad \forall \sigma,\nu \in \{I,E\}\,.
\]
for all $\sigma,\nu \in \{\textsc{i},\textsc{e}\}$, where $\one^{\{\nu\}} =[1 \dots 1]^T \in \Re^{s^{\{\nu\}} }$.

%%%%%%%%%%%%%%%%%%%%%%%
\begin{example}[Classical IMEX RK methods]
%%%%%%%%%%%%%%%%%%%%%%%

In \eqref{eqn:imexRK} we use
\[
\mathbf{A}^\IE = \mathbf{A}^\EE \equiv \mathbf{A}^\E  \,, \quad \mathbf{A}^\EI  = \mathbf{A}^\II  \equiv \mathbf{A}^\I \,,
\]
with $\mathbf{A}^\I$ lower triangular and $\mathbf{A}^\E$ strictly lower triangular.
We obtain $s^\E=s^\I=s$, $Y_i = Z_i$, and recover the classical IMEX RK of Ascher, Ruuth, and Spiteri 
\cite{Ascher_1997_IMEX_RK}:
\begin{subequations}
\label{eqn:classical-imexRK}
\begin{eqnarray}
\label{eqn:classical-imex-stage}
Y_i &=& y_{n} + h  \sum_{j=1}^{i-1}  a_{i,j}^\E \, f(Y_j)   + h  \sum_{j=1}^{i} a_{i,j}^\I \, g(Y_j)\,, ~~ i = 1,\dots, s, ~~\\
\label{eqn:classical-imex-sol}
y_{n+1} &=& y_n + h \sum_{i=1}^{s}  b_{i}^\E \, f(Y_i)  + h \sum_{i=1}^{s} b_{i}^\I \, g(Y_i)\,.
\end{eqnarray}
\end{subequations}
\end{example}

%%%%%%%%%%%%%%%%%%%%%%%
\begin{example}[Classical-transposed IMEX RK methods]
%%%%%%%%%%%%%%%%%%%%%%%

In \eqref{eqn:imexRK} we use
\[
\mathbf{A}^\EI = \mathbf{A}^\EE  \equiv \mathbf{A}^\E \,, \quad \mathbf{A}^\IE  = \mathbf{A}^\II  \equiv \mathbf{A}^\I \,,
\]
with $\mathbf{A}^\I$ lower triangular and $\mathbf{A}^\E$ strictly lower triangular,
to obtain $s^\E=s^\I=s$, and the interesting family of schemes:
\begin{subequations}
\label{eqn:transposed-imexRK}
\begin{eqnarray}
\label{eqn:transposed-imex-explicit-stage}
Y_i &=& y_{n} + h  \sum_{j=1}^{i-1}  a_{i,j}^\E \, \left( f(Y_j)   + g(Z_j) \right)\,, ~~ i = 1,\dots, s, ~~\\
\label{eqn:transposed-imex-implicit-stage}
Z_i  &=& y_{n} + h  \sum_{j=1}^{i} a_{i,j}^\I \, \left( f(Y_j)   + g(Z_j) \right)\,, 
~~ i = 1,\dots, s,\\
\label{eqn:transposed-imex-sol}
y_{n+1} &=& y_n + h \sum_{i=1}^{s}  b_{i}^\E \, f(Y_i)  + h \sum_{i=1}^{s} b_{i}^\I \, g(Z_i)\,,
\end{eqnarray}
\end{subequations}
\end{example}

The low storage IMEX methods recently presented by Higueras and Roldan \cite{Higueras_2013_scicade}
can also be formulated as IMEX-GARK schemes \eqref{eqn:imexRK}.

%%%%%%%%%%%%%%%%%%%%%%%
\subsection{Order conditions}
%%%%%%%%%%%%%%%%%%%%%%%

The GARK order conditions in matrix form are given in \eqref{eqn:GARK-order-conditions-1-to-4}.
Each of the implicit and explicit methods $\left(\mathbf{A}^{\{\sigma,\sigma\}}, \mathbf{b}^{\{\sigma\}}\right)$ and $\mathbf{c}^{\{\sigma,\sigma\}} = \mathbf{A}^{\{\sigma,\sigma\}} \cdot \one$ has to satisfy the corresponding order conditions for $\sigma  \in \{ \textsc{e},\textsc{i} \}$. In addition, two  coupling conditions \eqref{eqn:coupling-2} are required for second order,
and 12 coupling conditions \eqref{eqn:coupling-3} are required for third order accuracy. These conditions are listed in Appendix \ref{sec:oder-conditions-imex}.

The GARK internal consistency condition \eqref{eqn:GARK-internal-consistency} reads
\begin{equation}
\label{eqn:simplifying-assumption-c}
\mathbf{c}^\IE = \mathbf{c}^\II = \mathbf{c}^\I\,, \quad \mathbf{c}^\EI = \mathbf{c}^\EE  = \mathbf{c}^\E\,.
\end{equation}
These conditions are automatically satisfied in the case of classical-transposed IMEX RK  methods.
In case of classical IMEX RK equations \eqref{eqn:simplifying-assumption-c} are equivalent to $\mathbf{c}^\I = \mathbf{c}^\E$.

Assuming that both the implicit and the explicit methods have order at least three,
\eqref{eqn:simplifying-assumption-c} implies that the
the second order coupling conditions \eqref{eqn:coupling-2} are automatically satisfied.
The third order coupling conditions \eqref{eqn:coupling-3} reduce to:
\begin{eqnarray*}
\mathbf{b}^\E\,^T \cdot \mathbf{A}^\EI  \cdot \mathbf{c}^\I   & = & \frac{1}{6}\,, \quad
\mathbf{b}^\I\,^T \cdot \mathbf{A}^\IE  \cdot \mathbf{c}^\E    =  \frac{1}{6}\,. 
\end{eqnarray*}

Assume that the implicit and explicit methods have order at least four. With \eqref{eqn:simplifying-assumption-c} 
the order conditions \eqref{eqn:GARK-order-conditions-4a}--\eqref{eqn:GARK-order-conditions-4d} reduce to the following order four coupling relations:
\begin{eqnarray}
\nonumber
\left(\mathbf{b}^{\{\sigma\}} \mathbf{c}^{\{\sigma\}} \right)^T \cdot \mathbf{A}^{\{\sigma,\nu\}}  \cdot \mathbf{c}^{\{\nu\}}  
& = & \frac{1}{8},  \quad \forall \sigma,\nu  \in \{ \textsc{e},\textsc{i} \}, ~ \sigma \ne \nu \,,\\
\label{eqn:imex-coupling-order-4}
\mathbf{b}^{\{\sigma\}}\,^T \cdot \mathbf{A}^{\{\sigma,\nu\}}  \cdot
\left( \mathbf{c}^{\{\nu\}}  \mathbf{c}^{\{\nu\}} \right)
& = & \frac{1}{12},  \quad  \forall \sigma,\nu  \in \{ \textsc{e},\textsc{i} \}, ~ \sigma \ne \nu \,, \\
\nonumber
\mathbf{b}^{\{\sigma\}}\,^T \cdot  \mathbf{A}^{\{\sigma,\nu\}}  \cdot \mathbf{A}^{\{\lambda,\mu\}}  \cdot \mathbf{c}^{\{\lambda\}}
& = & \frac{1}{24},  \quad \forall \sigma,\nu,\lambda, \mu  \in \{ \textsc{e},\textsc{i} \}\,.
\end{eqnarray}
These order conditions are listed explicitly in Appendix \ref{sec:oder-conditions-imex-4}.

Additional simplifying assumptions 
\begin{equation}
\label{eqn:simplifying-assumption-additional}
s^\E = s^\I\,, \quad \mathbf{b}^\E = \mathbf{b}^\I =  \mathbf{b}\,, \quad  \mathbf{c}^\E = \mathbf{c}^\I = \mathbf{c}\,.
\end{equation}
further reduce the number of order conditions; these are are listed  in Appendix \ref{sec:oder-conditions-imex-4}.

%%%%%%%%%%%%%%%%%%%%%%%
\subsection{Construction of classical-transposed IMEX RK}
%%%%%%%%%%%%%%%%%%%%%%%

Consider now the classical-transposed  methods \eqref{eqn:transposed-imexRK} and assume that both the explicit and the implicit method
have order three. Since \eqref{eqn:simplifying-assumption-c} holds, the remaining third order coupling conditions read:

\begin{equation}
\label{eqn:transposed-copuling-order-3}
\mathbf{b}^\E\,^T \cdot \mathbf{A}^\E  \cdot \mathbf{c}^\I  = \frac{1}{6}\,, \quad
\mathbf{b}^\I\,^T \cdot \mathbf{A}^\I  \cdot \mathbf{c}^\E   = \frac{1}{6}\,. 
\end{equation}

Assuming that the explicit and implicit methods have each order at least four,
the order four coupling conditions are:
\begin{equation}
\label{eqn:transposed-copuling-order-4} 
\begin{array}{rcl}
\left(\mathbf{b}^\E \mathbf{c}^\E \right)^T \cdot \mathbf{A}^\E   \cdot \mathbf{c}^\I  & = & \frac{1}{8} \\ 
\mathbf{b}^\E\,^T \cdot \mathbf{A}^\E  \cdot\left( \mathbf{c}^\I  \mathbf{c}^\I \right)& = & \frac{1}{12} \\ 
\mathbf{b}^\E\,^T \cdot \mathbf{A}^\E\cdot \mathbf{A}^\I\cdot \mathbf{c}^\I & = & \frac{1}{24} \\ 
\left(\mathbf{b}^\I \mathbf{c}^\I \right)^T \cdot \mathbf{A}^\I   \cdot \mathbf{c}^\E  & = & \frac{1}{8} \\ 
\mathbf{b}^\I\,^T \cdot \mathbf{A}^\I  \cdot\left( \mathbf{c}^\E  \mathbf{c}^\E \right)& = & \frac{1}{12} \\ 
\mathbf{b}^\I\,^T \cdot \mathbf{A}^\I\cdot \mathbf{A}^\E\cdot \mathbf{c}^\E & = & \frac{1}{24} 
\end{array}
\end{equation}
If, in addition, $\mathbf{b}^\E=\mathbf{b}^\I=\mathbf{b}$ and $\mathbf{c}^\E=\mathbf{c}^\I=\mathbf{c}$, then the third order coupling conditions  
\eqref{eqn:transposed-copuling-order-3} are automatically satisfied. The remaining order four coupling conditions \eqref{eqn:transposed-copuling-order-4} read:
\[
\mathbf{b}^T \cdot \mathbf{A}^\E\cdot \mathbf{A}^\I\cdot \mathbf{c}  =  \frac{1}{24} \,, \quad 
\mathbf{b}^T \cdot \mathbf{A}^\I\cdot \mathbf{A}^\E\cdot \mathbf{c}  = \frac{1}{24} \,.
\]

\paragraph{Example: an order three IMEX method} 

%As an example, consider Butcher's order three SDIRK method:
%%
%\[
%\mathbf{A}^\I = \begin{bmatrix}
%\gamma & 0 & 0 \\
%    (-\gamma+1)/2 & \gamma & 0 \\
%    (-6 \gamma^2+16 \gamma-1)/4 & (6 \gamma^2-20 \gamma+5)/4 & \gamma 
%    \end{bmatrix}\,, \quad
%\mathbf{b}^\I = \begin{bmatrix}    
%(-6 \gamma^2+16 \gamma-1)/4 \\ (6 \gamma^2-20 \gamma+5)/4 \\ \gamma 
%\end{bmatrix}\,,
%\]
%%
%with
%

The implicit part is Kvaerno's four stages, order three method \cite[ESDIRK 3/2]{Kvaerno_2004_ESDIRK}:
\[
\gamma = 0.435866521508459\,,
\]
\[
\mathbf{A}^\I = \begin{bmatrix} 
                 0           &         0          &          0 & 0 \\
   \gamma    & \gamma &   0 & 0 \\
   0.490563388421781    &0.073570090069760  &  \gamma & 0 \\
   0.308809969976747    &1.490563388421781   & -1.235239879906987 & \gamma \\
\end{bmatrix}\,,
\]
\[
\mathbf{b}^\I = \left[ \begin{array}{r}    
  0.308809969976747 \\
   1.490563388421781 \\
  -1.235239879906987 \\
   0.435866521508459 \\
\end{array}\right]\,, \quad
\mathbf{c}^\I = \begin{bmatrix}  
                   0\\
   0.871733043016918\\
   1\\
   1
\end{bmatrix}\,. 
\]   
The explicit method has $\mathbf{b}^\E=\mathbf{b}^\I$, $\mathbf{c}^\E=\mathbf{c}^\I$, and
\[
\mathbf{A}^\E = \begin{bmatrix}                   
0         &          0          &         0 & 0 \\
   0.871733043016918     &              0            &       0 & 0 \\
   1            &       0             &      0 & 0 \\
   0.5   & 0.916993298352020 & -0.416993298352020 & 0 
\end{bmatrix}\,.
\]

\paragraph{Example: an order four IMEX method} 
The implicit part is Kvaerno's five stages, order four method \cite[ESDIRK 4/3]{Kvaerno_2004_ESDIRK}:
\[
\begin{array}{rcrrcr}
\mathbf{A}^\I_{2,2} &=& \mathbf{A}^\I_{3,3} = \mathbf{A}^\I_{4,4} = \mathbf{A}^\I_{5,5} &=& 0.572816062482134 \\
\mathbf{A}^\I_{2,1}  &=&   0.572816062482134, \quad
\mathbf{A}^\I_{3,1}  &=&   0.167235462027210, \\
\mathbf{A}^\I_{3,2}  &=&   -0.142946536857034, \quad
\mathbf{A}^\I_{4,1}  &=&    0.262603290252694, \\
\mathbf{A}^\I_{4,2}  &=&  -0.311904327420564, \quad
\mathbf{A}^\I_{4,3}  &=&   0.476484974685735, \\
\mathbf{A}^\I_{5,1}  &=&   0.197216548312835, \quad
\mathbf{A}^\I_{5,2}  &=&   0.176843783906372, \\
\mathbf{A}^\I_{5,3}  &=&   0.815442181350836, \quad
\mathbf{A}^\I_{5,4}  &=&  -0.762318576052177, \\
\end{array}
\]
with
\[
\mathbf{b} = \left[\begin{array}{r}     
       0.197216548312835 \\
   0.176843783906372 \\
   0.815442181350836 \\
  -0.762318576052177 \\
   0.572816062482134
 \end{array}\right]  \,, \quad
\mathbf{c} = \begin{bmatrix}     
    0 \\
   1.145632124964268 \\
   0.597104987652310 \\
   1 \\
   1
 \end{bmatrix}  \,,
\]
together with the explicit method
\[
\begin{array}{rcrrcr}
\mathbf{A}^\E_{2,1}  &=&   1.145632124964268, \quad
\mathbf{A}^\E_{3,1}  &=&      0.486402211775915, \\
\mathbf{A}^\E_{3,2}  &=&      0.110702775876395, \quad
\mathbf{A}^\E_{4,1}  &=&      0.527357281908146, \\
\mathbf{A}^\E_{4,2}  &=&     -0.234882275336215, \quad
\mathbf{A}^\E_{4,3}  &=&      0.707524993428070, \\
\mathbf{A}^\E_{5,1}  &=&                    0, \quad
\mathbf{A}^\E_{5,2}  &=&   -0.515140880433405, \\
\mathbf{A}^\E_{5,3}  &=&    1.515140880433405, \quad
\mathbf{A}^\E_{5,4}  &=&                    0\,.
\end{array}
\]
All other coefficients are zero.
The above form a transposed-classical  IMEX RK method of order four.

%%%%%%%%%%%%%%%%%%%%%%%%%%%%%%%%%
\subsection{Prothero-Robinson analysis}\label{sec:PR}
%%%%%%%%%%%%%%%%%%%%%%%%%%%%%%%%%

We consider the Prothero-Robinson (PR) \cite{Prothero_1974_PR} test problem written as a split system \eqref{eqn:imex-ode}
%
%\begin{equation}
%y' = \underbrace{ \mu\, (y - \phi(t)) }_{g(y)} + \underbrace{\phi'(t) }_{f(y)} ~, \quad \mu < 0~, \quad y(0)=\phi(0)~,
%\label{Prothero-Robinson}
%\end{equation}
%
\begin{equation}
\begin{bmatrix} y \\  t \end{bmatrix}' = \underbrace{ \begin{bmatrix} \mu\, (y - \phi(t))  \\ 0  \end{bmatrix} }_{g(t,y)} + \underbrace{  \begin{bmatrix}  \phi'(t) \\ 1 \end{bmatrix} }_{f(t,y)} ~, \quad \mu < 0~, \quad y(0)=\phi(0)\,,
\label{Prothero-Robinson}
\end{equation}
where the exact solution is $y(t)=\phi(t)$. An IMEX-GARK method \eqref{eqn:imexRK} is PR-convergent with order $p$
if its application to (\ref{Prothero-Robinson}) gives a solution whose global error decreases as $\mathcal{O}(h^p)$
for $h \rightarrow 0$ and $h\mu \rightarrow -\infty$.

\begin{theorem}[PR convergence of IMEX-GARK methods]
\label{thm:PR-convergence}
Consider the IMEX-GARK method \eqref{eqn:imexRK} of order $p$.
Assume that the implicit component has a nonsingular coefficient matrix $\mathbf{A}^\II$, and an implicit
transfer function stable at infinity, 
\[
R^\II(h \mu) =  \left( 1 +  h   \, \mu \, \mathbf{b}^\I\,^T \, \left( \mathbf{I} - h\,\mu\, \mathbf{A}^\II \right)^{-1}\, \one \right)\,,
\qquad |R^\II(- \infty)| \le 1\,.
\]
Assume also that the additional order conditions hold:
\begin{equation}
\label{eqn:PR-order-conditions}
k\,  \mathbf{b}^\I\,^T \,  \left( \mathbf{A}^\II \right)^{-1}\,  \mathbf{A}^\IE \left(\mathbf{c}^\EE\right)^{k-1} =  \mathbf{b}^\I\,^T \,\left(\mathbf{A}^\II \right)^{-1}\, \left(\mathbf{c}^\IE\right)^k 
\end{equation}
for $k=1,\dots,q$, $q \le p$. Then the IMEX-GARK method is PR-convergent with order:
\begin{itemize}
\item $q+1$ if $-1 \le R^\II(- \infty) \le 1$, and
\item $q$ if $R^\II(- \infty) = 1$.
\end{itemize} 
\end{theorem}

\begin{proof}
%
%The method \eqref{eqn:imexRK} applied to the scalar equation (\ref{Prothero-Robinson}) reads:
%%
%\begin{subequations}
%%\label{eqn:imex-on-PR}
%\begin{eqnarray*}
%%\label{eqn:imex-PR-explicit-stage}
%Y &=& y_{n} \, \one + h  \,  \mathbf{A}^\EE \, f(Y)   + h  \, \mathbf{A}^\EI \, g(Z)\,, \\
%%\label{eqn:imex-PR-implicit-stage}
%Z  &=& y_{n} \, \one + h  \,  \mathbf{A}^\IE \, f(Y)   + h \mathbf{A}^\II \, g(Z)\,, \\
%%\label{eqn:imex-sol}
%y_{n+1} &=& y_n + h \, \mathbf{b}^\E\,^T \, f(Y)  + h  \, \mathbf{b}^\I\,^T \, g(Z)\,.
%\end{eqnarray*}
%\end{subequations}
%%

The method \eqref{eqn:imexRK} applied to the scalar equation (\ref{Prothero-Robinson}) reads
\begin{subequations}
\label{eqn:imex-on-PR}
\begin{eqnarray}
\label{eqn:imex-PR-explicit-stage}
Y &=& y_{n} \, \one + h  \,  \mathbf{A}^\EE \, \phi'^\E   + h  \, \mu \, \mathbf{A}^\EI \, \left( Z - \phi^\I \right) \,, \\
\label{eqn:imex-PR-implicit-stage}
Z  &=& y_{n}\, \one  + h  \,  \mathbf{A}^\IE \, \, \phi'^\E   + h  \, \mu \, \mathbf{A}^\II \,   \left( Z - \phi^\I \right) \,, \\
\label{eqn:imex-sol2}
y_{n+1} &=& y_n + h \, \mathbf{b}^\E\,^T \,\, \phi'^\E  + h   \, \mu \, \mathbf{b}^\I\,^T \,  \left( Z - \phi^\I \right) \,.
\end{eqnarray}
\end{subequations}
Here 
\begin{eqnarray*}
\phi^\E &=&  \phi\left(t_{n-1} + \mathbf{d}^\E\, h \right) = \left[  \phi(t_{n-1} +d_1^\E\, h ),\ldots,\phi(t_{n-1} + d_{s^\E}^\E\, h ) \right]^T\,, \\
\phi^\I &=&  \phi\left(t_{n-1} + \mathbf{d}^\I\, h \right) = \left[  \phi(t_{n-1} +d_1^\I\, h ),\ldots,\phi(t_{n-1} + d_{s^\I}^\I\, h ) \right]^T\,, \\
\end{eqnarray*}
where $\mathbf{d}^\E$, $\mathbf{d}^\I$ are the stage approximation times.  Due to the structure of the test problem \eqref{Prothero-Robinson},
the method \eqref{eqn:imexRK} uses an explicit approach for the time variable, therefore
\begin{equation}
\label{eqn:d-values}
\mathbf{d}^\E = \mathbf{c}^\EE\,, \quad \mathbf{d}^\I = \mathbf{c}^\IE\,.
\end{equation}
The exact solution is expanded in Taylor series about $t_{n}$:
\begin{equation}
\label{eqn:phi-taylor}
\begin{array}{rcl}
\phi\left(t_{n} + \mathbf{d}\, h \right)-\one\,\phi(t_{n}) &=& \displaystyle \sum_{k=1}^\infty \frac{h^k \mathbf{d}^k }{k!}\phi^{(k)}(t_{n}) \,,\\
h\, \phi'\left(t_{n} + \mathbf{d}\, h \right) &=& \displaystyle  \sum_{k=1}^\infty \frac{k h^k \mathbf{d}^{k-1} }{k!}\phi^{(k)}(t_{n}) \,,
\end{array}
\end{equation}
where the vector power $\mathbf{d}^{k}$ is taken componentwise.

Consider the global errors
\begin{eqnarray*}
&& e_n = y_n - \phi(t_n)~, \quad E_Y = Y - \phi^\E, \quad E_Z = Z - \phi^\I\,.
\end{eqnarray*}
%
%Write the stage equation \eqref{eqn:imex-PR-explicit-stage} in terms of the exact solution and global errors
%%
%\begin{eqnarray}
% \label{eqn:imex-PR-explicit-stage-error}
% E_Y &=& e_{n} \, \one + \phi(t_n) \, \one  - \phi^\E + h  \,  \mathbf{A}^\EE \, \phi'^\E   + h  \, \mu \, \mathbf{A}^\EI \,  E_Z\,,
%\end{eqnarray}
%
Write the stage equation \eqref{eqn:imex-PR-implicit-stage} in terms of the exact solution and global errors,
and use the Taylor expansions \eqref{eqn:phi-taylor} to obtain
\begin{eqnarray*}
\left( \mathbf{I} - h\,\mu\, \mathbf{A}^\II \right)\, E_Z  &=& e_{n} \, \one + \phi(t_n) \, \one  - \phi^\I  + h  \,  \mathbf{A}^\IE \, \, \phi'^\E  \\
 &=& e_{n} \, \one + \sum_{k=1}^\infty \left( k\, \mathbf{A}^\IE \left(\mathbf{d}^\E\right)^{k-1} - \left(\mathbf{d}^\I\right)^k  \right) \frac{h^k}{k!}\phi^{(k)}(t_{n})\,.
\end{eqnarray*}
%
%and
%%
%\begin{eqnarray*}
%E_Y &=& e_{n} \, \one + \delta_Y  + h  \, \mu \, \mathbf{A}^\EI \,  \left( \mathbf{I} - h\,\mu\, \mathbf{A}^\II \right)^{-1}\, \left( \one + \delta_Z \right)\,, \\
%\delta_Y &=&   \sum_{k=1}^\infty \left( k\, \mathbf{A}^\EE \left(\mathbf{c}^\EE\right)^{k-1} - \left(\mathbf{c}^\EE\right)^k  \right) \frac{h^k}{k!}\phi^{(k)}(t_{n})\,.
%\end{eqnarray*}
%
Similarly, write the solution equation (\ref{eqn:imex-sol2}) in terms of the exact solution and global errors:
\begin{eqnarray*}
e_{n+1} &=& e_n  + \phi(t_n) - \phi(t_{n+1})+ h \, \mathbf{b}^\E\,^T \,\, \phi'^\E  + h   \, \mu \, \mathbf{b}^\I\,^T \, E_Z \\
%&=& e_n + \sum_{k=1}^\infty \left(k\,  \mathbf{b}^\E\,^T\left( \mathbf{d}^\E \right)^{k-1}-1\right) \frac{h^k \,  }{k!}\phi^{(k)}(t_{n}) \\
%&& + h   \, \mu \, \mathbf{b}^\I\,^T \, \left( \mathbf{I} - h\,\mu\, \mathbf{A}^\II \right)^{-1}\, \left( e_{n} \, \one + \delta_Z \right) \\
&=& R^\II(h \mu) \, e_{n} + \sum_{k=1}^\infty \left(k\,  \mathbf{b}^\E\,^T\left( \mathbf{d}^\E \right)^{k-1}-1\right) \frac{h^k \,  }{k!}\phi^{(k)}(t_{n}) \\
&& + h   \, \mu \, \mathbf{b}^\I\,^T \, \left( \mathbf{I} - h\,\mu\, \mathbf{A}^\II \right)^{-1}\cdot \\
&& \quad \cdot\sum_{k=1}^\infty \left( k\, \mathbf{A}^\IE \left(\mathbf{d}^\E\right)^{k-1}   
-\left(\mathbf{d}^\I\right)^k  \right) \frac{h^k}{k!}\phi^{(k)}(t_{n})\,,
\end{eqnarray*}
where the stability function of the implicit component method is
\[
R^\II(h \mu) =  \left( 1 +  h   \, \mu \, \mathbf{b}^\I\,^T \, \left( \mathbf{I} - h\,\mu\, \mathbf{A}^\II \right)^{-1}\, \one \right)\,.
\]
Since the explicit component method (by itself) has at least order $p$, it follows from \eqref{eqn:d-values} and the explicit order conditions that
\begin{equation}
\label{eqn:PR-explicit-order-cancellation}
k \cdot \left( \mathbf{b}^\E\right)^T\left( \mathbf{d}^\E \right)^{k-1}-1 = 0 \quad \mbox{for}~~ k = 1,\dots,p\,.
\end{equation}
Consequently, the global error recurrence reads
\begin{eqnarray*}
e_{n+1} &=& R^\II(h \mu)\, e_n  + \mathcal{O}\left( h^{p+1} \right)  + \sum_{k=1}^\infty r_k(h \mu)\, \frac{h^k \,  }{k!}\phi^{(k)}(t_{n})\,, \\
 r_k(h \mu) &=&  h   \, \mu \, \mathbf{b}^\I\,^T \, \left( \mathbf{I} - h\,\mu\, \mathbf{A}^\II \right)^{-1}\,
 \left( k\, \mathbf{A}^\IE \left(\mathbf{d}^\E\right)^{k-1} - \left(\mathbf{d}^\I\right)^k  \right)\,.
\end{eqnarray*}
For $h\mu \to -\infty$
\begin{eqnarray*}
 r_k(\infty) &=&  - \mathbf{b}^\I\,^T \, \left(  \mathbf{A}^\II \right)^{-1}\, \left( k\, \mathbf{A}^\IE \left(\mathbf{d}^\E\right)^{k-1} - \left(\mathbf{d}^\I\right)^k  \right)\,.
\end{eqnarray*}

If conditions \eqref{eqn:PR-order-conditions} hold we infer from \eqref{eqn:d-values} that $r_k(\infty) = 0$ for $k=1,\dots,q$.
The global error recurrence reads
\begin{eqnarray*}
e_{n+1} &=& R^\II(\infty)\, e_n  + \mathcal{O}\left( h^{\min\{q+1,p+1\}} \right) \,.
\end{eqnarray*}
This gives the desired result.
%
%\qed
\end{proof}

\begin{remark}
If one defines the test problem \eqref{Prothero-Robinson} such that the time is an implicit variable then \eqref{eqn:d-values} becomes
\begin{equation}
\label{eqn:PR-implicit-time}
\mathbf{d}^\E = \mathbf{c}^\EI\,, \quad \mathbf{d}^\I = \mathbf{c}^\II\,.
\end{equation}
Equation \eqref{eqn:PR-explicit-order-cancellation} is satisfied automatically by the IMEX coupling conditions of order $p$.
The additional order conditions \eqref{eqn:PR-order-conditions} become
\begin{equation}
\label{eqn:alternatice-PR-conditions}
k\,  \mathbf{b}^\I\,^T \,  \left( \mathbf{A}^\II \right)^{-1}\,  \mathbf{A}^\IE \left(\mathbf{c}^\EI\right)^{k-1} =  \mathbf{b}^\I\,^T \,\left(\mathbf{A}^\II \right)^{-1}\, \left(\mathbf{c}^\II\right)^k \,.
\end{equation}
\end{remark}

%%%%%%%%%%%%%%%%%%%%%%%%%%%%%%%%%
\subsection{Stiff semi-linear analysis}\label{sec:SL}
%%%%%%%%%%%%%%%%%%%%%%%%%%%%%%%%%

Consider now the semi-linear (SL) problem
\begin{equation}
y' = \underbrace{ \mu\, y  }_{g(y)} + f(y) ~, \quad \mu < 0~, \quad y(0)=y_0~,
\label{semi-linear}
\end{equation}
where $f(y(t))$ is smooth and non-stiff. An IMEX-GARK method \eqref{eqn:imexRK} is SL-convergent with order $p$
if its application to (\ref{semi-linear}) gives a solution whose global error decreases as $\mathcal{O}(h^p)$
for $h \rightarrow 0$ and $h\mu \rightarrow -\infty$.

\begin{theorem}[SL convergence of IMEX-GARK methods]
\label{thm:SL-convergence}
Consider an internally consistent IMEX-GARK method \eqref{eqn:imexRK} of order $p$.
Assume that the implicit component has a nonsingular coefficient matrix $\mathbf{A}^\II$, and an implicit
transfer function strictly stable at infinity, 
\[
R^\II(h \mu) =  \left( 1 +  h   \, \mu \, \mathbf{b}^\I\,^T \, \left( \mathbf{I} - h\,\mu\, \mathbf{A}^\II \right)^{-1}\, \one \right)\,,
\quad
\left|R^\II(- \infty)\right| < 1\,.
\]
Assume also that the additional order conditions hold:
\begin{subequations}
\label{eqn:SL-order-conditions}
\begin{eqnarray}
\mathbf{b}^\I\,^T  \, \mathbf{A}^\II\,^{-1} \mathbf{c}^\I\,^k & =& 1 \,, \quad k=1,\dots,q\,, \\
\mathbf{b}^\I\,^T  \, \mathbf{A}^\II\,^{-1}  \mathbf{A}^\IE \mathbf{c}^\E\,^{k-1}  &=& \frac{1}{k}\,,
\quad k=1,\dots,q\,,
\end{eqnarray}
\end{subequations}
for $q \le p$. Then the IMEX-GARK method is SL-convergent with order $q+1$.
\end{theorem}

\begin{proof}
The exact solution is expanded in Taylor series about $t_{n}$:
\begin{eqnarray*}
y \left(t_{n} + \mathbf{c}\, h \right)-\one\,y(t_{n}) &=& \sum_{k=1}^\infty \frac{h^k \mathbf{c}^k }{k!} y^{(k)}(t_{n}) \,,\\
h\, y'\left(t_{n} + \mathbf{c}\, h \right) &=& \sum_{k=1}^\infty \frac{k h^k \mathbf{c}^{k-1} }{k!} y^{(k)}(t_{n}) \,.
\end{eqnarray*}
Application of the method \eqref{eqn:imexRK} to problem \eqref{semi-linear} gives
\begin{eqnarray}
\nonumber
Y &=& y_{n} \, \one + h  \,  \mathbf{A}^\EE \, f(Y)   + h  \, \mu \, \mathbf{A}^\EI \,  Z \,, \\
\label{eqn:IMEX-on-SL}
Z  &=& y_{n}\, \one  + h  \,  \mathbf{A}^\IE \, f(Y)   + h  \, \mu \, \mathbf{A}^\II \,  Z \,, \\
\nonumber
y_{n+1} &=& y_n + h \, \mathbf{b}^\E\,^T\, f(Y)  + h   \, \mu \, \mathbf{b}^\I\,^T \, Z \,.
\end{eqnarray}
Insert the exact solutions in the numerical scheme \eqref{eqn:IMEX-on-SL} to obtain
\begin{eqnarray}
\nonumber
y(t_n^\E) &=&  y(t_n) \, \one + h  \,  \mathbf{A}^\EE \, f\left(y(t_n^\E)\right)   + h  \, \mu \, \mathbf{A}^\EI \,  y(t_n^\I) + \delta^\E\,, \\
\label{eqn:IMEX-on-SL-exact}
\quad y(t_n^\I)  &=&  y(t_n)\, \one  + h  \,  \mathbf{A}^\IE \, f\left(y(t_n^\E)\right)   + h  \, \mu \, \mathbf{A}^\II \,  y(t_n^\I) + \delta^\I \,, \\
\nonumber
y(t_{n+1}) &=& y(t_n) + h \, \mathbf{b}^\E\,^T \,\, f\left(y(t_n^\E)\right)  + h   \, \mu \, \mathbf{b}^\I\,^T \, y(t_n^\I) +\delta\,.
\end{eqnarray}
where $t_n^\E = t_n+\mathbf{c}^\E h$ and $t_n^\I = t_n+\mathbf{c}^\I h$.
The exact solutions satisfy the numerical scheme \eqref{eqn:IMEX-on-SL} only approximately.
The residuals are as follows:
%
%\begin{eqnarray*}
% \delta^\E &=& y(t+\mathbf{c}^\E h) - y(t_n) \, \one  -   \mathbf{A}^\EE \, hy'(t+\mathbf{c}^\E h) \\
%&& + h \mu \,  \mathbf{A}^\EE \, y(t+\mathbf{c}^\E h)  - h  \, \mu \, \mathbf{A}^\EI \,  y(t+\mathbf{c}^\I h) \\
%&=& \sum_{k=1}^\infty \frac{h^k \mathbf{c}^\E\,^k }{k!}y^{(k)}(t_{n}) 
%-  \sum_{k=1}^\infty \frac{k h^k \mathbf{A}^\EE \mathbf{c}^\E\, ^{k-1} }{k!}y^{(k)}(t_{n}) \\
%&& + h\mu (\mathbf{c}^\EE-\mathbf{c}^\EI) y(t_n) \\
%&& + \sum_{k=1}^\infty \frac{h^k (\mathbf{A}^\EE \mathbf{c}^\E\,^k - \mathbf{A}^\II \mathbf{c}^\I\,^k)}{k!}y^{(k)}(t_{n}) 
%\end{eqnarray*}
%
\begin{eqnarray*}
\delta^\I &=&
y(t+\mathbf{c}^\I h)  -  y(t_n)\, \one  - h  \,  \mathbf{A}^\IE \, \, f(y(t+\mathbf{c}^\E h))   \\
&& + h \mathbf{A}^\II \,f(y(t+\mathbf{c}^\I h))    -  \mathbf{A}^\II \,  h y'(t+\mathbf{c}^\I h) \\ 
%&=& \sum_{k=1}^\infty \left(\mathbf{c}^\I\,^k - k  \mathbf{A}^\II \mathbf{c}^\I\,^{k-1} \right)\, \frac{h^k  }{k!}y^{(k)}(t_{n}) \\
%&&  - h \sum_{k=0}^\infty \frac{h^k \mathbf{A}^\IE \mathbf{c}^\E\,^k }{k!}f^{(k)}(t_{n})  \\
%&&  + h \sum_{k=0}^\infty \frac{h^k \mathbf{A}^\II \mathbf{c}^\I\,^k }{k!}f^{(k)}(t_{n}) \\
&=& \sum_{k=1}^\infty \left(\mathbf{c}^\I\,^k - k  \mathbf{A}^\II \mathbf{c}^\I\,^{k-1} \right)\, \frac{h^k  }{k!}y^{(k)}(t_{n}) \\
&& + h \sum_{k=0}^\infty \left(  \mathbf{A}^\II \mathbf{c}^\I\,^k - \mathbf{A}^\IE \mathbf{c}^\E\,^k \right) \frac{h^k }{k!}f^{(k)}(t_{n})
\end{eqnarray*}
%
%or that
%%
%\begin{eqnarray*}
% \delta^\I &=& y(t+\mathbf{c}^\I h) - y(t_n) \, \one  -   \mathbf{A}^\IE \, hy'(t+\mathbf{c}^\I h) \\
%&&  + h  \, \mu \, \mathbf{A}^\IE \,  y(t+\mathbf{c}^\E h) - h \mu \,  \mathbf{A}^\II \, y(t+\mathbf{c}^\I h)  \\
%&=& \sum_{k=1}^\infty \frac{h^k \left( \mathbf{c}^\I\,^k - k \mathbf{A}^\IE \mathbf{c}^\I\, ^{k-1} \right) }{k!}y^{(k)}(t_{n}) 
% \\
%&& + h\mu\, \sum_{k=0}^\infty \frac{h^k (\mathbf{A}^\IE \mathbf{c}^\E\,^k - \mathbf{A}^\II \mathbf{c}^\I\,^k)}{k!}y^{(k)}(t_{n}) 
%\end{eqnarray*}
%
For the solution equation we have that
\begin{eqnarray*}
\delta &=& y(t_{n+1}) - y(t_n) - h \, \mathbf{b}^\E\,^T \,\, f(y(t+\mathbf{c}^\E h))  \\
&& + \mathbf{b}^\I\,^T \, h f(y(t+\mathbf{c}^\I h)) - \mathbf{b}^\I\,^T \, h y'(t+\mathbf{c}^\I h) \\
&=& \sum_{k=1}^\infty \left(1 - k  \mathbf{b}^\I\,^T (\mathbf{c}^\I)^{k-1} \right) \frac{h^k}{k!}y^{(k)}(t_{n}) \\
&&  -h \mathbf{b}^\E\,^T \one\,f(t_{n}) - h\, \sum_{k=1}^\infty \frac{h^k \mathbf{b}^\E\,^T \mathbf{c}^\E\,^k }{k!}f^{(k)}(t_{n}) \\
&&  + h \mathbf{b}^\I\,^T\one\,f(t_{n}) + h\, \sum_{k=1}^\infty \frac{h^k \mathbf{b}^\I\, ^T \mathbf{c}^\I\, ^k }{k!}f^{(k)}(t_{n}) \\
&=& \sum_{k=1}^\infty \left(1 - k  \mathbf{b}^\I\,^T (\mathbf{c}^\I)^{k-1} \right) \frac{h^k}{k!}y^{(k)}(t_{n}) \\
&&  + h\, \sum_{k=1}^\infty \left( \mathbf{b}^\I\,^T \mathbf{c}^\I\,^k-\mathbf{b}^\E\,^T \mathbf{c}^\E\,^k \right)\, \frac{h^k }{k!}f^{(k)}(t_{n}) \\
&=& \mathcal{O}\left( h^{p+1} \right)
\end{eqnarray*}
where the last equality follows from the  order $p$ conditions for the implicit and explicit components.
%
%or that
%%
%\begin{eqnarray*}
%\delta &=& 
%y(t_{n+1}) - y(t_n) - \mathbf{b}^\E\,^T \,h y'(t+\mathbf{c}^\E h)  \\
%&& + h\mu \, \mathbf{b}^\E\,^T \,y(t+\mathbf{c}^\E h) - h   \, \mu \, \mathbf{b}^\I\,^T \, y(t+\mathbf{c}^\I h)\\
%&=& \sum_{k=1}^\infty \left(1 - k  \mathbf{b}^\E\,^T (\mathbf{c}^\E)^{k-1} \right) \frac{h^k}{k!}y^{(k)}(t_{n}) \\
%&& + h\mu\, \sum_{k=1}^\infty \left( \mathbf{b}^\E\,^T\mathbf{c}^\E\,^k - \mathbf{b}^\I\,^T\mathbf{c}^\I\,^k \right) \frac{h^k }{k!}y^{(k)}(t_{n})
%\end{eqnarray*}

Consider the global errors
\begin{eqnarray*}
&& e_n = y_n - y(t_n)~, \quad E_Y = Y - y(t_n+\mathbf{c}^\E h), \quad E_Z = Z - y(t_n+\mathbf{c}^\I h)\,.
\end{eqnarray*}
Relations for these errors are obtained by subtracting \eqref{eqn:IMEX-on-SL-exact} from \eqref{eqn:IMEX-on-SL}:
\begin{eqnarray}
\nonumber
E_Y &=& e_{n} \, \one + h  \,  \mathbf{A}^\EE \, \Delta f^\E  + h  \, \mu \, \mathbf{A}^\EI \,  E_Z + \delta^\E\,, \\
 \label{eqn:IMEX-on-SL-error}
 E_Z  &=& e_{n}\, \one  + h  \,  \mathbf{A}^\IE \, \, \Delta f^\E + h  \, \mu \, \mathbf{A}^\II \, E_Z + \delta^\I \,, \\
 \nonumber
e_{n+1} &=& e_n + h \, \mathbf{b}^\E\,^T \,\Delta f^\E + h   \, \mu \, \mathbf{b}^\I\,^T \, E_Z + \delta \,.
\end{eqnarray}
Using the mean function theorem,
\begin{eqnarray*}
\Delta f^\E &=&  f(Y)  -  f(y(t+\mathbf{c}^\E h)) = \lambda\, E_Y\quad \textnormal{where} ~~\lambda = f_y(\cdot)
\end{eqnarray*}
%
%\textcolor{red}{This material is not needed in the following, and can be skipt - Michael.
%\begin{eqnarray*}
%E_Y &=& e_{n} \, \one + h\lambda  \,  \mathbf{A}^\EE \, E_Y  + h  \, \mu \, \mathbf{A}^\EI \,  E_Z + \delta^\E\,, \\
%&=& S^\E \,\one\, e_{n} + h  \, \mu \, S^\E \,\mathbf{A}^\EI \,  E_Z + S^\E \, \delta^\E\,, \\
%S^\E &=& \left( \mathbf{I} - h\lambda  \,  \mathbf{A}^\EE \right)^{-1} \\
%S^\I &=& \left( \mathbf{I} - h\mu  \,  \mathbf{A}^\II \right)^{-1} \\
%E_Z  &=& S^\I\,\one \, e_{n} + h\lambda  \, S^\I\, \mathbf{A}^\IE \, E_Y + S^\I\, \delta^\I \,, \\
%&=& e_{n}\, S^\I\,\one   + S^\I\, \delta^\I  + h\lambda  \, S^\I\, \mathbf{A}^\IE \, S^\E \,\one\, e_{n} \\
%&&  + (h   \mu)\, (h\lambda)  \, S^\I\, \mathbf{A}^\IE \, S^\E \,\mathbf{A}^\EI \,  E_Z + h\lambda  \, S^\I\, \mathbf{A}^\IE \, S^\E \, \delta^\E\,,  \\
%e_{n+1} &=& e_n + h\lambda \, \mathbf{b}^\E\,^T \,E_Y + h   \, \mu \, \mathbf{b}^\I\,^T \, E_Z + \delta \,.
%\end{eqnarray*}
%}
%
%
the error equations \eqref{eqn:IMEX-on-SL-error} become
\begin{subequations}
\begin{eqnarray}
 \label{eqn:IMEX-on-SL-error-a}
\begin{bmatrix}
\mathbf{I}- h \lambda \, \mathbf{A}^\EE  & - h  \, \mu \, \mathbf{A}^\EI \\
-  h \lambda \,  \mathbf{A}^\IE & \mathbf{I}- h\mu \, \mathbf{A}^\II 
\end{bmatrix}
\,
\begin{bmatrix}
E_Y \\ E_Z
\end{bmatrix}
&=&
e_{n} \, \one + \begin{bmatrix} \delta^\E \\ \delta^\I  \end{bmatrix} \,, 
\end{eqnarray}
\begin{eqnarray}
\nonumber
e_{n+1} &=& 
\left( 1 + 
\begin{bmatrix} h \lambda\, \mathbf{b}^\E\,^T ~~  h   \, \mu \, \mathbf{b}^\I\,^T \end{bmatrix}
\begin{bmatrix}
\mathbf{I}- h \lambda \, \mathbf{A}^\EE  & - h  \, \mu \, \mathbf{A}^\EI \\
-  h \lambda \,  \mathbf{A}^\IE & \mathbf{I}- h\mu \, \mathbf{A}^\II 
\end{bmatrix}^{-1} \, \one \right)\, e_n \\
 \label{eqn:IMEX-on-SL-error-b}
&& + 
\begin{bmatrix} h \lambda\, \mathbf{b}^\E\,^T ~~  h   \, \mu \, \mathbf{b}^\I\,^T \end{bmatrix}
\begin{bmatrix}
\mathbf{I}- h \lambda \, \mathbf{A}^\EE  & - h  \, \mu \, \mathbf{A}^\EI \\
-  h \lambda \,  \mathbf{A}^\IE & \mathbf{I}- h\mu \, \mathbf{A}^\II 
\end{bmatrix}^{-1}  \begin{bmatrix} \delta^\E \\ \delta^\I  \end{bmatrix} 
 + \delta \,.
\end{eqnarray}
\end{subequations}
Inserting  \eqref{eqn:IMEX-on-SL-error-b} into  \eqref{eqn:IMEX-on-SL-error-b} and rescaling gives 
\begin{eqnarray*}
e_{n+1} &=& 
\left( 1 + 
\begin{bmatrix} h\lambda \mathbf{b}^\E\,^T ~~   \mathbf{b}^\I\,^T \end{bmatrix}
\begin{bmatrix}
 \mathbf{I} -  h\lambda\mathbf{A}^\EE  & - \mathbf{A}^\EI \\
-  h\lambda\mathbf{A}^\IE & \frac{1}{h\mu}  \mathbf{I}-  \mathbf{A}^\II 
\end{bmatrix}^{-1} \, \one \right)\, e_n \\
&& + 
\begin{bmatrix} h\lambda \mathbf{b}^\E\,^T ~~   \mathbf{b}^\I\,^T \end{bmatrix}
\begin{bmatrix}
 \mathbf{I} -  h\lambda\mathbf{A}^\EE  & - \mathbf{A}^\EI \\
-  h\lambda\mathbf{A}^\IE & \frac{1}{h\mu}  \mathbf{I}-  \mathbf{A}^\II 
\end{bmatrix}^{-1} \, \,  \begin{bmatrix} \delta^\E \\ \delta^\I  \end{bmatrix} 
 + \delta \,,
\end{eqnarray*}
which for $h\mu \to -\infty$ gives the following recurrence for the global error:
\begin{eqnarray}
\label{eqn:global-error-recurrence}
e_{n+1} &=& 
\left( 1 + 
\begin{bmatrix} h\lambda \mathbf{b}^\E\,^T ~~   \mathbf{b}^\I\,^T \end{bmatrix}
\begin{bmatrix}
 \mathbf{I} -  h\lambda\mathbf{A}^\EE  & - \mathbf{A}^\EI \\
-  h\lambda\mathbf{A}^\IE &  -  \mathbf{A}^\II 
\end{bmatrix}^{-1} \, \one \right)\, e_n \\
\nonumber
&& + 
\begin{bmatrix} h\lambda \mathbf{b}^\E\,^T ~~   \mathbf{b}^\I\,^T \end{bmatrix}
\begin{bmatrix}
 \mathbf{I} -  h\lambda\mathbf{A}^\EE  & - \mathbf{A}^\EI \\
-  h\lambda\mathbf{A}^\IE &  -  \mathbf{A}^\II 
\end{bmatrix}^{-1} \, \,  \begin{bmatrix} \delta^\E \\ \delta^\I  \end{bmatrix} 
 + \delta \,.
\end{eqnarray}
%
%
%The matrices are
%%
%\begin{eqnarray*}
%%
% \begin{bmatrix}
% %
% \mathbf{I} -  h\lambda\mathbf{A}^\EE  & - \mathbf{A}^\EI \\
%-  h\lambda\mathbf{A}^\IE &  -  \mathbf{A}^\II 
%\end{bmatrix}^{-1}
%&=&  
%%
%\left( \begin{bmatrix}
% \mathbf{I}  & - \mathbf{A}^\EI \\
%0 &  -  \mathbf{A}^\II 
%\end{bmatrix}
%-  h\lambda
%\begin{bmatrix}
% \mathbf{A}^\EE  & 0 \\
% \mathbf{A}^\IE & 0 
%\end{bmatrix}
%\right)^{-1} \\
%&=&
%\begin{bmatrix}
% \mathbf{I}  & - \mathbf{A}^\EI \\
%0 &  -  \mathbf{A}^\II 
%\end{bmatrix}^{-1} \;
%\left( \mathbf{I}
%-  h\lambda\;
%\begin{bmatrix}
% \mathbf{A}^\EE  & 0 \\
% \mathbf{A}^\IE & 0 
%\end{bmatrix}
%\;
% \begin{bmatrix}
% \mathbf{I}  & - \mathbf{A}^\EI \\
%0 &  -  \mathbf{A}^\II 
%\end{bmatrix}^{-1} \right)^{-1} 
%%
% \\
%%
%&=&  \begin{bmatrix}
% \mathbf{I}  & ~~- \mathbf{A}^\EI\, \mathbf{A}^\II\,^{-1}  \\
%0 & ~~ -  \mathbf{A}^\II\,^{-1} 
%\end{bmatrix} \;
%\left( \mathbf{I} + \mathcal{O}(h) \right)
%%
%\end{eqnarray*}
%
%and
%
We have that
\begin{eqnarray*}
\lefteqn{  \mathbf{b}^\I\,^T  \, \left(\mathbf{A}^\II\right)^{-1}  \, \delta^\I } \\
&=& \sum_{k=1}^\infty \left(\mathbf{b}^\I\,^T  \, \mathbf{A}^\II\,^{-1} \mathbf{c}^\I\,^k - k  \mathbf{b}^\I\,^T  \,  \mathbf{c}^\I\,^{k-1} \right)\, \frac{h^k  }{k!}y^{(k)}(t_{n}) \\
&& + \sum_{k=0}^\infty \left( \mathbf{b}^\I\,^T  \mathbf{c}^\I\,^k - \mathbf{b}^\I\,^T  \, \mathbf{A}^\II\,^{-1} \mathbf{A}^\IE \mathbf{c}^\E\,^k \right) \frac{h^{k+1} }{k!}f^{(k)}(t_{n}) \\
&=& \sum_{k=1}^p \left(\mathbf{b}^\I\,^T  \, \mathbf{A}^\II\,^{-1} \mathbf{c}^\I\,^k - 1 \right)\, \frac{h^k  }{k!}y^{(k)}(t_{n}) \\
&& + \sum_{k=1}^{p} \left(  \frac{1}{k}- \mathbf{b}^\I\,^T  \, \mathbf{A}^\II\,^{-1} \mathbf{A}^\IE \mathbf{c}^\E\,^{k-1} \right) \frac{h^k }{(k-1)!}f^{(k-1)}(t_{n}) + \mathcal{O}\left(h^{p+1} \right) \\
&=&  \mathcal{O}\left(h^{\min(p+1,q+1)} \right)
\end{eqnarray*}
where the last equality follows from the additional order conditions \eqref{eqn:SL-order-conditions}.

Since
\begin{eqnarray*}
&& \begin{bmatrix} h\lambda \mathbf{b}^\E\,^T ~~   \mathbf{b}^\I\,^T \end{bmatrix}
 \cdot
 \begin{bmatrix}
 \mathbf{I} -  h\lambda\mathbf{A}^\EE  & - \mathbf{A}^\EI \\
-  h\lambda\mathbf{A}^\IE &  -  \mathbf{A}^\II 
\end{bmatrix}^{-1} \\
&&\quad =
\bigl( \mathbf{I} + \mathcal{O}(h) \bigr) \; \begin{bmatrix}
0 & ~~ -  \mathbf{b}^\I\,^T\, \mathbf{A}^\II\,^{-1} 
\end{bmatrix} \,,
\end{eqnarray*}
 the global error recurrence \eqref{eqn:global-error-recurrence} becomes
\[
e_{n+1} = \left( R^\II(- \infty) + \mathcal{O}(h) \right) \, e_n  + \mathcal{O}\left(h^{\min(p+1,q+1)} \right)
\]
which gives the desired result in the limit $h \to 0$.
%
%\qed
\end{proof}

%%%%%%%%%%%%%%%%%%%%%%%%%%%%
\subsection{Stiffly accurate GARK methods}
%%%%%%%%%%%%%%%%%%%%%%%%%%%%

The following extension of the stiff accuracy concept \cite{Hairer_book_II} offers a convenient way to satisfy 
the additional order conditions \eqref{eqn:PR-order-conditions} and \eqref{eqn:SL-order-conditions}.

\begin{definition}[Stiffly accurate GARK methods.]
A GARK method \eqref{eqn:general-Butcher-tableau} is {\em stiffly accurate} if 
\begin{equation}
\label{eqn:stiff-accuracy}
\mathbf{e}_{s^{\{N\}}}^T\mathbf{A}^{\{N,m\}} =  \left( \mathbf{b}^{\{m\}} \right)^T\,, \quad m = 1,\dots,N\,. 
\end{equation}
\end{definition}
Note that stiff stability can be formulated with respect to any component method $q$ by 
replacing $N$ with $q$ in \eqref{eqn:stiff-accuracy}.

%We have the following result.
%%
%\begin{theorem}[PR convergence of stiffly accurate IMEX-GARK methods]
%\label{thm:PR-convergence}
%A stiffly accurate IMEX-GARK method \eqref{eqn:imexRK} of order $p$ is PR-convergent of order $p$.
%%
%\end{theorem}
%
%\begin{proof}
%
A stiffly accurate IMEX-GARK satisfies
\[
 \mathbf{e}_{s^\I}^T \, \mathbf{A}^\IE =  \left( \mathbf{b}^\E \right)^T\,, \quad
 \mathbf{e}_{s^\I}^T \, \mathbf{A}^\II =  \left( \mathbf{b}^\I \right)^T\,, \quad
 R^\II(\infty) = 0\,,
\]
and consequently
\[
\mathbf{b}^\I\,^T \,  \left( \mathbf{A}^\II \right)^{-1} =  \mathbf{e}_{s^\I}^T \,, \quad
 %\mathbf{c}^\IE_{s^\I} = 
\mathbf{c}^\IE_{s^\I} = 1\,.
\]
The Prothero-Robinson order conditions \eqref{eqn:PR-order-conditions} are equivalent to
\[
k\,  \mathbf{e}_{s^\I}^T \,  \mathbf{A}^\IE\, \left(\mathbf{c}^\EE\right)^{k-1} =   \mathbf{e}_{s^\I}^T \,\, \left(\mathbf{c}^\IE\right)^k
 \,, \quad k=1,\dots,q\,,
\]
and therefore to
\[
\left( \mathbf{b}^\E \right)^T \, \left(\mathbf{c}^\EE\right)^{k-1} =  \frac{1}{k} \,, \quad k=1,\dots,q\,.
\]
The conditions are automatically satisfied for $q=p$ as they are part of the explicit
component order conditions.
%
%%\qed
%\end{proof}

For a stiffly accurate IMEX method applied to Prothero-Robinson with implicit time \eqref{eqn:PR-implicit-time} 
the order conditions \eqref{eqn:alternatice-PR-conditions} are equivalent to
\[
\mathbf{b}^\E\,^T \,  \left(\mathbf{c}^\EI\right)^{k-1} = \frac{1}{k} \,,
\]
and is satisfied automatically for $k=1,\dots,p$ due to the IMEX coupling conditions of order $p$.
Thus a stiffly accurate method is PR-convergent with order $p+1$ regardless of the form of the test problem \eqref{Prothero-Robinson}.

For a stiffly accurate IMEX-GARK the semi-linear oder conditions \eqref{eqn:SL-order-conditions} read
\begin{eqnarray*}
\mathbf{e}_{s^\I}^T \,  \mathbf{c}^\I\,^k & =& 1 \,,  \\
\mathbf{e}_{s^\I}^T \,   \mathbf{A}^\IE \mathbf{c}^\E\,^{k-1}  &=& \mathbf{b}^\E\,^T \,   \mathbf{c}^\E\,^{k-1} 
= \frac{1}{k}
\,, 
\end{eqnarray*}
and are satisfied automatically through the explicit order conditions.

%%%%%%%%%%%%%%%%%%%%%%%%
\section{Stability and monotonicity}\label{sec:stability}
%%%%%%%%%%%%%%%%%%%%%%%%
In this section a stability and monotonicity analysis is performed. We derive a linear stability theory, as well as nonlinear stability theories for both dispersive and coercive problems.

%%%%%%%%%%%%%%%%%%%%%%%%
\subsection{Linear stability analysis}
%%%%%%%%%%%%%%%%%%%%%%%%
%
We apply the GARK scheme \eqref{eqn:GARK} to the linear scalar test problem
\[
y' = \sum_{m=1}^N\, \lambda^{\{m\}}\, y\,.
\]
With $z^{\{m\}}=h\, \lambda^{\{m\}}$ one obtains
\begin{eqnarray*}
Y_i^{\{q\}} &=& y_{n} + \sum_{m=1}^N \sum_{j=1}^{s^{\{m\}}} a_{i,j}^{\{q,m\}} \, z^{\{m\}}\, Y_j^{\{m\}}, \quad q=1,\ldots,N  \\
y_{n+1} &=& y_n +  \sum_{q=1}^N\, \sum_{i=1}^{s^{\{q\}}} b_{i}^{\{q\}} \, z^{\{q\}}\, Y_i^{\{q\}} \,.
\end{eqnarray*}

Denote
\begin{equation}
\label{eqn:tilde-notation}
s = \sum_{q=1}^N s^{\{q\}}\,, \quad
\mathbf{A} = \begin{bmatrix}
 \mathbf{A}^{\{1,1\}} & \dots & \mathbf{A}^{\{1,N\}}  \\
&\ddots& \\
\mathbf{A}^{\{N,1\}}  & \dots & \mathbf{A}^{\{N,N\}}
\end{bmatrix}\,, \quad
\mathbf{b} = \begin{bmatrix}
 \mathbf{b}^{\{1\}}  \\
\vdots \\
\mathbf{b}^{\{N\}}  
\end{bmatrix}\,, 
\end{equation}
and
\[
z = \begin{bmatrix} z^{\{1\}} \\ \vdots \\ z^{\{N\}}  \end{bmatrix}\,, \quad
Z = \begin{bmatrix}
z^{\{1\}}\,\mathbf{I}_{ s^{\{1\}} \times s^{\{1\}}} & \dots & \mathbf{0} \\
&\ddots& \\
\mathbf{0} & \dots & z^{\{N\}}\,\mathbf{I}_{ s^{\{N\}} \times s^{\{N\}}}
\end{bmatrix}\,.
\]
%
%By direct manipulations
%%
%\begin{eqnarray*}
%&& \begin{bmatrix}
%\mathbf{I} -  z^{\{1\}} \mathbf{A}^{\{1,1\}} & \dots & -  z^{\{N\}} \mathbf{A}^{\{1,N\}}  \\
%&\ddots& \\
% -  z^{\{1\}} \mathbf{A}^{\{N,1\}}  & \dots & \mathbf{I} -  z^{\{N\}} \mathbf{A}^{\{N,N\}}
%\end{bmatrix}
%\begin{bmatrix}
%Y^{\{1\}}  \\
%\vdots \\
%Y^{\{N\}}
%\end{bmatrix}
%=
%y_n \, \begin{bmatrix}
%\one  \\
%\vdots \\
%\one \end{bmatrix} \\
%&& y_{n+1} = y_n + \left[  z^{\{1\}}\, (b^{\{1\}})^T \dots  z^{\{N\}}\, (b^{\{N\}})^T \right] \cdot \begin{bmatrix}
%Y^{\{1\}}  \\
%\vdots \\
%Y^{\{N\}}
%\end{bmatrix}
%\end{eqnarray*}
%%
%which leads to the following stability function:
%%
%\begin{eqnarray*}
%R( z^{\{1\}}, \dots, z^{\{N\}}) &=& 1 + \left[  z^{\{1\}}\, (b^{\{1\}})^T \dots  z^{\{N\}}\, (b^{\{N\}})^T \right] \cdot \\
%&& \cdot \begin{bmatrix}
%\mathbf{I} -  z^{\{1\}} \mathbf{A}^{\{1,1\}} & \dots & -  z^{\{N\}} \mathbf{A}^{\{1,N\}}  \\
%&\ddots& \\
% -  z^{\{1\}} \mathbf{A}^{\{N,1\}}  & \dots & \mathbf{I} -  z^{\{N\}} \mathbf{A}^{\{N,N\}}
%\end{bmatrix}^{-1}\cdot
%\begin{bmatrix}
%\one  \\
%\vdots \\
%\one
%\end{bmatrix}\,.
%\end{eqnarray*}
%
We have that
\[
y_{n+1} = R\left( z^{\{1\}}, \dots, z^{\{N\}}\right) \, y_n\,,
\]
where the stability function can be written compactly as
\begin{eqnarray}
\label{eqn:stability-function}
R( z) &=& 1 + \mathbf{b}^T \cdot Z   \cdot \left( \mathbf{I}_{s\times s} -   \mathbf{A}\, Z \right)^{-1}\cdot \one_{s\times 1} \,.
\end{eqnarray}

\begin{example}[Linear stability of stiffly accurate GARK methods]
Consider a stiffly accurate GARK method \eqref{eqn:stiff-accuracy} and assume that
$Re(z_i)<0$ for $i=1,\dots,N$.  We have 
\begin{eqnarray*}
R(z) &=& 1 + \mathbf{e}_s^T \cdot \mathbf{A}\, Z   \cdot \left( \mathbf{I}_{s\times s} -   \mathbf{A}\, Z \right)^{-1}\cdot \one_{s\times 1} \\
 &=&  \mathbf{e}_s^T \cdot \left( \mathbf{I}_{s\times s} -   \mathbf{A}\, Z \right)^{-1}\cdot \one_{s\times 1} \\
 &=&  \mathbf{e}_s^T \, Z^{-1} \cdot \left(Z^{-1} -   \mathbf{A} \right)^{-1}\cdot \one_{s\times 1} \\
 &=&  \frac{1}{z^{\{_N\}}} \, \mathbf{e}_s^T \cdot \left(Z^{-1} -   \mathbf{A} \right)^{-1}\cdot \one_{s\times 1}\,.
\end{eqnarray*}
Assume that the system integrated with the stiffly accurate component is very stiff, $z^{\{N\}} \to -\infty$. Then the {\em entire}
GARK stability function becomes zero, $R(z) \to 0$.
\end{example}

%%%%%%%%%%%%%%%%%%%%%%%%
\subsection{Nonlinear stability analysis}
%%%%%%%%%%%%%%%%%%%%%%%%

We now study the nonlinear stability of GARK methods \eqref{eqn:GARK}  applied to
partitioned systems \eqref{eqn:additive-ode} where each of the component functions is dispersive
with respect to the same scalar product $\langle  \cdot, \cdot \rangle$:
\begin{equation}
\label{eqn:dispersive-condition}
\left\langle  f^{\{m\}} (y)- f^{\{m\}} (z)\,,\, y-z \right\rangle \le \nu^{\{m\}}\, \left\Vert y-z  \right\Vert^2\,, \quad \nu^{\{m\}} < 0\,. 
\end{equation}
Consider two solutions $y(t)$ and $\widetilde{y}(t)$  of \eqref{eqn:additive-ode}, each starting from a different initial condition.
Equation \eqref{eqn:dispersive-condition} implies that
\[
\left\langle  f (y)- f (z)\,,\, y-z \right\rangle \le \left( \sum_{m=1}^N \nu^{\{m\}} \right)\, \left\Vert y-z  \right\Vert^2\,, \quad \sum_{m=1}^N \nu^{\{m\}} < 0\,, 
\]
and consequently the norm of the solution difference $\Delta y(t) = \widetilde{y}(t)- y(t)$ is non-increasing, 
$\lim_{\varepsilon>0, \varepsilon \to 0} \norm{\Delta y(t+\varepsilon)} \le \norm{\Delta y(t)}$.
It is desirable that the difference of the corresponding numerical solutions 
is also non-increasing, $\norm{ \Delta y_{n+1} } \le \norm{ \Delta y_{n} }$. The analysis is carried out in the norm associated with the scalar product 
in \eqref{eqn:dispersive-condition}.

Several matrices are defined from the coefficients of  \eqref{eqn:GARK} for $m,\ell = 1,\dots,N$:
\begin{eqnarray}
\label{eqn:B-matrix}
\mathbf{B}^{\{m\}} &=& \mbox{diag}\left(\mathbf{b}^{\{m\}}\right) \,, \\
\label{eqn:P-matrix}
\mathbf{P}^{\{m,\ell\}} &=& \left(\mathbf{A}^{\{\ell,m\}}\right)^T  \mathbf{B}^{\{\ell\}} + \mathbf{B}^{\{m\}} \mathbf{A}^{\{m,\ell\}} -  \mathbf{b}^{\{m\}} \left(\mathbf{b}^{\{\ell\}}\right)^T\,, 
\end{eqnarray}
with $\mathbf{B}^{\{m\}} \in \Re^{s^{\{m\}} \times s^{\{m\}}}$, $\mathbf{P}^{\{m,\ell\}} \in \Re^{s^{\{m\}} \times s^{\{\ell\}}}$, and
\[
\mathbf{P}^{\{\ell,m\}} = \left(\mathbf{P}^{\{m,\ell\}}\right)^T \,.
\]
The following definition and analysis generalize the ones in  \cite{Higueras_2005_monotonicity}.

\begin{definition}[Algebraicaly stable GARK methods]
A generalized additive Runge-Kutta method \eqref{eqn:GARK} is algebraically stable if the weight vectors are non-negative
\begin{subequations}
\begin{equation}
\label{eqn:algebraic-stability-b}
b_{i}^{\{m\}} \ge 0 \quad \textnormal{for all  } i=1,\dots,s^{\{m\}}\,, \quad m=1,\dots,N\,,
\end{equation}
and the following matrix is non-negative definite:
\begin{equation}
\label{eqn:algebraic-stability-P}
\mathbf{P} = \begin{bmatrix}
\mathbf{P}^{\{\ell,m\}}\end{bmatrix}_{1 \le \ell,m \le N}
 \in \Re^{s \times s}\,, \quad \mathbf{P} \ge 0\,.
\end{equation}
\end{subequations}
\end{definition}

We have the following result.

\begin{theorem}[Algebraic stability of GARK methods]
An algebraically stable GARK method \eqref{eqn:GARK} applied to a 
partitioned system \eqref{eqn:additive-ode} with dispersive component functions
\eqref{eqn:dispersive-condition} is unconditionally nonlinearly stable, in the sense that
the difference of any two numerical solutions in non-increasing
\[
\norm{\Delta y_{n+1}} \le \norm{\Delta y_{n}}
\]
for any step size $h>0$.
\end{theorem}

\begin{proof}
The difference  between solutions $ \Delta y_n = \widetilde{y}_n - y_n$ advances in time as follows:
\begin{subequations}
\begin{eqnarray}
\label{eqn:delta-stage2-diff}
\Delta Y_i^{\{q\}} &=& \Delta y_{n} + h \sum_{m=1}^N \sum_{j=1}^{s^{\{m\}}} a_{i,j}^{\{q,m\}} \, \Delta f^{\{m\}}_j, \quad q=1,\ldots,N\,,  \\
\label{eqn:delta-sol2-diff}
\Delta y_{n+1} &=& \Delta y_n + h \sum_{q=1}^N\, \sum_{i=1}^{s^{\{q\}}} b_{i}^{\{q\}} \, \Delta f^{\{q\}}_i \,.
\end{eqnarray}
\end{subequations}
where
\[
\Delta Y^{\{m\}}_j = \widetilde Y_j^{\{m\}} - Y_j^{\{m\}}\,, \quad
\Delta f^{\{m\}}_j = f^{\{m\}}\left(\widetilde Y_j^{\{m\}}\right) - f^{\{m\}}\left(Y_j^{\{m\}}\right).
\]
From \eqref{eqn:delta-sol2-diff} we get 
\begin{eqnarray}
\label{eqn:tmp1}
\norm{\Delta y_{n+1}}^2 &=& \norm{\Delta y_{n}}^2 \\
\nonumber
&& + 2\, h\,  \sum_{m=1}^N\, \sum_{i=1}^{s^{\{m\}}} b_{i}^{\{m\}} \, \scalar{ \Delta f^{\{m\}}_i \, , \, \Delta y_{n} } \\
\nonumber
&& + h^2 \, \sum_{m,q=1}^N\, \sum_{i=1}^{s^{\{m\}}} \sum_{j=1}^{s^{\{q\}}}  b_{i}^{\{m\}} \,  b_{j}^{\{q\}} \, \scalar{ \Delta f^{\{m\}}_i \,,\, \Delta f^{\{q\}}_j }\,.
\end{eqnarray}
From \eqref{eqn:delta-stage2-diff} it follows that
\begin{eqnarray}
\label{eqn:tmp2}
 \Delta y_{n} &=&\Delta Y_i^{\{q\}}  -  h \, \sum_{m=1}^N\, \sum_{j=1}^{s^{\{m\}}} a_{i,j}^{\{q,m\}} \, \Delta f^{\{m\}}_j \,.
 \end{eqnarray}
Substituting \eqref{eqn:tmp2} into \eqref{eqn:tmp1} leads to
\begin{eqnarray}
\label{eqn:tmp3}
\norm{\Delta y_{n+1}}^2 &=& \norm{\Delta y_{n}}^2 \\
\nonumber
&& + 2\, h\,  \sum_{m=1}^N\, \sum_{i=1}^{s^{\{m\}}} b_{i}^{\{m\}} \, \scalar{ \Delta f^{\{m\}}_i \, , \, \Delta Y_i^{\{m\}} } \\
\nonumber
&& - 2\, h^2 \,   \sum_{m,l=1}^N\, \sum_{i=1}^{s^{\{m\}}} \sum_{j=1}^{s^{\{l\}}}  a_{i,j}^{\{m,l\}}  b_{i}^{\{m\}} \, \scalar{ \Delta f^{\{m\}}_i \, , \,\Delta f^{\{l\}}_j  } \\
\nonumber
&& + h^2 \, \sum_{m,l=1}^N\, \sum_{i=1}^{s^{\{m\}}} \sum_{j=1}^{s^{\{l\}}}   b_{i}^{\{m\}} \,  b_{j}^{\{l\}} \, \scalar{ \Delta f^{\{m\}}_i \,,\, \Delta f^{\{l\}}_j } \\
\label{eqn:tmp3b}
&=& \norm{\Delta y_{n}}^2 \\
\nonumber
&& + 2\, h\,  \sum_{m=1}^N\, \sum_{i=1}^{s^{\{m\}}}  b_{i}^{\{m\}} \, \scalar{ \Delta f^{\{m\}}_i \, , \, \Delta Y_i^{\{m\}} } \\
\nonumber
&& -  h^2 \,   \sum_{m,l=1}^N\, \sum_{i=1}^{s^{\{m\}}} \sum_{j=1}^{s^{\{l\}}}  a_{i,j}^{\{m,l\}}  b_{i}^{\{m\}} \, \scalar{ \Delta f^{\{m\}}_i \, , \,\Delta f^{\{l\}}_j  } \\
\nonumber
&& -  h^2 \,   \sum_{m,l=1}^N\, \sum_{i=1}^{s^{\{m\}}} \sum_{j=1}^{s^{\{l\}}}  a_{j,i}^{\{l,m\}}  b_{j}^{\{l\}} \, \scalar{ \Delta f^{\{m\}}_i \, , \,\Delta f^{\{l\}}_j  } \\
\nonumber
&& + h^2 \, \sum_{m,l=1}^N\, \sum_{i=1}^{s^{\{m\}}} \sum_{j=1}^{s^{\{l\}}}  b_{i}^{\{m\}} \,  b_{j}^{\{l\}} \, \scalar{ \Delta f^{\{m\}}_i \,,\, \Delta f^{\{l\}}_j }
\,.
\end{eqnarray}
Equation \eqref{eqn:tmp3b} can be written in the equivalent form
\begin{eqnarray}
\nonumber
\norm{\Delta y_{n+1}}^2 
&=& \norm{\Delta y_{n}}^2  + 2\, h\,  \sum_{m=1}^N\, \sum_{i=1}^{s^{\{m\}}} b_{i}^{\{m\}} \, \scalar{ \Delta f^{\{m\}}_i \, , \, \Delta Y_i^{\{m\}} } \\
\nonumber
&& - h^2 \,  \sum_{m,l=1}^N\, \sum_{i=1}^{s^{\{m\}}} \sum_{j=1}^{s^{\{l\}}}  \left( b_{i}^{\{m\}} a_{i,j}^{\{m,l\}}   +  b_{j}^{\{l\}}  a_{j,i}^{\{l,m\}}  -  b_{i}^{\{m\}} \,  b_{j}^{\{\l\}} \right) \, \scalar{ \Delta f^{\{m\}}_i \, , \,\Delta f^{\{l\}}_j  } \\
\label{eqn:tmp4}
&=& \norm{\Delta y_{n}}^2  + 2\, h \,  \sum_{m=1}^N\, \sum_{i=1}^{s^{\{m\}}} b_{i}^{\{m\}} \, \scalar{ \Delta f^{\{m\}}_i \, , \, \Delta Y_i^{\{m\}} } \\
\nonumber
&& - h^2 \,  \sum_{m,l=1}^N\,  (\Delta f^{\{m\}})^T\, \left( \mathbf{P}^{\{m,l\}} \otimes \mathbf{I}_{d \times d} \right) \, \Delta f^{\{l\}}
\end{eqnarray}
where
\[
\Delta f^{\{m\}} = \left[\, \bigl(\Delta f_1^{\{m\}}\bigr)^T, \dots, \bigl(\Delta f_{s^{\{m\}}}^{\{m\}}\bigr)^T \, \right]^T.
\]
From \eqref{eqn:tmp4} and the positive definiteness of $\mathbf{P}$  \eqref{eqn:algebraic-stability-P} we have that
\begin{eqnarray*}
\norm{\Delta y_{n+1}}^2 
&\le& \norm{\Delta y_{n}}^2  + 2\, h \,  \sum_{m=1}^N\, \sum_{i=1}^s b_{i}^{\{m\}} \, \scalar{ \Delta f^{\{m\}}_i \, , \, \Delta Y_i^{\{m\}} } \,.
\end{eqnarray*}
The positivity of the weights \eqref{eqn:algebraic-stability-b} and dispersion condition \eqref{eqn:dispersive-condition} give the desired result:
\begin{eqnarray*}
\norm{\Delta y_{n+1}}^2 
&\le& \norm{\Delta y_{n}}^2 + 2\, h \,  \sum_{m=1}^N\, \nu^{\{m\}} \sum_{i=1}^s b_{i}^{\{m\}} \, \norm{ \Delta Y_i^{\{m\}} }^2 \le \norm{\Delta y_{n}}^2  \,.  %\qed
\end{eqnarray*}
%
%%\qed
\end{proof}

\begin{definition}[Stability-decoupled GARK schemes]
A GARK method \eqref{eqn:GARK}  is {\em stability-decoupled} if 
\begin{equation}
\label{eqn:stability-decoupled}
\mathbf{P}^{\{m,\ell\}}
= \mathbf{0}
\quad \textnormal{for}~~m \ne \ell\,.
\end{equation}
\end{definition}
For stability decoupled GARK methods the interaction of different components does not influence the overall nonlinear stability. 
If each of the component methods is nonlinearly stable, perhaps under a suitable step size restriction, then 
the overall method is nonlinearly stable (under a step size restriction that satisfies each of the components).
In particular, if each of the component Runge-Kutta scheme is algebraically stable
\[
b_i^{\{m\}} > 0\,, \quad
\mathbf{B}^{\{m\}} = \mbox{diag}(\mathbf{b}^{\{m\}})\,, \quad
\mathbf{P}^{\{m,m\}}
\ge 0 \quad m=1,\ldots,N,
\]
then equation \eqref{eqn:tmp4} shows that \eqref{eqn:stability-decoupled} is a sufficient condition for the algebraic stability 
of the GARK scheme.

\begin{remark}
 Equation \eqref{eqn:tmp4} shows that the nonlinear stability of each of the component methods is sufficient to obtain a nonlinear stable GARK scheme in the case of component partitioning
 \eqref{eqn:component-partitioning}.
\end{remark}

%%%%%%%%%%%%%%%%%%%%%%%%
\subsection{Conditional stability for coercive problems}
%%%%%%%%%%%%%%%%%%%%%%%%

Next, consider partitioned systems \eqref{eqn:additive-ode} where each of the component functions is coercive \cite{Higueras_2005_monotonicity}:
\begin{equation}
\label{eqn:dispersive-condition-f}
\left\langle  f^{\{m\}} (y)- f^{\{m\}} (z)\,,\, y-z \right\rangle \le \mu^{\{m\}}\, \left\Vert f^{\{m\}} (y)- f^{\{m\}} (z) \right\Vert^2\,, \quad \mu^{\{m\}} < 0\,. 
\end{equation}
\begin{theorem}[Conditional stability of GARK methods]
Consider a partitioned system \eqref{eqn:additive-ode} with coercive component functions
\eqref{eqn:dispersive-condition-f}
solved by a GARK method \eqref{eqn:GARK}. Assume that there exist $r^{\{m\}} \ge 0$ such that the following matrix is positive definite
\begin{equation}
\label{eqn:conditional-stability-P}
\widetilde{\mathbf{P}} = \mathbf{P} + \mbox{diag}_{m=1,\dots,N} \left\{ r^{\{m\}} \mathbf{B}^{\{m\}} \right\} \ge 0\,,
\end{equation}
where $\mathbf{P}$ was defined in \eqref{eqn:algebraic-stability-P}. 
Then the solution is conditionally nonlinearly stable, in the sense that $\norm{\Delta y_{n+1}} \le \norm{\Delta y_{n}}$,
under the step size restriction 
\[
 h \le \min_{m=1,\dots,N} \frac{  -2\, \mu^{\{m\}} }{ r^{\{m\}} }\,.
\]
\end{theorem}

\begin{proof}
Equation \eqref{eqn:tmp4} and condition \eqref{eqn:dispersive-condition-f} yield
\begin{eqnarray*}
\norm{\Delta y_{n+1}}^2 
&\le& \norm{\Delta y_{n}}^2  + 2\, h \,  \sum_{m=1}^N\, \sum_{i=1}^s b_{i}^{\{m\}} \mu^{\{m\}}\, \left\Vert \Delta f^{\{m\}}_i\right\Vert^2 \\
\nonumber
&& - h^2 \,  \sum_{m,l=1}^N\,  (\Delta f^{\{m\}})^T\, \left( \widetilde{\mathbf{P}}^{\{m,l\}} \otimes \mathbf{I}_{d \times d} \right) \, \Delta f^{\{l\}} \\
&& + h^2 \,  \sum_{m=1}^N\,  (\Delta f^{\{m\}})^T\, \left( r^{\{m\}} \mathbf{B}^{\{m\}} \otimes \mathbf{I}_{d \times d} \right) \, \Delta f^{\{m\}}\\
&=& \norm{\Delta y_{n}}^2  + 2\, h \,  \sum_{m=1}^N\, \sum_{i=1}^s b_{i}^{\{m\}} \mu^{\{m\}}\, \left\Vert \Delta f^{\{m\}}_i\right\Vert^2 \\
\nonumber
&& - h^2 \,  \sum_{m,l=1}^N\,  (\Delta f^{\{m\}})^T\, \left( \widetilde{\mathbf{P}}^{\{m,l\}} \otimes \mathbf{I}_{d \times d} \right) \, \Delta f^{\{l\}} \\
&& + h^2 \,  \sum_{m=1}^N\,   \sum_{i=1}^s  r^{\{m\}} b_{i}^{\{m\}}  \left\Vert \Delta f^{\{m\}}_i\right\Vert^2\,.
\end{eqnarray*}
From the positive definiteness assumption \eqref{eqn:conditional-stability-P}
\begin{eqnarray*}
\norm{\Delta y_{n+1}}^2 
&\le& \norm{\Delta y_{n}}^2  + h \,  \sum_{m=1}^N\, \sum_{i=1}^s b_{i}^{\{m\}} \left(2\,  \mu^{\{m\}} + h  r^{\{m\}}  \right)\, \left\Vert \Delta f^{\{m\}}_i\right\Vert^2\,.
\end{eqnarray*}
We see that if $2\,  \mu^{\{m\}} + h  r^{\{m\}}  \le 0$ for all $m=1,\dots,N$ then $\norm{\Delta y_{n+1}}^2  \le \norm{\Delta y_{n}}^2$.
This proves the desired result, which extends the one given in \cite{Higueras_2005_monotonicity} for classical additive Runge-Kutta methods.
%\qed
\end{proof}

\begin{remark}
If the GARK method is stability decoupled \eqref{eqn:stability-decoupled} then the weights $r^{\{m\}}$ in \eqref{eqn:conditional-stability-P} are chosen independently for each 
component. In this case each component method, applied to the corresponding subsystem, is conditionally stable under a step restriction $h \le h_{\rm max}^{\{m\}}$.
The GARK method's step size restriction is given by the bounds for individual components, $h \le \min_{m=1,\dots,N} h_{\rm max}^{\{m\}}$, i.e.,
no additional stability restrictions are imposed on the step size.
\end{remark}

%%%%%%%%%%%%%%%%%%%
\begin{example}[A second order,  stability-decoupled IMEX-GARK scheme]
%%%%%%%%%%%%%%%%%%%

We construct a second order IMEX-GARK method where the implicit and explicit parts have different numbers of stages. The method has a free parameter denoted $\beta$.
The implicit method
\renewcommand{\arraystretch}{1.25}
\[
\mathbf{A}^\II  = \left(\begin{array}{cc} \frac{1}{4} & 0\\ \frac{1}{2} & \frac{1}{4} \end{array}\right)\,, \quad
\mathbf{b}^\I  =\left(\begin{array}{c} \frac{1}{2}\\ \frac{1}{2} \end{array}\right)\,, \quad
\mathbf{c}^\II  = \left(\begin{array}{c} \frac{1}{4}\\ \frac{3}{4} \end{array}\right)\,,
\]
\renewcommand{\arraystretch}{1}
is second order accurate and algebraically stable since
\[
\mathbf{P}^\II = \mathbf{A}^\II\,^T  \mathbf{B}^\I + \mathbf{B}^\I \mathbf{A}^\II -  \mathbf{b}^\I \mathbf{b}^\I\,^T = 0\,.
\]
The explicit method is:
\renewcommand{\arraystretch}{1.25}
\[
\mathbf{A}^\EE  = \left[\begin{array}{ccc} 0 & 0 & 0\\ \frac{1}{2} & 0 & 0\\ 1 - \beta & \beta & 0 \end{array}\right]\,, \quad
\mathbf{b}^\E  =  \left[\begin{array}{c} \frac{1}{4}\\ \frac{1}{2}\\ \frac{1}{4} \end{array}\right] \,, \quad
\mathbf{c}^\EE =   \left[\begin{array}{c} 0\\ \frac{1}{2}\\ 1 \end{array}\right]
  \,.
\]
The explicit method is conditionally stable for coercive problems.
A good value of the free parameter for stability is $\beta = -1/4$ for which \eqref{eqn:conditional-stability-P}
holds with $r\approx 2.6$.
The coupling coefficients are
\[
\mathbf{A}^\EI  = \left[\begin{array}{cc} 0 & 0\\ \frac{1}{2} & 0\\ \frac{1}{2} & \frac{1}{2} \end{array}\right]\,, \quad
\mathbf{c}^\EI  = \mathbf{c}^\EE  = \left[\begin{array}{c} 0\\ \frac{1}{2}\\ 1 \end{array}\right]\,,
\]
and
\[
\mathbf{A}^\IE  = \left[\begin{array}{ccc} \frac{1}{4} & 0 & 0\\ \frac{1}{4} & \frac{1}{2} & 0 \end{array}\right]\,, \quad
\mathbf{c}^\IE = \mathbf{c}^\II = \left[\begin{array}{c} \frac{1}{4}\\ \frac{3}{4} \end{array}\right]\,.
\]
\renewcommand{\arraystretch}{1}
The IMEX-GARK method is stability-decoupled
\[
\mathbf{P}^\EI = \left[\mathbf{A}^\IE\right]^T  \mathbf{B}^\I + \mathbf{B}^\E \mathbf{A}^\EI -  \mathbf{b}^\E \left[\mathbf{b}^\I\right]^T = \mathbf{0}\,.
\]
This property, and the algebraic stability of the implicit part, imply that the GARK method is nonlinearly stable under the same step size restriction  for which the explicit component
is nonlinearly stable (e.g.,  for $\beta = -1/4$ we have  $h \le -2\, \mu^\E/ 2.6$).

This IMEX-GARK scheme is represented compactly by its generalized Butcher tableau \eqref{eqn:generalized-tableau} as:
\renewcommand{\arraystretch}{1.25}
\[
\begin{array}{c|c|c|c}
\mathbf{c}^\EE & \mathbf{A}^\EE &\mathbf{A}^\EI & \mathbf{c}^\EI  \\
\hline 
\mathbf{c}^\IE & \mathbf{A}^\IE & \mathbf{A}^\II  & \mathbf{c}^\II \\
\hline
& \mathbf{b}^\E & \mathbf{b}^\I 
\end{array}
~~:=~~
\begin{array}{c|ccc|cc|c} 0 ~&~ 0 ~&~ 0 ~&~ 0 ~&~ 0 ~&~ 0 ~&~ 0\\ 
\frac{1}{2} ~&~ \frac{1}{2} ~&~ 0 ~&~ 0 ~&~ \frac{1}{2} ~&~ 0 ~&~ \frac{1}{2}\\  
1 ~&~ 1 - \beta ~&~ \beta ~&~ 0 ~&~ \frac{1}{2} ~&~ \frac{1}{2} ~&~ 1\\ 
\hline
\frac{1}{4} ~&~ \frac{1}{4} ~&~ 0 ~&~ 0 ~&~ \frac{1}{4} ~&~ 0 ~&~ \frac{1}{4}\\ \frac{3}{4} ~&~ \frac{1}{4} ~&~ \frac{1}{2} ~&~ 0 ~&~ \frac{1}{2} ~&~ \frac{1}{4} ~&~ \frac{3}{4}\\ 
\hline
 ~&~ \frac{1}{4} ~&~ \frac{1}{2} ~&~ \frac{1}{4} ~&~ \frac{1}{2} ~&~ \frac{1}{2} ~&~  \end{array}~.
\]
\end{example}
\renewcommand{\arraystretch}{1}

%%%%%%%%%%%%%%%%%%%%%%%%
\subsection{Monotonicity analysis}
%%%%%%%%%%%%%%%%%%%%%%%%

This section studies the contractivity and monotonicity of the generalized additively partitioned Runge-Kutta methods.
We are concerned with partitioned systems \eqref{eqn:additive-ode} where there exist $\rho^{\{1\}}, \dots, \rho^{\{N\}} > 0$ such that 
\begin{equation} 
\label{eqn:monotone-ode}
\forall \, y ~: \quad
\left\Vert y + \rho^{\{m\}}\, f^{\{m\}} (y) \right\Vert \le \left\Vert y  \right\Vert\,, \quad m = 1,\dots,N\,.
\end{equation}
This implies that condition \eqref{eqn:monotone-ode} holds for any $0 \le \tau^{\{m\}} \le \rho^{\{m\}}$, i.e.,  for each individual subsystem $m$ the solution of one forward Euler step 
is monotone under this step size restriction. The condition \eqref{eqn:monotone-ode} also implies that
the system \eqref{eqn:additive-ode} has a solution of non increasing norm.
To see this write an Euler step with the full system as a convex combination
\begin{eqnarray*}
\left\Vert y + \theta\, \sum_{m=1}^N f^{\{m\}} (y)  \right\Vert &=& \left\Vert \sum_{m=1}^N \frac{1}{N} \left( y + N \theta\, f^{\{m\}} (y)\right)  \right\Vert
\le \sum_{m=1}^N \frac{1}{N} \left\Vert y \right\Vert = \left\Vert y \right\Vert 
\end{eqnarray*}
if $0< \theta \le \min\{ \rho^{\{m\}}\}/N$, and consequently $\lim_{\varepsilon>0, \varepsilon \to 0} \norm{\ y(t+\varepsilon)} \le \norm{y(t)}$ \cite{Higueras_2006_SSP-ARK}.

We seek to construct GARK schemes which guarantee a monotone numerical solution $\left\Vert y_{n+1}  \right\Vert \le \left\Vert y_{n}  \right\Vert$ for \eqref{eqn:monotone-ode} under suitable step size restrictions.

A comprehensive study of contractivity of Runge-Kutta methods is given in \cite{Kraaijevanger_1991_contractivity}.
Step size conditions for monotonicity are discussed in \cite{Spijker_2007}.
Strong stability preserving methods suitable for hyperbolic PDEs are reviewed in \cite{Gottlieb_2001_SSP-review,Higueras_2004_SSP,Higueras_2005_SSP-representation}. 
Monotonicity for Runge-Kutta methods in inner product norms is discussed in \cite{Higueras_2005_monotonicity}.
This study follows the approach of Higueras and co-workers, who have extended the monotonicity theory to additive Runge-Kutta methods \cite{Higueras_2006_SSP-ARK,Higueras_2009_SSP-ARK,Garcia-Celayeta_2006_ARK-monotonicity}.

The scheme \eqref{eqn:GARK} can be represented in matrix form as
\begin{subequations}
\label{eqn:genaddRK-matrix}
\begin{eqnarray}
\label{eqn:genaddRK-stage-matrix}
Y^{\{q\}} &=& \one_{{\{q\}} \times 1} \otimes y_{n} + h \sum_{m=1}^N \left( \mathbf{A}^{\{q,m\}} \otimes \mathbf{I}_{d \times d} \right)\, f^{\{m\}}(Y^{\{m\}})\,, \\ %\quad q=1,\ldots,N  \\
\label{eqn:genaddRK-sol-matrix}
y_{n+1} &=& y_n + h \sum_{q=1}^N\,  \left(\mathbf{b}^{\{q\}}\,^T \otimes \mathbf{I}_{d \times d} \right) \, f^{\{q\}}(Y^{\{q\}}) \,.
\end{eqnarray}
\end{subequations}
Using notation \eqref{eqn:tilde-notation} and
\[
\widehat{\mathbf{A}} = \begin{bmatrix} \mathbf{A} & \mathbf{0} \\ (\mathbf{b})^T  & \mathbf{0} \end{bmatrix}\,, \quad
\widehat{Y} = \begin{bmatrix}
 Y^{\{1\}}  \\
\vdots \\
Y^{\{N\}}  \\
y_{n+1}
\end{bmatrix}\,,\quad
\widehat{f}\left(\widehat{Y}\right) = \begin{bmatrix}
f^{\{1\}}\left( Y^{\{1\}} \right) \\
\vdots \\
f^{\{N\}}\left( Y^{\{N\}}  \right) \\
0
\end{bmatrix}\,,
\]
equations \eqref{eqn:genaddRK-matrix} become
\begin{eqnarray}
\label{eqn:genaddRK-compact}
\widehat{Y} &=& \one_{(s+1) \times 1} \otimes y_{n} + h \, \left( \widehat{\mathbf{A}} \otimes \mathbf{I}_{d \times d} \right)\, \widehat{f}\left(\widehat{Y}\right)\,. 
\end{eqnarray}

The following definitions extend the corresponding ones from  \cite{Higueras_2006_SSP-ARK,Higueras_2009_SSP-ARK}.
\begin{definition}[Absolutely monotonic GARK] 
Let $r^{\{1\}},\dots,r^{\{N\}}>0$ and 
\begin{equation}
\label{eqn:R-hat}
r = \begin{bmatrix} r^{\{1\}} \\ \vdots \\ r^{\{N\}} \end{bmatrix}\,, \quad
\widehat{\mathbf{R}}=\textnormal{diag}\left\{  r^{\{1\}}\,\mathbf{I}_{s^{\{1\}} \times  s^{\{1\}}}, \dots,  r^{\{N\}}\,\mathbf{I}_{s^{\{N\}} \times  s^{\{N\}}} , 1 \right\}\,.
\end{equation}
A GARK scheme \eqref{eqn:GARK} defined by $\widehat{\mathbf{A}} \ge 0$ is called {\em absolutely monotonic} (a.m.) at $r \in \Re^N$ if 
\begin{subequations}
\label{eqn:am-conditions}
\begin{eqnarray}
\label{eqn:am-condition-e}
\alpha(r) &=&  \left(\mathbf{I}_{\hat{s} \times \hat{s}} + \widehat{\mathbf{A}}\widehat{\mathbf{R}} \right)^{-1}   \one_{\hat{s} \times 1} \ge 0\,, ~~~\textnormal{and} \\
\label{eqn:am-condition-alpha}
\beta(r) &=&  \left(\mathbf{I}_{\hat{s} \times \hat{s}} + \widehat{\mathbf{A}}\widehat{\mathbf{R}} \right)^{-1}  \widehat{\mathbf{A}}\widehat{\mathbf{R}}
=  \mathbf{I}_{\hat{s} \times \hat{s}} - \left(\mathbf{I}_{\hat{s} \times \hat{s}} + \widehat{\mathbf{A}}\widehat{\mathbf{R}} \right)^{-1} \ge 0\,,
\end{eqnarray}
\end{subequations}
where  $\hat{s}=s+1$. Here all the inequalities are taken component-wise.
\end{definition}

\begin{definition}[Region of absolute monotonicity] 
The region of absolute monotonicity of the GARK scheme \eqref{eqn:GARK}  is
\begin{equation}
\mathcal{R}(\widehat{\mathbf{A}}) = \left\{ r \in \Re^N_+~:~ \widehat{\mathbf{A}} \textnormal{ is a.m. on } \left[0,r^{\{1\}}\right] \times \dots \times \left[0,r^{\{N\}}\right] \right\}\,.
\end{equation}
\end{definition}

\begin{theorem}[Monotonicity of solutions] 
Consider the GARK scheme \eqref{eqn:GARK}  defined by $\widehat{\mathbf{A}}$ and a point in the interior of its absolute monotonicity region
\[
r \in \mathcal{R}(\widehat{\mathbf{A}})\,, \quad r > \mathbf{0}_{N \times 1}\,.
\]
For any step size obeying the restriction
\begin{equation}
\label{eqn:step-restriction-monotonicity}
h \le  \max_{q=1,\dots,N} \left\{ r^{\{q\}} \, \rho^{\{q\}} \right\}
\end{equation}
the stage values and the solution of \eqref{eqn:GARK} are monotonic
\begin{subequations}
\label{eqn:monotonicity-conclusions}
\begin{eqnarray}
\left\Vert  Y_i^{\{q\}} \right\Vert &\le& \left\Vert y_n \right\Vert\,, \quad q=1,\dots,N, ~~i=1,\dots,s^{\{q\}}\,, \\
\ \left\Vert  y_{n+1} \right\Vert &\le& \left\Vert y_n \right\Vert\,.
\end{eqnarray}
\end{subequations}
\end{theorem}
In practice we are interested in the largest upper bound for the time step that ensures monotonicity.

\begin{proof}
The proof is a direct extension of the corresponding one for classical additively partitioned Runge-Kutta methods given in \cite{Higueras_2006_SSP-ARK}.
Construct the matrix $\widehat{\mathbf{R}}$ as in \eqref{eqn:R-hat}.
Add the same quantity $\widehat{\mathbf{A}}\widehat{\mathbf{R}}  \otimes \mathbf{I}_{d \times d}$ to both sides of \eqref{eqn:genaddRK-compact} to obtain
\begin{eqnarray*}
 \left(\mathbf{I}_{d\hat{s} \times d\hat{s}} + \widehat{\mathbf{A}}\widehat{\mathbf{R}}  \otimes \mathbf{I}_{d \times d} \right) \widehat{Y} &=& \one_{\hat{s} \times 1} \otimes y_{n} \\
 && +  \left( \widehat{\mathbf{A}}\widehat{\mathbf{R}} \otimes \mathbf{I}_{d \times d} \right)\, 
\left( \, \widehat{Y} +  (h\,\widehat{\mathbf{R}}^{-1} \otimes \mathbf{I}_{d \times d}) \, \widehat{f}(\widehat{Y}) \, \right)\,.
\end{eqnarray*}
Using the notation of \eqref{eqn:am-conditions} this relation can be written in the equivalent form
\begin{equation}
\label{eqn:tmp-form-1}
\widehat{Y} =\bigl(\alpha(r) \otimes \mathbf{I}_{d \times d} \bigr) \cdot y_{n} +  \bigl( \beta(r) \otimes \mathbf{I}_{d \times d} \bigl)\, 
\left(  \widehat{Y} +  (h\,\widehat{\mathbf{R}}^{-1} \otimes \mathbf{I}_{d \times d}) \, \widehat{f}(\widehat{Y}) \right)\,.
\end{equation}
Denote a vector of norms by
\[
\left\llbracket \widehat{Y} \right\rrbracket  = 
\begin{bmatrix}
\left\Vert Y^{\{1\}}  \right\Vert , \dots, 
\left\Vert Y^{\{N\}}  \right\Vert ,
\left\Vert y_{n+1} \right\Vert
\end{bmatrix}^T\,.
\]
Since $r \in \mathcal{R}(\widehat{\mathbf{A}})$ we have that $\alpha(r) \ge 0$ and $\beta(r) \ge 0$. Taking norms in \eqref{eqn:tmp-form-1}
leads to 
\begin{eqnarray*}
\left\llbracket \widehat{Y} \right\rrbracket &\le&\alpha(r)\, \left\Vert y_{n}\right\Vert  +   \beta(r) \, 
\left\llbracket   \widehat{Y} +  (h\,\widehat{\mathbf{R}}^{-1} \otimes \mathbf{I}_{d \times d})\, \widehat{f}(\widehat{Y}) \right\rrbracket \,.
\end{eqnarray*}
Under the step size restriction \eqref{eqn:step-restriction-monotonicity} we have $h\,\left(r^{\{q\}}\right)^{-1} \le \rho^{\{q\}}$ for any $q=1,\dots,N$, and from 
\eqref{eqn:monotone-ode}
\[
\left\llbracket   \widehat{Y} +  (h\,\widehat{\mathbf{R}}^{-1} \otimes \mathbf{I}_{d \times d}) \, \widehat{f}(\widehat{Y}) \right\rrbracket \le 
\left\llbracket   \widehat{Y} \right\rrbracket \,.
\]
It follows that
\begin{eqnarray*}
\left\llbracket \widehat{Y} \right\rrbracket &\le&\alpha(r)\, \left\Vert y_{n}\right\Vert  +   \beta(r) \,  \left\llbracket   \widehat{Y}  \right\rrbracket\\
&=&
\left(\mathbf{I}_{\hat{s} \times \hat{s}} + \widehat{\mathbf{A}}\widehat{\mathbf{R}} \right)^{-1} \,  \one_{\hat{s} \times 1} \, \left\Vert y_{n}\right\Vert 
+ \left(  \mathbf{I}_{\hat{s} \times \hat{s}} - \left(\mathbf{I}_{\hat{s} \times \hat{s}} + \widehat{\mathbf{A}}\widehat{\mathbf{R}} \right)^{-1} \right)\,  \left\llbracket   \widehat{Y}  \right\rrbracket
\end{eqnarray*}
and
\begin{eqnarray*}
 \left(\mathbf{I}_{\hat{s} \times \hat{s}} + \widehat{\mathbf{A}}\widehat{\mathbf{R}} \right)^{-1}\, \left\llbracket \widehat{Y} \right\rrbracket &\le&
\left(\mathbf{I}_{\hat{s} \times \hat{s}} + \widehat{\mathbf{A}}\widehat{\mathbf{R}} \right)^{-1} \,  \one_{\hat{s} \times 1} \, \left\Vert y_{n}\right\Vert \,.
\end{eqnarray*}
Multiplication by the matrix $\mathbf{I}_{\hat{s} \times \hat{s}} + \widehat{\mathbf{A}}\widehat{\mathbf{R}} \ge 0$, whose entries are all non-negative, 
implies that
\[
\left\llbracket \widehat{Y} \right\rrbracket  \le  \one_{\hat{s} \times 1}\otimes \left\Vert y_{n}\right\Vert\,,
\]
and the monotonicity relation \eqref{eqn:monotonicity-conclusions} follows.
%\qed
\end{proof}

%%%%%%%%%%%%%%%%%%%%%%%
\begin{example}[Monotonicity of classical IMEX RK]
%%%%%%%%%%%%%%%%%%%%%%%

For $N=2$ we have
\begin{equation}
\widehat{\mathbf{A}} = 
\begin{bmatrix}
 \mathbf{A}^{\{1,1\}} &\mathbf{A}^{\{1,2\}}  & 0 \\ 
 \mathbf{A}^{\{2,1\}} & \mathbf{A}^{\{2,2\}}  & 0 \\
(\mathbf{b}^{\{1\}})^T & (\mathbf{b}^{\{2\}})^T & 0
\end{bmatrix}\,,
\end{equation}
which we write in the equivalent form
\begin{equation}
\widehat{\mathbf{A}} = 
\begin{bmatrix}
 \mathbf{A}^{\{1,1\}}  & 0&\mathbf{A}^{\{1,2\}}  & 0 \\ 
(\mathbf{b}^{\{1\}})^T  & 0& (\mathbf{b}^{\{2\}})^T & 0 \\
 \mathbf{A}^{\{2,1\}}  & 0& \mathbf{A}^{\{2,2\}}  & 0 \\
(\mathbf{b}^{\{1\}})^T  & 0& (\mathbf{b}^{\{2\}})^T & 0
\end{bmatrix}
=
\begin{bmatrix}
 \widehat{\mathbf{A}}^{\{1,1\}}  & \widehat{\mathbf{A}}^{\{1,2\}}   \\ 
 \widehat{\mathbf{A}}^{\{2,1\}}  & \widehat{\mathbf{A}}^{\{2,2\}}  
\end{bmatrix}\,,
\end{equation}
where the extra stage does not contribute to the final solution.
%
%\begin{equation}
%\mathbf{I} -  \widehat{\mathbf{A}} \widehat{\mathbf{R}} = 
%\begin{bmatrix}
%\mathbf{I} - r^{\{1\}} \widehat{\mathbf{A}}^{\{1,1\}}  & -r^{\{2\}} \widehat{\mathbf{A}}^{\{1,2\}}   \\ 
%-r^{\{1\}} \widehat{\mathbf{A}}^{\{2,1\}}  & \mathbf{I} -r^{\{2\}} \widehat{\mathbf{A}}^{\{2,2\}}  
%\end{bmatrix}\,,
%\end{equation}
%
In particular, for classical IMEX RK we have
\[
\widehat{\mathbf{A}} = 
\begin{bmatrix}
 \widehat{\mathbf{A}}^\EE &\widehat{\mathbf{A}}^\II  \\ 
 \widehat{\mathbf{A}}^\EE & \widehat{\mathbf{A}}^\II  
\end{bmatrix}\,.
\]
Consequently
\[
\mathbf{I} +  \widehat{\mathbf{A}} \widehat{\mathbf{R}} = 
\begin{bmatrix}
\mathbf{I}_{(s+1) \times (s+1)} + r^\E \widehat{\mathbf{A}}^\EE & r^\I\widehat{\mathbf{A}}^\II  \\ 
r^\E \widehat{\mathbf{A}}^\EE & \mathbf{I}_{(s+1) \times (s+1)} + r^\I\widehat{\mathbf{A}}^\II  
\end{bmatrix}\,.
\]
With
\[
\mathbf{S} = \mathbf{I}_{(s+1) \times (s+1)} + r^\E \widehat{\mathbf{A}}^\EE + r^\I \widehat{\mathbf{A}}^\II 
\]
we have
\begin{equation}
\alpha(r) = \left( \mathbf{I} + \widehat{\mathbf{A}} \widehat{\mathbf{R}} \right)^{-1}  \one_{(2s+2) \times 1} = 
\begin{bmatrix}
    \mathbf{S}^{-1} \one_{(s+1) \times 1} \\
    \mathbf{S}^{-1} \one_{(s+1) \times 1}
\end{bmatrix}\,,
\end{equation}
and
\begin{eqnarray*}
\beta(r) = \left( \mathbf{I} + \widehat{\mathbf{A}} \widehat{\mathbf{R}} \right)^{-1}  \widehat{\mathbf{A}} \widehat{\mathbf{R}} = 
\begin{bmatrix}
   r^\E \, \mathbf{S}^{-1} \widehat{\mathbf{A}}^\EE & \quad
   r^\I \, \mathbf{S}^{-1} \widehat{\mathbf{A}}^\II \\ 
   r^\E \, \mathbf{S}^{-1} \widehat{\mathbf{A}}^\EE & \quad
    r^\I \, \mathbf{S}^{-1} \widehat{\mathbf{A}}^\II
\end{bmatrix}\,.
\end{eqnarray*}
GARK absolute monotonicity conditions \eqref{eqn:am-conditions} are equivalent to the absolute traditional additive 
RK monotonicity conditions obtained by Higueras \cite{Higueras_2006_SSP-ARK}
\begin{eqnarray*}
 \mathbf{S}^{-1}\cdot \one_{(s+1) \times 1} \ge 0\,, \quad
 \mathbf{S}^{-1}\cdot \widehat{\mathbf{A}}^\EE \ge 0\,, \quad
 \mathbf{S}^{-1}\cdot \widehat{\mathbf{A}}^\II \ge 0\,.
\end{eqnarray*}
\end{example}

%%%%%%%%%%%%%%%%%%%%%%%
\begin{example}[Monotonicity of classical-transposed IMEX RK]
%%%%%%%%%%%%%%%%%%%%%%%

Here we have
\[
\widehat{\mathbf{A}} = 
\begin{bmatrix}
 \widehat{\mathbf{A}}^\EE &\widehat{\mathbf{A}}^\EE  \\ 
 \widehat{\mathbf{A}}^\II & \widehat{\mathbf{A}}^\II  
\end{bmatrix}
\]
and
\begin{eqnarray*}
\alpha(r) &=& \left( \mathbf{I} +  \widehat{\mathbf{A}} \widehat{\mathbf{R}} \right)^{-1}  \one_{(2s+2) \times 1}\\
 &=&  \begin{bmatrix}
    \left( \mathbf{I}_{(s+1) \times (s+1)} -  r^\I \,\widehat{\mathbf{A}}^\II +  r^\E \,\widehat{\mathbf{A}}^\EE \right)\, \mathbf{S}^{-1} \one_{(s+1) \times 1} \\
   \left( \mathbf{I}_{(s+1) \times (s+1)} - r^\E \,\widehat{\mathbf{A}}^\EE + r^\I \,\widehat{\mathbf{A}}^\II  \right)\, \mathbf{S}^{-1} \one_{(s+1) \times 1} 
\end{bmatrix} \\
&=& \begin{bmatrix}
    \one_{(s+1) \times 1} - 2\,r^\I \,\widehat{\mathbf{A}}^\II\, \mathbf{S}^{-1} \one_{(s+1) \times 1} \\
    \one_{(s+1) \times 1} - 2\,r^\E \,\widehat{\mathbf{A}}^\EE\, \mathbf{S}^{-1} \one_{(s+1) \times 1}
\end{bmatrix}\,,
\end{eqnarray*}
\begin{eqnarray*}
\beta(r) = \left( \mathbf{I} +  \widehat{\mathbf{A}} \widehat{\mathbf{R}} \right)^{-1}  \widehat{\mathbf{A}} \widehat{\mathbf{R}} = 
\begin{bmatrix}
   r^\E \,\widehat{\mathbf{A}}^\EE  \mathbf{S}^{-1} & \quad
    r^\E \,\widehat{\mathbf{A}}^\EE  \mathbf{S}^{-1} \\
        r^\I \,\widehat{\mathbf{A}}^\II  \mathbf{S}^{-1} & \quad
    r^\I \,\widehat{\mathbf{A}}^\II  \mathbf{S}^{-1} 
\end{bmatrix}\,.
\end{eqnarray*}
The GARK absolute monotonicity conditions \eqref{eqn:am-conditions} lead to
\begin{eqnarray*}
\widehat{\mathbf{A}}^\II\, \mathbf{S}^{-1}\cdot \one_{(s+1) \times 1} &\le& \frac{1}{2\,r^\I} \,  \one_{(s+1) \times 1}\,, \\
\widehat{\mathbf{A}}^\EE\, \mathbf{S}^{-1}\cdot \one_{(s+1) \times 1} &\le& \frac{1}{2\,r^\E} \,  \one_{(s+1) \times 1}\,, \\
 \widehat{\mathbf{A}}^\II\cdot  \mathbf{S}^{-1} &\ge& 0\,, \quad \textnormal{and} \\
 \widehat{\mathbf{A}}^\EE\cdot  \mathbf{S}^{-1} &\ge& 0\,.
\end{eqnarray*}
\end{example}

%%%%%%%%%%%%%%%%%%%%%%%
\begin{example}[A monotonic IMEX-GARK scheme]
%%%%%%%%%%%%%%%%%%%%%%%

Consider the following second order IMEX-GARK scheme.
The explicit method has order two, is strong stability preserving, and has an absolute stability radius $\mathcal{R}^\EE=1$
 \begin{subequations}
 \label{eqn:imex2-monotone}
 \renewcommand{\arraystretch}{1.25}
 \begin{equation}
 \mathbf{A}^\EE =  \begin{bmatrix}  0 & 0 \\ 1 & 0 \end{bmatrix} \,, \quad 
 \mathbf{b}^\E   =  \begin{bmatrix}  \frac{1}{2}\\ \frac{1}{2} \end{bmatrix} \,, \quad
\mathbf{c}^\EE  =   \begin{bmatrix} 0\\ 1 \end{bmatrix}\,.
\end{equation}
The implicit method has order two, is stiffly accurate, and has an absolute stability radius $\mathcal{R}^\II=\infty$.
The coefficients are $ \gamma =(1-\sqrt{2})/2$ and
 \begin{equation}
% \gamma = \frac{1-\sqrt{2}}{2}\,, \quad
 \mathbf{A}^\II =  \begin{bmatrix}  \gamma & ~~0\\ 1 - \gamma & ~~\gamma \end{bmatrix} \,, \quad 
 \mathbf{b}^\I   =  \begin{bmatrix}  1 - \gamma\\ \gamma \end{bmatrix} \,, \quad
\mathbf{c}^\II  =   \begin{bmatrix}  \gamma\\ 1 \end{bmatrix}\,.
\end{equation}
The two methods cannot be paired as a classical IMEX Runge-Kutta method. The following coupling terms
 \begin{equation}
 \mathbf{A}^\EI =  \begin{bmatrix}  0 & ~~0\\ 1 & ~~0 \end{bmatrix} \,, \quad 
 \mathbf{A}^\IE =  \begin{bmatrix}  \gamma & ~~0\\  \alpha & ~~1 - \alpha  \end{bmatrix} \,,
\end{equation}
\renewcommand{\arraystretch}{1}
\end{subequations}
ensure that the GARK scheme is second order. This can be seen from the fact that the internal consistency conditions \eqref{eqn:simplifying-assumption-c} are satisfied.

The coupling has one free parameter $\alpha$. For $\alpha=1-\gamma$ the IMEX scheme is of transposed-classical type. 
Different values of $\alpha$ lead to different regions of monotonicity, as illustrated in Figure \ref{fig:imex2-monotone}. 
A numerical search has revealed that the largest region is obtained for $\alpha=0.5$.

\begin{figure}
\centering
\subfigure[$\alpha=0.25$]{
\includegraphics[width=0.3\textwidth,height=0.3\textwidth]{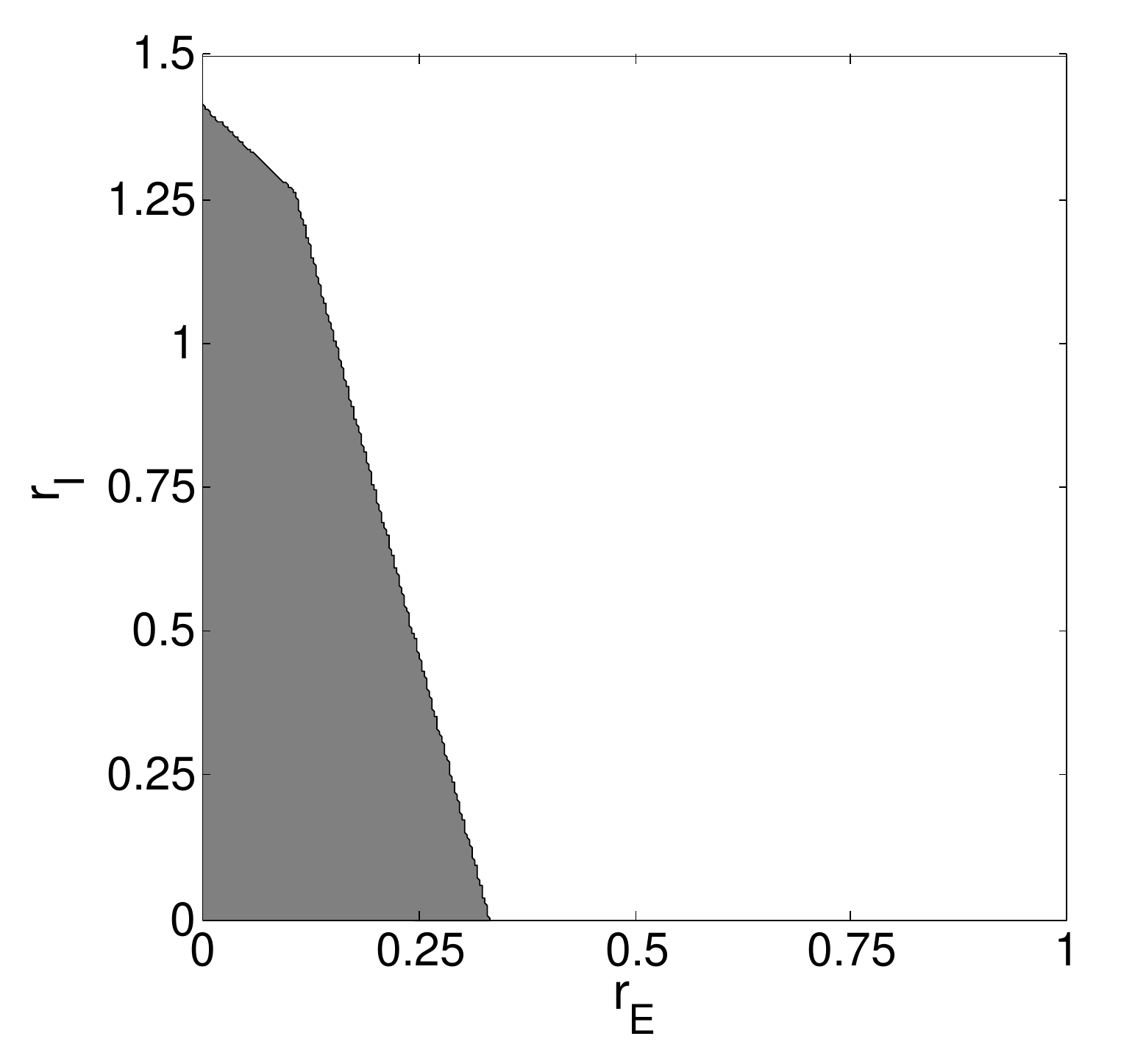}
}
\subfigure[$\alpha=0.5$]{
\includegraphics[width=0.3\textwidth,height=0.3\textwidth]{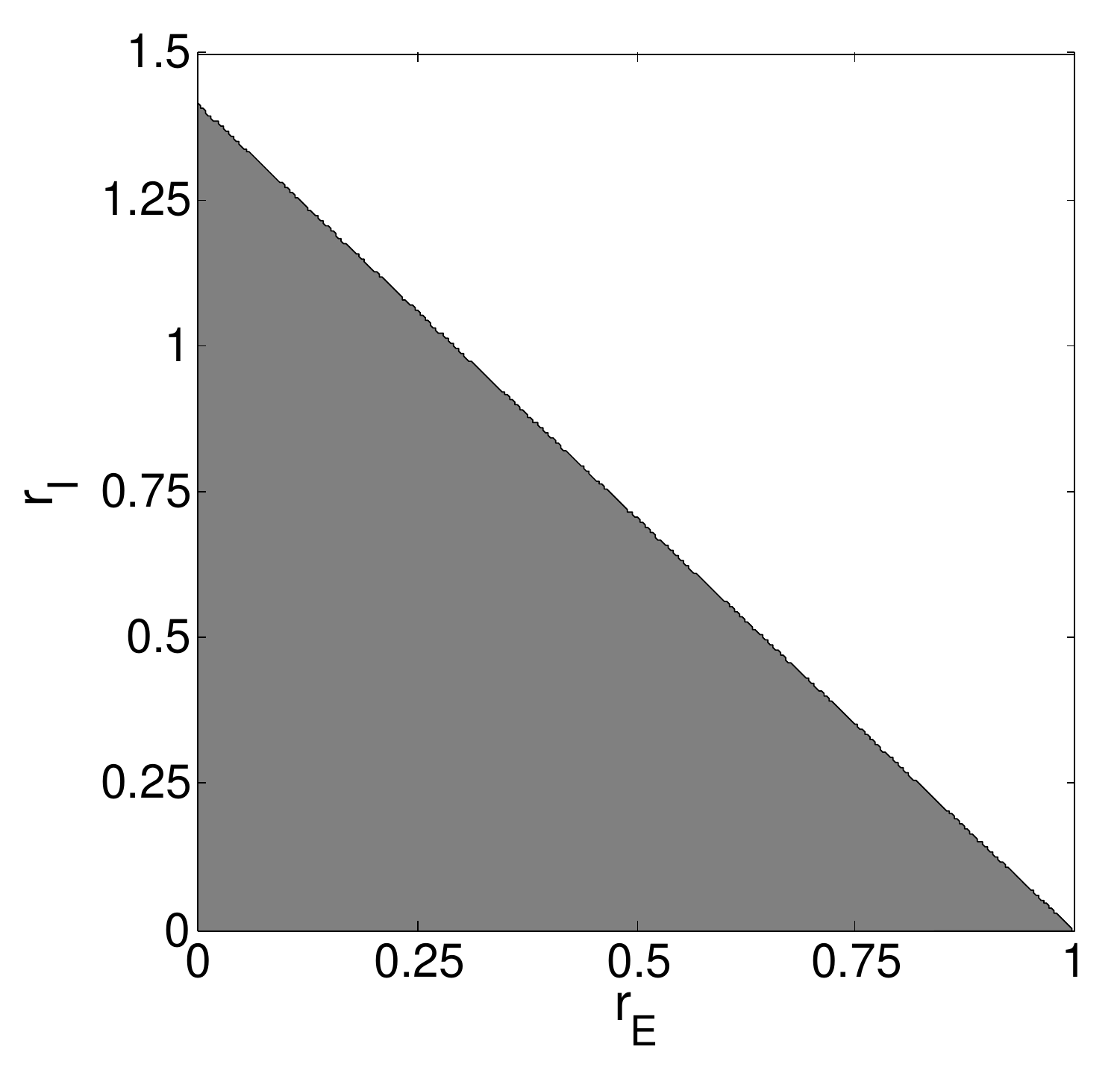}
}
\caption{Absolute monotonicity regions for the second order IMEX scheme \eqref{eqn:imex2-monotone}.}
\label{fig:imex2-monotone}
\end{figure}
The method has the following generalized Butcher tableau:
\renewcommand{\arraystretch}{1.25}
\[
\begin{array}{c|c|c|c}
\mathbf{c}^\EE & \mathbf{A}^\EE &\mathbf{A}^\EI & \mathbf{c}^\EI  \\
\hline 
\mathbf{c}^\IE & \mathbf{A}^\IE & \mathbf{A}^\II  & \mathbf{c}^\II \\
\hline
& \mathbf{b}^\E & \mathbf{b}^\I 
\end{array}
~~:=~~
\begin{array}{c|cc|cc|c} 
0 ~                 &~ 0 ~&~ 0                &~ 0 ~&~ 0 ~                                                               &~ 0\\ 
1 ~                &~ 1 ~&~ 0 ~             &~  1~&~ 0 ~                                                 &~ 1 \\  
\hline
\gamma ~&                ~ \gamma ~&~ 0 ~                           &~ \gamma ~&~ 0 ~                                    &~ \gamma\\ 
1 ~&                ~ \alpha ~&~ 1-\alpha ~             &~ 1-\gamma ~&~ \gamma ~            &~ 1 \\ 
\hline
 ~&                               ~ \frac{1}{2} ~&~ \frac{1}{2} ~&~ ~1-\gamma ~&~ \gamma ~&~  \end{array}~.
\]
\renewcommand{\arraystretch}{1}

\end{example}

Monotonicity conditions for several multirate and partitioned explicit Runge-Kutta schemes are also discussed by Hundsdorfer, Mozartova, and Savcenco in a recent report \cite{Hundsdorfer_2013_mr-monotonicity}.

\section{Implicit-implicit GARK schemes}\label{sec:imim}
%%%%%%%%%%%%%%%%%%%%%%%%%%%%%%

We now consider systems \eqref{eqn:additive-ode} with two way partitioned right hand sides
 where both components $f$ and $g$ are stiff. We apply a
two way partitioned GARK method \eqref{eqn:GARK}
\begin{subequations}
\label{eqn:imimRK}
\begin{eqnarray}
\label{eqn:imim-explicit-stage}
Y_i &=& y_{n} + h  \sum_{j=1}^{i}  a_{i,j}^{\{1,1\}} \, f(Y_j)   + h  \sum_{j=1}^{i-1} a_{i,j}^{\{1,2\}} \, g(Z_j)\,, ~~ i = 1,\dots, s^{\{1\}}, ~~\\
\label{eqn:imim-implicit-stage}
Z_i  &=& y_{n} + h  \sum_{j=1}^{i} a_{i,j}^{\{2,1\}} \, f(Y_j)   + h  \sum_{j=1}^{i}  a_{i,j}^{\{2,2\}} \, g(Z_j)\,, 
~~ i = 1,\dots, s^{\{2\}},\\
\label{eqn:imim-sol}
y_{n+1} &=& y_n + h \sum_{i=1}^{s^{\{1\}}}  b_{i}^{\{1\}} \, f(Y_i)  + h \sum_{i=1}^{s^{\{2\}}} b_{i}^{\{2\}} \, g(Z_i)\,.
\end{eqnarray}
\end{subequations}
The scheme \eqref{eqn:imimRK} has the following characteristics:
\begin{itemize}
\item The discretization is implicit-implicit (IMIM); stages \eqref{eqn:imim-explicit-stage} are implicit in $Y_i$, while stages 
\eqref{eqn:imex-implicit-stage}  are implicit in $Z_i$.
\item One solves in succession nonlinear subsystems corresponding to each individual component.
\item If each of the implicit schemes is algebraically stable, and the GARK scheme is stability decoupled, 
then the separation of subsystem solutions does not affect the algebraic stability  
of the overall method.
\end{itemize}

\begin{example}[An algebraically stable, stability-decoupled DIRK-DIRK method]
 
We consider a pair of DIRK schemes and compute the corresponding coupling conditions.
The first method is second order accurate and algebraically stable with $\mathbf{P}^{\{1,1\}} =\mathbf{0}$
\renewcommand{\arraystretch}{1.25}
\[
\mathbf{A}^{\{1,1\}}  = \left[\begin{array}{cc} \frac{1}{8} & ~~0\\ \frac{1}{4} & ~~\frac{3}{8} \end{array}\right]\,,\quad
\mathbf{b}^{\{1\}}   = \left[\begin{array}{c} \frac{1}{4}\\ \frac{3}{4} \end{array}\right]\,, \quad
\mathbf{c}^{\{1,1\}}  = \left[\begin{array}{c} \frac{1}{8}\\ \frac{5}{8} \end{array}\right]\,.
\]
The second method is second order accurate and algebraically stable with $\mathbf{P}^{\{2,2\}} =\mathbf{0}$
\[
\mathbf{A}^{\{2,2\}}  = \left[\begin{array}{cc} \frac{1}{3} & ~~0\\ \frac{2}{3} & ~~\frac{1}{6} \end{array}\right]\,, \quad
\mathbf{b}^{\{2\}}  =   \left[\begin{array}{c} \frac{2}{3}\\ \frac{1}{3} \end{array}\right]\,, \quad
\mathbf{c}^{\{2,2\}} =   \left[\begin{array}{c} \frac{1}{3}\\ \frac{5}{6} \end{array}\right]\,.
\]
The coupling coefficients
\[
\mathbf{A}^{\{1,2\}}  = \left[\begin{array}{cc} 0 & ~~0\\ \frac{2}{3} & ~~0 \end{array}\right]\,, \quad
\mathbf{c}^{\{1,2\}} = \left[\begin{array}{c} 0\\ \frac{2}{3} \end{array}\right]\,,
\]
and
\[
\mathbf{A}^{\{2,1\}}  = \left[\begin{array}{cc} \frac{1}{4} & ~~0\\ \frac{1}{4} & ~~\frac{3}{4} \end{array}\right]\,, \quad
\mathbf{c}^{\{2,1\}}  = \left[\begin{array}{c} \frac{1}{4}\\ 1 \end{array}\right]\,,
\]
ensure that the GARK method is second order accurate and is stability-decoupled, $\mathbf{P}^{\{1,2\}} =(\mathbf{P}^{\{2,1\}})^T=\mathbf{0}$.
The generalized Butcher tableau \eqref{eqn:generalized-tableau} of the scheme reads:
\[
\begin{array}{c|c|c|c}
\mathbf{c}^{\{1,1\}} & \mathbf{A}^{\{1,1\}} &\mathbf{A}^{\{1,2\}} & \mathbf{c}^{\{1,2\}} \\
\hline 
\mathbf{c}^{\{2,1\}} & \mathbf{A}^{\{2,1\}} & \mathbf{A}^{\{2,2\}}  & \mathbf{c}^{\{2,2\}} \\
\hline
& \mathbf{b}^{\{1\}} & \mathbf{b}^{\{2\}}
\end{array}
~~:=~~
\begin{array}{c|cc|cc|c} 
\frac{1}{8} ~&~ \frac{1}{8} ~&~ 0 ~&~ 0 ~&~ 0 ~&~0 \\ 
\frac{5}{8} ~&~ \frac{1}{4} ~&~ \frac{3}{8} ~&~ \frac{2}{3} ~&~ 0 ~&~  \frac{2}{3} \\ 
\hline
\frac{1}{4}  ~&~ \frac{1}{4} ~&~ 0 ~&~ \frac{1}{3} ~&~ 0 ~&~ \frac{1}{3}\\ 
1 ~&~ \frac{1}{4} ~&~ \frac{3}{4}  ~&~ \frac{2}{3} ~&~ \frac{1}{6} ~&~ \frac{5}{6}\\
\hline
~&~ \frac{1}{4} ~&~ \frac{3}{4} ~&~ \frac{2}{3} ~&~ \frac{1}{3} ~&~  
\end{array}~.
\]
\renewcommand{\arraystretch}{1}
The method proceeds as follows:
\begin{eqnarray*}
\underline{Y_1} &=& y_{n} + \frac{1}{8}\, h \, f(\underline{Y_1})  \,,\\
\underline{Z_1}  &=& y_{n} + \frac{1}{4}\, h \, f(Y_1)   + \frac{1}{3}\, h \, g(\underline{Z_1})\,, \\
\underline{Y_2} &=& y_{n} + \frac{1}{4}\, h \, f(Y_1)  + \frac{3}{8}\, h \, f(\underline{Y_2})  + \frac{2}{3}\, h \, g(Z_1)\,,\\
\underline{Z_2}  &=& y_{n} + \frac{1}{4}\, h \, f(Y_1)  + \frac{3}{4}\, h \, f(Y_2) + \frac{2}{3}\, h \, g(Z_1) +\frac{1}{6}\, h \, g(\underline{Z_2})\,, \\
y_{n+1} &=& y_n + \frac{1}{4} \,h\, f(Y_1)  + \frac{3}{4}\,h \, f(Y_2) + \frac{2}{3} \, h\, g(Z_1) + \frac{1}{3}\, h \, g(Z_2)\,,
\end{eqnarray*}
where in each stage one nonlinear system is solved for the underlined variable.

%One notes that this scheme is symplectic according to Theorem~\ref{theorem.symplectic}. It is straightforward to show that 
%a second order DIRK scheme of type~\eqref{eqn:imimRK} with two stages cannot be symmetric and symplectic at the same time.
\end{example}

\section{Conclusions and future work}\label{sec:conclusions}
%%%%%%%%%%%%%%%%%%%%%%%%%%%%%%%%%

This work develops a generalized additive Runge-Kutta family of methods. 
The new GARK schemes extend the class of additively partitioned Runge-Kutta methods by allowing for different stage values as  arguments of different components of the right hand side.

The theoretical investigations develop order conditions for the GARK family using the NB-series theory. We carry out linear and nonlinear stability analyses, extend the definition of algebraic stability to the new generalized family of schemes, and show that it is possible to construct stability-decoupled methods. We also perform a monotonicity analysis, extend the concept of absolute monotonicity to our new family, and prove monotonic behavior under step size restrictions.

We develop  implicit-explicit GARK schemes in the new framework. We show that classical implicit-explicit Runge-Kutta methods are a particular subset, and develop a new set of transposed-classical schemes. A theoretical investigation of the stiff convergence motivates an extension of the stiff accuracy concept. We construct implicit-implicit GARK methods
where the nonlinear system at each stage involves only one component of the system.

Future work will search for practical methods of high order, and will test their performance on relevant problems. Partitioned implicit-implicit methods \eqref{eqn:imimRK} with the same coefficients on the diagonal of each method (SDIRK type) are desirable, as  are orders higher than two and a partitioning of the right hand side into three or more components. Also stiffly accurate methods with $\mathbf{P}^{\{i,i\}}>0$ and $\mathbf{P}^{\{i,j\}}=0$ when $i \ne j$ are desirable for systems driven by multiple stiff physical processes. We are developing multirate schemes \cite{Guenther_2013_GARK-MR} as well as symplectic schemes  \cite{Guenther_2013_GARK-Hamiltonian} based on the generalized additive Runge-Kutta framework presented here.

%\section*{Acknowledgements}
%The work of A. Sandu has been supported in part by NSF through awards NSF
%OCI--8670904397, NSF CCF--0916493, NSF DMS--0915047, NSF CMMI--1130667, 
%NSF CCF--1218454, AFOSR FA9550--12--1--0293--DEF, AFOSR 12-2640-06,
%and by the Computational Science Laboratory at Virginia Tech.
%
%The work of M. G\"unther has been supported in part by BMBF through grant 
%03MS648E.

\bibliographystyle{siam}
%\bibliography{../../../Bib/ode_monotone,../../../Bib/ode_multirate,../gark_local}

\begin{thebibliography}{10}

\bibitem{SaSe1997}
{\sc A.~L. Araujo, A.~Murua, and J.~M. Sanz-Serna}, {\em Symplectic methods
  based on decompositions}, SIAM Journal on Numerical Analysis, 34 (1997),
  pp.~1926--1947.

\bibitem{Cooper_1983_ARK}
{\sc G.J. Cooper and A.~Sayfy}, {\em Additive {Runge-Kutta} methods for stiff
  ordinary differential equations}, Mathematics of Computation, 40 (1983),
  pp.~207--218.

\bibitem{Garcia-Celayeta_2006_ARK-monotonicity}
{\sc B.~Garcia-Celayeta, I.~Higueras, and T.~Roldan}, {\em
  Contractivity/monotonicity for additive {R}unge-{K}utta methods: inner
  product norms}, Applied Numerical Mathematics, 56 (2006), pp.~862--878.

\bibitem{Gottlieb_2001_SSP-review}
{\sc S~Gottlieb, CW~Shu, and E~Tadmor}, {\em Strong stability-preserving
  high-order time discretization methods}, SIAM Review, 43 (2001), pp.~89--112.

\bibitem{Guenther_2013_GARK-Hamiltonian}
{\sc M.~Guenther and A.~Sandu}, {\em {GARK} methods for {Hamiltonian} systems}.
\newblock In preparation, 2013.

\bibitem{Guenther_2013_GARK-MR}
\leavevmode\vrule height 2pt depth -1.6pt width 23pt, {\em Multirate {GARK}
  methods}.
\newblock In preparation, 2013.

\bibitem{Hairer_1981_Pseries}
{\sc E.~Hairer}, {\em Order conditions for numerical methods for partitioned
  ordinary differential equations}, Numerische Mathematik, 36 (1981),
  pp.~431--445.

\bibitem{Hairer_book_II}
{\sc E.~Hairer and G.~Wanner}, {\em Solving Ordinary Differential Equations
  {II}: Stiff and Differential-Algebraic Problems}, Springer, 1993.

\bibitem{Higueras_2004_SSP}
{\sc I.~Higueras}, {\em On strong stability preserving time discretization
  methods}, Journal of Scientific Computing, 21 (2004), pp.~193--223.

\bibitem{Higueras_2005_monotonicity}
\leavevmode\vrule height 2pt depth -1.6pt width 23pt, {\em Monotonicity for
  {Runge-Kutta} methods: inner product norms}, Journal of Scientific Computing,
  24 (2005), pp.~97--117.

\bibitem{Higueras_2005_SSP-representation}
\leavevmode\vrule height 2pt depth -1.6pt width 23pt, {\em Representations of
  {Runge-Kutta} methods and strong stability preserving methods}, SIAM Journal
  on Numerical Analysis, 43 (2005), pp.~924--948.

\bibitem{Higueras_2006_SSP-ARK}
\leavevmode\vrule height 2pt depth -1.6pt width 23pt, {\em Strong stability for
  additive {Runge-Kutta} methods}, SIAM Journal on Numerical Analysis, 44
  (2006), pp.~1735--1758.

\bibitem{Higueras_2009_SSP-ARK}
\leavevmode\vrule height 2pt depth -1.6pt width 23pt, {\em Characterizing
  strong stability preserving additive {Runge-Kutta} methods}, Journal of
  Scientific Computing, 39 (2009), pp.~115--128.

\bibitem{Higueras_2013_scicade}
{\sc I.~Higueras and T.~Roldan}, {\em Efficient implicit-explicit {Runge-Kutta}
  methods with low storage requirements}.
\newblock Presentation at SciCADE 2013, Valladolid, Spain, September 2013.

\bibitem{Hundsdorfer_2013_mr-monotonicity}
{\sc W.~Hundsdorfer, A.~Mozartova, and V.~Savcenco}, {\em Monotonicity
  conditions for multirate and partitioned explicit {Runge-Kutta} schemes}.
\newblock CWI report, unpublished, 2013.

\bibitem{Kennedy_2003}
{\sc A.C. Kennedy and M.H. Carpenter}, {\em Additive {R}unge-{K}utta schemes
  for convection-diffusion-reaction equations}, Appl. Numer. Math., 44 (2003),
  pp.~{139--181}.

\bibitem{Kraaijevanger_1991_contractivity}
{\sc J.F.B.M. Kraaijevanger}, {\em Contractivity of {Runge-Kutta} methods},
  {BIT Numerical Mathematics}, 31 (1991), pp.~482--528.

\bibitem{Kvaerno_2004_ESDIRK}
{\sc A.~Kvaerno}, {\em {Singly diagonally implicit Runge-Kutta methods with an
  explicit first stage}}, BIT Numerical Mathematics, 44 (2004), pp.~489--502.

\bibitem{Prothero_1974_PR}
{\sc A.~Prothero and A.~Robinson}, {\em On the stability and accuracy of
  one-step methods for solving stiff systems of ordinary differential
  equations}, Mathematics of Computation, 28 (1974), pp.~145--162.

\bibitem{Rentrop_1985}
{\sc P.~Rentrop}, {\em Partitioned {Runge-Kutta} methods with stepsize control
  and stiffness detection}, Numerische Mathematik, 47 (1985), pp.~545--564.

\bibitem{Rice_1960}
{\sc J.R. Rice}, {\em Split {R}unge-{K}utta methods for simultaneous
  equations}, Journal of Research of the National Institute of Standards and
  Technology, 64 (1960).

\bibitem{Spijker_2007}
{\sc M.N. Spijker}, {\em Stepsize conditions for general monotonicity in
  numerical initial value problems}, SIAM Journal on Numerical Analysis, 45
  (2007), pp.~1226--1245.

\bibitem{Ascher_1997_IMEX_RK}
{\sc {U.M. Ascher and S.J. Ruuth and R.J. Spiteri}}, {\em {Implicit-explicit
  Runge-Kutta methods for time-dependent partial differential equations}},
  {Applied Numerical Mathematics}, 25 ({1997}), pp.~151--167.

\bibitem{Weiner_1993_partitioning}
{\sc R.~Weiner, M.~Arnold, P.~Rentrop, and K.~Strehmel}, {\em Partitioning
  strategies in {Runge-Kutta} type methods}, IMA Journal on Numerical Analysis,
  13 (1993), pp.~303--319.

\end{thebibliography}

\newpage
\appendix

%%%%%%%%%%%%%%%%%%%%%%%%%%%%%%%%%%%%%%%%%%%
\section{GARK order conditions}\label{sec:oder-conditions-detailed}
%%%%%%%%%%%%%%%%%%%%%%%%%%%%%%%%%%%%%%%%%%%

The specific conditions for orders one to four are as follows.
\begin{subequations} 
\label{eqn:GARK-order-conditions-1-to-4-detail}
\par\noindent{Order 1:}
\begin{equation} 
\label{eqn:GARK-order-conditions-1-detail}
\sum_{i=1}^{s^{\{\sigma\}}} b_i^{\{\sigma\}} = 1,  \quad \forall\; \sigma=1,\ldots, N.
\end{equation}
{Order 2:}
\begin{eqnarray}
\label{eqn:GARK-order-conditions-2-detail}
 \sum_{i=1}^{s^{\{\sigma\}}} \sum_{j=1}^{s^{\{\nu\}}} b_i^{\{\sigma\}} a_{i,j}^{\{\sigma,\nu\}} & = & \frac{1}{2},  \quad \forall \; \sigma, \nu=1,\ldots, N. 
\end{eqnarray}
{Order 3:}
\begin{eqnarray}
\label{eqn:GARK-order-conditions-3a-detail}
 \sum_{i=1}^{s^{\{\sigma\}}} \sum_{j=1}^{s^{\{\nu\}}} \sum_{k=1}^{s^{\{\mu\}}} b_i^{\{\sigma\}} a_{i,j}^{\{\sigma,\nu\}}  a_{i,k}^{\{\sigma,\mu\}}
& = & \frac{1}{3},  \\
\label{eqn:GARK-order-conditions-3b-detail}
 \sum_{i=1}^{s^{\{\sigma\}}} \sum_{j=1}^{s^{\{\nu\}}} \sum_{k=1}^{s^{\{\mu\}}} b_i^{\{\sigma\}} a_{i,j}^{\{\sigma,\nu\}}  a_{j,k}^{\{\nu, \mu\}} 
& = & \frac{1}{6},  \\
\nonumber
\qquad \forall \; \sigma,\nu, \mu =1,\ldots, N. &&
\end{eqnarray}
{Order 4:}
\begin{eqnarray}
\label{eqn:GARK-order-conditions-4a-detail}
 \sum_{i=1}^{s^{\{\sigma\}}} \sum_{j=1}^{s^{\{\nu\}}} \sum_{\ell=1}^{s^{\{\lambda\}}}  \sum_{m=1}^{s^{\{\mu\}}} b_i^{\{\sigma\}} a_{i,j}^{\{\sigma,\nu\}}  a_{i,l}^{\{\sigma,\lambda\}}  a_{i,m}^{\{\sigma,\mu\}}
& = & \frac{1}{4}, \\
\label{eqn:GARK-order-conditions-4b-detail}
 \sum_{i=1}^{s^{\{\sigma\}}} \sum_{j=1}^{s^{\{\nu\}}} \sum_{\ell=1}^{s^{\{\lambda\}}}  \sum_{m=1}^{s^{\{\mu\}}} b_i^{\{\sigma\}} a_{i,j}^{\{\sigma,\nu\}}  a_{j,\ell}^{\{\nu,\lambda\}}  a_{i,m}^{\{\sigma,\mu\}}
& = & \frac{1}{8},  \\
\label{eqn:GARK-order-conditions-4c-detail}
 \sum_{i=1}^{s^{\{\sigma\}}} \sum_{j=1}^{s^{\{\nu\}}} \sum_{\ell=1}^{s^{\{\lambda\}}}  \sum_{m=1}^{s^{\{\mu\}}} b_i^{\{\sigma\}} a_{i,j}^{\{\sigma,\nu\}}  a_{j,\ell}^{\{\nu,\lambda\}}  a_{j,m}^{\{\nu,\mu\}}
& = & \frac{1}{12},  \\
\label{eqn:GARK-order-conditions-4d-detail}
 \sum_{i=1}^{s^{\{\sigma\}}} \sum_{j=1}^{s^{\{\nu\}}} \sum_{\ell=1}^{s^{\{\lambda\}}}  \sum_{m=1}^{s^{\{\mu\}}} b_i^{\{\sigma\}} a_{i,j}^{\{\sigma,\nu\}}  a_{j,\ell}^{\{\nu,\lambda\}}  a_{\ell,m}^{\{\lambda,\mu\}}
& = & \frac{1}{24},  \\
\nonumber
\quad \forall\; \sigma,\nu,\lambda, \mu =1,\ldots, N\,. &&
\end{eqnarray}
\end{subequations}

\newpage
%%%%%%%%%%%%%%%%%%%%%%%%%%%%%%%%%%%%%%%%%%%
\section{GARK IMEX third order conditions}\label{sec:oder-conditions-imex}
%%%%%%%%%%%%%%%%%%%%%%%%%%%%%%%%%%%%%%%%%%%

Each of the implicit and explicit methods $\left(\mathbf{A}^{\{\sigma,\sigma\}}, \mathbf{b}^{\{\sigma\}}\right)$ and 
$\mathbf{c}^{\{\sigma,\sigma\}} = \mathbf{A}^{\{\sigma,\sigma\}} \cdot \one^{\{\sigma\}} $ has to satisfy the corresponding order conditions for $\sigma  \in \{ \textsc{e},\textsc{i} \}$. In addition, the following  coupling conditions are required for third order accuracy.

The IMEX order two coupling conditions are:
\begin{subequations}
\label{eqn:coupling-2}
\begin{eqnarray}
\label{eqn:coupling-2a}
\mathbf{b}^\E\,^T\cdot \mathbf{c}^\EI& = & \frac{1}{2}\,, \\ 
\label{eqn:coupling-2b}
\mathbf{b}^\I\,^T\cdot \mathbf{c}^\IE& = & \frac{1}{2} \,.
\end{eqnarray}
\end{subequations}
These are equivalent to the requirement that the coupling methods $\left(\mathbf{A}^{\{\sigma,\nu\}}, \mathbf{b}^{\{\sigma\}}\right)$ for $\sigma \ne \mu$ are second order.

The IMEX order three coupling conditions read:
\begin{subequations}
\label{eqn:coupling-3}
\begin{eqnarray}
\mathbf{b}^\E\,^T \cdot \left( \mathbf{c}^\EE \mathbf{c}^\EI\right) & = & \frac{1}{3},   \\ 
\mathbf{b}^\E\,^T \cdot \mathbf{A}^\EE  \cdot \mathbf{c}^\EI  & = & \frac{1}{6}, \\ 
\mathbf{b}^\E\,^T \cdot \left( \mathbf{c}^\EI \mathbf{c}^\EE\right) & = & \frac{1}{3},   \\ 
\mathbf{b}^\E\,^T \cdot \mathbf{A}^\EI  \cdot \mathbf{c}^\IE  & = & \frac{1}{6}, \\ 
\mathbf{b}^\E\,^T \cdot \left( \mathbf{c}^\EI \mathbf{c}^\EI\right) & = & \frac{1}{3},   \\ 
\mathbf{b}^\E\,^T \cdot \mathbf{A}^\EI  \cdot \mathbf{c}^\II  & = & \frac{1}{6}, \\ 
\mathbf{b}^\I\,^T \cdot \left( \mathbf{c}^\IE \mathbf{c}^\IE\right) & = & \frac{1}{3},   \\ 
\mathbf{b}^\I\,^T \cdot \mathbf{A}^\IE  \cdot \mathbf{c}^\EE  & = & \frac{1}{6}, \\ 
\mathbf{b}^\I\,^T \cdot \left( \mathbf{c}^\IE \mathbf{c}^\II\right) & = & \frac{1}{3},   \\ 
\mathbf{b}^\I\,^T \cdot \mathbf{A}^\IE  \cdot \mathbf{c}^\EI  & = & \frac{1}{6}, \\ 
\mathbf{b}^\I\,^T \cdot \left( \mathbf{c}^\II \mathbf{c}^\IE\right) & = & \frac{1}{3},   \\ 
\mathbf{b}^\I\,^T \cdot \mathbf{A}^\II  \cdot \mathbf{c}^\IE  & = & \frac{1}{6}. 
\end{eqnarray}
\end{subequations}

\newpage
%%%%%%%%%%%%%%%%%%%%%%%%%%%%%%%%%%%%%%%%%%%
\section{GARK IMEX fourth order conditions}\label{sec:oder-conditions-imex-4}
%%%%%%%%%%%%%%%%%%%%%%%%%%%%%%%%%%%%%%%%%%%

The order four IMEX order conditions \eqref{eqn:imex-coupling-order-4} under the simplifying assumption \eqref{eqn:simplifying-assumption-c} are as follows.
We have two compatibility relation between the implicit and the explicit methods:
\begin{eqnarray*}
\mathbf{b}^\E\,^T \cdot \mathbf{A}^\EE\cdot \mathbf{A}^\II\cdot \mathbf{c}^\I & = & \frac{1}{24} \,, \\
\mathbf{b}^\I\,^T \cdot \mathbf{A}^\II\cdot \mathbf{A}^\EE\cdot \mathbf{c}^\E & = & \frac{1}{24} \,, 
\end{eqnarray*}
and 16 conditions involving the coupling terms:
\begin{eqnarray*}
\mathbf{b}^\E\,^T \cdot \mathbf{A}^\EE\cdot \mathbf{A}^\IE\cdot \mathbf{c}^\I & = & \frac{1}{24} \,, \\ 
\mathbf{b}^\E\,^T \cdot \mathbf{A}^\EE\cdot \mathbf{A}^\EI\cdot \mathbf{c}^\E & = & \frac{1}{24}\,,  \\ 
\left(\mathbf{b}^\E \mathbf{c}^\E \right)^T \cdot \mathbf{A}^\EI   \cdot \mathbf{c}^\I  & = & \frac{1}{8}\,,  \\ 
\mathbf{b}^\E\,^T \cdot \mathbf{A}^\EI  \cdot\left( \mathbf{c}^\I  \mathbf{c}^\I \right)& = & \frac{1}{12}\,, 
\end{eqnarray*}
\begin{eqnarray*}
\mathbf{b}^\E\,^T \cdot \mathbf{A}^\EI\cdot \mathbf{A}^\EE\cdot \mathbf{c}^\E & = & \frac{1}{24}\,,  \\ 
\mathbf{b}^\E\,^T \cdot \mathbf{A}^\EI\cdot \mathbf{A}^\IE\cdot \mathbf{c}^\I & = & \frac{1}{24}\,,  \\ 
\mathbf{b}^\E\,^T \cdot \mathbf{A}^\EI\cdot \mathbf{A}^\EI\cdot \mathbf{c}^\E & = & \frac{1}{24}\,,  \\ 
\mathbf{b}^\E\,^T \cdot \mathbf{A}^\EI\cdot \mathbf{A}^\II\cdot \mathbf{c}^\I & = & \frac{1}{24}\,,   
\end{eqnarray*}
\begin{eqnarray*}
\left(\mathbf{b}^\I \mathbf{c}^\I \right)^T \cdot \mathbf{A}^\IE   \cdot \mathbf{c}^\E  & = & \frac{1}{8}\,,  \\ 
\mathbf{b}^\I\,^T \cdot \mathbf{A}^\IE  \cdot\left( \mathbf{c}^\E  \mathbf{c}^\E \right)& = & \frac{1}{12}\,,  \\ 
\mathbf{b}^\I\,^T \cdot \mathbf{A}^\IE\cdot \mathbf{A}^\EE\cdot \mathbf{c}^\E & = & \frac{1}{24} \,, \\ 
\mathbf{b}^\I\,^T \cdot \mathbf{A}^\IE\cdot \mathbf{A}^\IE\cdot \mathbf{c}^\I & = & \frac{1}{24} \,, 
\end{eqnarray*}
\begin{eqnarray*}
\mathbf{b}^\I\,^T \cdot \mathbf{A}^\IE\cdot \mathbf{A}^\EI\cdot \mathbf{c}^\E & = & \frac{1}{24}\,,  \\ 
\mathbf{b}^\I\,^T \cdot \mathbf{A}^\IE\cdot \mathbf{A}^\II\cdot \mathbf{c}^\I & = & \frac{1}{24}\,,  \\ 
\mathbf{b}^\I\,^T \cdot \mathbf{A}^\II\cdot \mathbf{A}^\IE\cdot \mathbf{c}^\I & = & \frac{1}{24}\,,  \\ 
\mathbf{b}^\I\,^T \cdot \mathbf{A}^\II\cdot \mathbf{A}^\EI\cdot \mathbf{c}^\E & = & \frac{1}{24} \,.
\end{eqnarray*}

When both simplifying assumptions \eqref{eqn:simplifying-assumption-c} and \eqref{eqn:simplifying-assumption-additional} are satisfied,
each of the coupling methods $(\mathbf{A}^\IE,\mathbf{b},\mathbf{c})$ and $(\mathbf{A}^\EI,\mathbf{b},\mathbf{c})$
needs to be fourth order accurate in its own right. In addition the following 12 coupling conditions are required:
\begin{eqnarray*}
\mathbf{b}^T \cdot \mathbf{A}^\EE\cdot \mathbf{A}^\II\cdot \mathbf{c} & = & \frac{1}{24}\,,  \\
\mathbf{b}^T \cdot \mathbf{A}^\EE\cdot \mathbf{A}^\IE\cdot \mathbf{c} & = & \frac{1}{24}\,,  \\ 
\mathbf{b}^T \cdot \mathbf{A}^\EE\cdot \mathbf{A}^\EI\cdot \mathbf{c} & = & \frac{1}{24} \,,  
\end{eqnarray*}
\begin{eqnarray*}
\mathbf{b}^T \cdot \mathbf{A}^\II\cdot \mathbf{A}^\IE\cdot \mathbf{c} & = & \frac{1}{24}\,,  \\ 
\mathbf{b}^T \cdot \mathbf{A}^\II\cdot \mathbf{A}^\EI\cdot \mathbf{c} & = & \frac{1}{24}\,,  \\
\mathbf{b}^T \cdot \mathbf{A}^\II\cdot \mathbf{A}^\EE\cdot \mathbf{c} & = & \frac{1}{24} \,, 
\end{eqnarray*}
\begin{eqnarray*}
%\left(\mathbf{b} \mathbf{c} \right)^T \cdot \mathbf{A}^\EI   \cdot \mathbf{c}  & = & \frac{1}{8} \\ 
%\left(\mathbf{b}\right)^T \cdot \mathbf{A}^\EI  \cdot\left( \mathbf{c}  \mathbf{c} \right)& = & \frac{1}{12}\\
%\mathbf{b}^T \cdot \mathbf{A}^\EI\cdot \mathbf{A}^\EI\cdot \mathbf{c} & = & \frac{1}{24} \\ 
\mathbf{b}^T \cdot \mathbf{A}^\EI\cdot \mathbf{A}^\EE\cdot \mathbf{c} & = & \frac{1}{24}\,,  \\ 
\mathbf{b}^T \cdot \mathbf{A}^\EI\cdot \mathbf{A}^\IE\cdot \mathbf{c} & = & \frac{1}{24} \,, \\ 
\mathbf{b}^T \cdot \mathbf{A}^\EI\cdot \mathbf{A}^\II\cdot \mathbf{c} & = & \frac{1}{24} \,,  
\end{eqnarray*}
\begin{eqnarray*}
%\left(\mathbf{b} \mathbf{c} \right)^T \cdot \mathbf{A}^\IE   \cdot \mathbf{c}  & = & \frac{1}{8} \\ 
%\left(\mathbf{b}\right)^T \cdot \mathbf{A}^\IE  \cdot\left( \mathbf{c}  \mathbf{c} \right)& = & \frac{1}{12} \\ 
%\mathbf{b}^T \cdot \mathbf{A}^\IE\cdot \mathbf{A}^\IE\cdot \mathbf{c} & = & \frac{1}{24} \\
\mathbf{b}^T \cdot \mathbf{A}^\IE\cdot \mathbf{A}^\EE\cdot \mathbf{c} & = & \frac{1}{24} \,, \\ 
\mathbf{b}^T \cdot \mathbf{A}^\IE\cdot \mathbf{A}^\EI\cdot \mathbf{c} & = & \frac{1}{24}\,,  \\ 
\mathbf{b}^T \cdot \mathbf{A}^\IE\cdot \mathbf{A}^\II\cdot \mathbf{c} & = & \frac{1}{24} \,. 
\end{eqnarray*}
Moreover, if the two coupling terms are equal, $\mathbf{A}^\IE=\mathbf{A}^\EI=\mathbf{A}^{\rm \{cpl\}}$, then 
$(\mathbf{A}^{\rm \{cpl\}},\mathbf{b},\mathbf{c})$ needs to be fourth order accurate. The remaining six coupling conditions are
\begin{eqnarray*}
\mathbf{b}^T \cdot \mathbf{A}^\EE\cdot \mathbf{A}^\II\cdot \mathbf{c} & = & \frac{1}{24}\,, \\
\mathbf{b}^T \cdot \mathbf{A}^\EE\cdot \mathbf{A}^{\rm \{cpl\}}\cdot \mathbf{c} & = & \frac{1}{24}\,, \\ 
\mathbf{b}^T \cdot \mathbf{A}^\II\cdot \mathbf{A}^\EE\cdot \mathbf{c} & = & \frac{1}{24}\,,  \\
\mathbf{b}^T \cdot \mathbf{A}^\II\cdot \mathbf{A}^{\rm \{cpl\}}\cdot \mathbf{c} & = & \frac{1}{24}\,, \\
\mathbf{b}^T \cdot \mathbf{A}^{\rm \{cpl\}}\cdot \mathbf{A}^\EE\cdot \mathbf{c} & = & \frac{1}{24} \,, \\ 
\mathbf{b}^T \cdot \mathbf{A}^{\rm \{cpl\}}\cdot \mathbf{A}^\II\cdot \mathbf{c} & = & \frac{1}{24}  \,.
\end{eqnarray*}

\newpage
%%%%%%%%%%%%%%%%%%%%%%%%%%%%%%%%%
\section{Extended Prothero-Robinson analysis}
%%%%%%%%%%%%%%%%%%%%%%%%%%%%%%%%%

We consider the modified Prothero-Robinson (PR) \cite{Prothero_1974_PR} test problem written as a split system \eqref{eqn:imex-ode}
\begin{equation}
\begin{bmatrix} y \\  w \end{bmatrix}' = \underbrace{ \begin{bmatrix} \mu\, (y - w)  \\ 0  \end{bmatrix} }_{g(t,y)} + \underbrace{  \begin{bmatrix}  \phi'(t) \\  \phi'(t) \end{bmatrix} }_{f(t,y)} ~, \quad \mu < 0~, \quad \begin{bmatrix} y(0) \\  w(0) \end{bmatrix}= \begin{bmatrix} \phi(0) \\  \phi(0) \end{bmatrix} \,,
\label{Prothero-Robinson-modified}
\end{equation}
The method \eqref{eqn:imexRK} applied to the scalar equation \eqref{Prothero-Robinson-modified} reads
\begin{subequations}
\label{eqn:imex-on-PRM}
\begin{eqnarray}
\label{eqn:imex-PRM-explicit-stage}
Y &=& y_{n} \, \one + h  \,  \mathbf{A}^\EE \, \phi'^\E   + h  \, \mu \, \mathbf{A}^\EI \,  \left( Z -  W \right)\,, \\
%\nonumber
%U &=& u_{n} \, \one + h  \,  \mathbf{A}^\EE \, \phi'^\E   \,, \\
\label{eqn:imex-PRM-implicit-stage}
Z  &=& y_{n}\, \one  + h  \,  \mathbf{A}^\IE \, \, \phi'^\E   + h  \, \mu \, \mathbf{A}^\II \,  \left( Z -   W \right)\,, \\
\label{eqn:imex-PRM-W-stage}
W &=& w_{n} \, \one + h  \,  \mathbf{A}^\IE \, \phi'^\I   \,, \\
\label{eqn:imex-sol3}
y_{n+1} &=& y_n + h \, \mathbf{b}^\E\,^T \,\, \phi'^\E  + h   \, \mu \, \mathbf{b}^\I\,^T \, \left( Z -  W\right)\,, \\
\label{eqn:imex-W-sol}
w_{n+1} &=& w_n + h \, \mathbf{b}^\E\,^T \, \phi'^\E  \,.
\end{eqnarray}
\end{subequations}
Here 
\begin{eqnarray*}
\phi^\E &=&  \phi\left(t_{n-1} + \mathbf{d}^\E\, h \right) = \left[  \phi(t_{n-1} +d_1^\E\, h ),\ldots,\phi(t_{n-1} + d_{s^\E}^\E\, h ) \right]^T\,, \\
\phi^\I &=&  \phi\left(t_{n-1} + \mathbf{d}^\I\, h \right) = \left[  \phi(t_{n-1} +d_1^\I\, h ),\ldots,\phi(t_{n-1} + d_{s^\I}^\I\, h ) \right]^T\,, \\
\end{eqnarray*}
where $\mathbf{d}^\E$, $\mathbf{d}^\I$ are the stage approximation times.  Due to the structure of the test problem \eqref{Prothero-Robinson-modified},
the method \eqref{eqn:imexRK} uses an explicit approach for the time variable, therefore
\begin{equation}
\label{eqn:d-values-m}
\mathbf{d}^\E = \mathbf{c}^\EE\,, \quad \mathbf{d}^\I = \mathbf{c}^\IE\,.
\end{equation}
The exact solution is expanded in Taylor series about $t_{n}$:
\begin{equation}
\label{eqn:phi-taylor-m}
\begin{array}{rcl}
\phi\left(t_{n} + \mathbf{d}\, h \right)-\one\,\phi(t_{n}) &=& \displaystyle \sum_{k=1}^\infty \frac{h^k \mathbf{d}^k }{k!}\phi^{(k)}(t_{n}) \,,\\
h\, \phi'\left(t_{n} + \mathbf{d}\, h \right) &=& \displaystyle  \sum_{k=1}^\infty \frac{k h^k \mathbf{d}^{k-1} }{k!}\phi^{(k)}(t_{n}) \,,
\end{array}
\end{equation}
where the vector power $\mathbf{d}^{k}$ is taken componentwise.

Consider the global errors
\begin{eqnarray*}
&& e_n = y_n - \phi(t_n)~, \quad E_Y = Y - \phi^\E, \quad E_Z = Z - \phi^\I,  \quad E_W = W - \phi^\I\,.
\end{eqnarray*}
From \eqref{eqn:imex-W-sol} and the quadrature property of the explicit component method we infer that
$w_n = \phi(t_n) + \mathcal{O}(h^p)$. From \eqref{eqn:imex-PRM-W-stage} and \eqref{eqn:phi-taylor-m}
\begin{eqnarray*}
W &=&\one \, \phi(t_n) +  \mathbf{A}^\IE \, \sum_{k=1}^\infty \frac{k h^k (\mathbf{d}^\I)^{k-1} }{k!}\phi^{(k)}(t_{n})   + \mathcal{O}(h^p) 
%\\ 
%E_W &=&  \sum_{k=1}^\infty \frac{h^k \, \left(k\, \mathbf{A}^\IE \, \mathbf{d}^\I\,^{k-1}-\mathbf{d}^\I\,^{k} \right) }{k!}\phi^{(k)}(t_{n})   + \mathcal{O}(h^p) \\
%&=&  \sum_{k=1}^\infty \frac{h^k \, \alpha^W_k }{k!}\phi^{(k)}(t_{n})   + \mathcal{O}(h^p) \\
%\alpha^W_k &=& k\, \mathbf{A}^\IE \, \mathbf{d}^\I\,^{k-1}-\mathbf{d}^\I\,^{k}
\,.
\end{eqnarray*}

Write the stage equation \eqref{eqn:imex-PRM-implicit-stage} 
\[ Z-W  = y_{n}\, \one  + h  \,  \mathbf{A}^\IE \, \, \phi'^\E   + h  \, \mu \, \mathbf{A}^\II \,  \left( Z -   W \right) -W \]
in terms of the exact solution and global errors,
and use the Taylor expansions \eqref{eqn:phi-taylor} to obtain
\begin{eqnarray*}
\left( \mathbf{I} - h\,\mu\, \mathbf{A}^\II \right)\, \left(E_Z-E_W\right)  &=& e_{n} \, \one + \phi(t_n) \, \one    + \mathbf{A}^\IE \,  \left( h  \,  \phi'^\E \right) - W \\
 &=& e_{n} \, \one + \delta_Z + \mathcal{O}(h^p) \,,  \\
\delta_Z &=&   \sum_{k=1}^\infty \left( k\, \mathbf{A}^\IE \left(\mathbf{d}^\E\right)^{k-1} - k\, \mathbf{A}^\IE \left(\mathbf{d}^\I\right)^{k-1}  \right) \frac{h^k}{k!}\phi^{(k)}(t_{n})\,.
\end{eqnarray*}
Similarly, write the solution equation (\ref{eqn:imex-sol}) in terms of the exact solution and global errors:
\begin{eqnarray*}
e_{n+1} &=& e_n  + \phi(t_n) - \phi(t_{n+1})+ \mathbf{b}^\E\,^T \, h\, \phi'^\E  + h   \, \mu \, \mathbf{b}^\I\,^T \, \left( E_Z -  E_W \right) \\
&=& R^\II(h \mu) \,e_n + \sum_{k=1}^\infty \left(k\,  \mathbf{b}^\E\,^T\left( \mathbf{d}^\E \right)^{k-1}-1\right) \frac{h^k \,  }{k!}\phi^{(k)}(t_{n}) \\
&& + h   \, \mu \, \mathbf{b}^\I\,^T \, \left( \mathbf{I} - h\,\mu\, \mathbf{A}^\II \right)^{-1}\;
\sum_{k=1}^\infty \left( k\, \mathbf{A}^\IE \left(\mathbf{d}^\E\right)^{k-1}  
\right. \\ && \qquad \left. 
-k\, \mathbf{A}^\IE \left(\mathbf{d}^\I\right)^{k-1}  \right) \frac{h^k}{k!}\phi^{(k)}(t_{n})  + \mathcal{O}(h^p)
\end{eqnarray*}

The stability function of the implicit component method is
\[
R^\II(h \mu) =  \left( 1 +  h   \, \mu \, \mathbf{b}^\I\,^T \, \left( \mathbf{I} - h\,\mu\, \mathbf{A}^\II \right)^{-1}\, \one \right)\,.
\]
Since the explicit component method (by itself) has at least order $p$, it follows from \eqref{eqn:d-values} and the explicit order conditions that
\begin{equation}
\label{eqn:PRM-explicit-order-cancellation}
k \cdot \left( \mathbf{b}^\E\right)^T\left( \mathbf{d}^\E \right)^{k-1}-1 = 0 \quad \mbox{for}~~ k = 1,\dots,p\,.
\end{equation}
Consequently, the global error recurrence reads
\begin{eqnarray*}
e_{n+1} 
&=& R^\II(\infty) \,e_n + \\
&& - \mathbf{b}^\I\,^T \, \mathbf{A}^\II \,^{-1}\;
\sum_{k=1}^\infty \left( k\, \mathbf{A}^\IE \left(\mathbf{d}^\E\right)^{k-1}  
\right. \\ && \qquad \left. 
-k\, \mathbf{A}^\IE \left(\mathbf{d}^\I\right)^{k-1}  \right) \frac{h^k}{k!}\phi^{(k)}(t_{n})  + \mathcal{O}(h^p)
\end{eqnarray*}

\end{document}